%% file: Manuscript.tex
\documentclass{article}

\include{packages.tex}

\include{commands.tex}

\newcommand{\PD}[2][]{
  \ifthenelse{\equal{#1}{}}
  {\textcolor{red}{\textbf{PD: #2}\marginpar{\textcolor{red}{PD}}}}
  {\textcolor{red}{\textbf{PD: #2}\marginpar{\textcolor{red}{PD: #1}}}}
  }

\title{Nonlinear Joint Spectral Radius}

\author[1,2]{Piero Deidda \thanks{piero.deidda@sns.it}}
\author[2]{Nicola Guglielmi \thanks{nicola.guglielmi@gssi.it}}
\author[3,2]{ Francesco Tudisco \thanks{f.tudisco@ed.ac.uk}}
\affil[1]{\small Scuola Normale Superiore, Piazza dei Cavalieri, 7, 56126 Pisa, Italy.}
\affil[2]{\small Gran Sasso Science Institute, Viale F. Crispi, 7, 67100 L'Aquila, Italy.}
\affil[3]{\small School of Mathematics and Maxwell Institute for Mathematical Sciences, University of Edinburgh,  Peter Guthrie Tait Road, EH9 3FD, Edinburgh,
UK.}
\date{}

\begin{document}

\maketitle

\begin{abstract}

We introduce a nonlinear extension of the joint spectral radius (JSR) 
for switched discrete-time dynamical systems governed by sub-homogeneous 
and order-preserving maps acting on cones. We show that this nonlinear 
JSR characterizes both the asymptotic stability of the system and the 
divergence or convergence rate of trajectories originating from different 
points within the cone.
Our analysis establishes upper and lower bounds on the nonlinear JSR of a sub-homogeneous family via the JSRs of two associated homogeneous families obtained through asymptotic scaling. In the homogeneous case, we develop a dual formulation of the JSR and investigate the equality between the joint spectral radius and the generalized joint spectral radius, extending classical results from linear theory to the nonlinear setting.
We also propose a polytopal-type algorithm to approximate the nonlinear JSR and provide conditions ensuring its finite-time convergence. The proposed 
framework is motivated by applications such as the analysis of deep neural networks, which can be modeled as switched systems with structured nonlinear layers. Our results offer new theoretical tools for studying the stability, 
robustness, and convergence behavior of such models.
\end{abstract}

%\tableofcontents

\section{Introduction}
In this work, we study extensions of the classical notion of joint spectral radius (JSR) as a means to investigate the stability properties of switched discrete-time nonlinear dynamical systems
\begin{equation}
    x_{k+1}=f_{\sigma(k)}(x_k), \qquad  f_{\sigma(k)}\in \family:=\{f_i\}_{i\in I}, \qquad \sigma: \N\rightarrow I,
\end{equation}
where $\family$ is a (possibly infinite) family of nonlinear maps, and $\sigma$ is an unknown switching rule. Our focus lies on systems whose trajectories evolve within a cone $\cone$, and where the maps $f_i$ are subhomogeneous and order-preserving. 

The study of the stability of monotone, eventually subhomogeneous or homogeneous (but not switched) systems has been extensively studied, especially in the context of Perron-Frobenius theory, \cite{lemmens2012nonlinear,akian2011stability, hirsch2006monotone, karlsson2001non, karlsson2024metric, lemmens2018denjoy}.
Similarly, different authors have investigated random dynamical systems that are monotone and subhomogeneous or homogeneous \cite{nussbaum1990some, reich1999convergence, gouezel2020subadditive, karlsson2021linear}. However, only very recently, monotone and homogeneous nonlinear switched systems have been partially investigated in \cite{akian2024competive}.

Differently, the theory of linear switched systems has a long history \cite{sun2006switched}. In particular, the study of stability of discrete linear switched systems has classically relied on the joint spectral radius to characterize the maximal rate of growth across all admissible switching sequences \cite{jungers2009joint, brayton1980constructive, brayton1979stability}. In the linear case, asymptotic stability and local stability (i.e., sensitivity to perturbations of initial conditions) are tightly linked. However, this relationship typically breaks down in nonlinear settings: while local stability can be probed through Jacobians and Lyapunov exponents, these do not capture the global dynamical behavior of trajectories.

To bridge this gap, we extend the notion of joint spectral radius to a nonlinear framework where the constituent maps are subhomogeneous and order-preserving. Such maps are naturally non-expansive with respect to Hilbert's and Thompson's metrics on cones \cite{lemmens2012nonlinear}, a property that allows us to connect asymptotic properties of the system to local behavior. Thus, the nonlinear JSR becomes a powerful tool to analyze both stability and expansion rates of switched systems under unknown switching laws.
Our first main contribution is a formal definition of the nonlinear JSR for families of subhomogeneous, order-preserving maps, which captures worst-case asymptotic behavior over all switching sequences. In particular, our definition of joint spectral radius extends to the case of a family of maps, the definition of the spectral radius of a homogeneous and order-preserving map.  
Interestingly, our definition of the joint spectral radius can be considered as a particular case of the competitive spectral radius recently introduced in \cite{akian2024competive} by Akian, Gaubert and Marchesini (see \Cref{reamrk_competitive_spectral_radius}).

We then provide a detailed theoretical study of this object. First, we show that the nonlinear JSR can be bounded from above and below using two families of homogeneous maps derived from $\family$ by asymptotic rescaling. Second, we establish that the JSR for homogeneous and order-preserving maps admits (a) a dual characterization in terms of monotone prenorms, and (b) a reformulation via the generalized JSR of the $\family$. These results extend key linear theory insights into the nonlinear domain \cite{jungers2009joint, rota1960note, BERGER199221, elsner1995generalized}.

We also present an algorithm for approximating the nonlinear JSR in the homogeneous setting. This algorithm generalizes polytopal approaches developed for linear systems \cite{guglielmi2001asymptotic, guglielmi2005complex, guglielmi2013exact, guglielmi2016balanced}, and builds on our dual reformulation and the equality between the JSR and the generalized JSR. Extensions to continuous problems, addressing the computation of Lyapunov functions, have also been considered (see e.g. \cite{guglielmi2017Lyapunov}).
For an extension to complex families and balanced complex polytopes we
refer the reader to \cite{guglielmi2007bcp}.

Our motivation for studying nonlinear JSR is twofold. From a theoretical perspective, subhomogeneous and order-preserving maps exhibit rich spectral properties that support a nonlinear extension of Perron–Frobenius theory \cite{lemmens2012nonlinear}. From an applied perspective, switched systems of this kind arise naturally in machine learning, most prominently in deep neural networks. In the infinitely deep regime, a neural network can be modeled as
\begin{equation}
    x_{k+1}=f_{\theta_k}(x_k), \qquad k=0,\dots,\infty,
\end{equation}
where $f_{\theta_k}$ denotes the $k$-th layer, typically composed of an affine transformation followed by an entrywise nonlinear activation:
\begin{equation}
   f_{\theta_k}(x):=\psi_k (A_kx+b_k), \qquad \theta_k:=\{A_k,b_k\}.
\end{equation}
Here, the switching rule $\theta_k$ arises from the training process and is not known a priori. Furthermore, architectural and regularization constraints often restrict the parameters $\theta_k$ to lie in a structured set $\Theta$, such as nonnegative matrices for monotonicity, Toeplitz matrices in CNNs, or antisymmetric structures for mass conservation \cite{lim70equivariant, savostianova2023robust, haber2017stable, celledoni2023dynamical, guglielmi2024contractivity}.

As a result, neural networks can be naturally modeled within our framework: they alternate maps from an infinite family of nonlinear functions,
\begin{equation}
    \family_k := \{f_{\theta_k}(x)\,|\; \theta_k \in \Theta\},
\end{equation}
with a switching rule implicitly defined by the training process and thus unknown a priori.

In particular,  the choice of positive parameters in most neural network architectures gives rise to subhomogeneous and order-preserving maps on the cone $\cone:=\R_+^N$.
Indeed, any affine map $f(x)=Ax+b$, with both $A$ and $b$ entry-wise nonnegative, is subhomogeneous and order-preserving. Moreover, many activation functions commonly used in machine learning, such as ReLU, sigmoid, and softplus, are subhomogeneous and order-preserving on $\R_+^N$; see, e.g., \cite{sittonisubhomogeneous, piotrowski2024fixed}.
Neural networks constrained to use nonnegative parameters have been studied in several settings, including monotonic neural networks \cite{archer1993application, sill1997monotonic, daniels_monotone2010}, deep equilibrium models \cite{piotrowski2024fixed}, convex and graph neural networks, and learned regularizers \cite{mukherjee2020learned, zhang2025rethinking}.

In the analysis of neural network stability, both global and local behavior of the feature dynamics are of interest. On the one hand, local behavior, such as the behavior of gradients, is related to training pathologies like vanishing or exploding gradients \cite{pascanu2013difficulty}, and to robustness against random or adversarial perturbations \cite{szegedy2014adversarial}. On the other hand, global properties such as boundedness, convergence to fixed points, and asymptotic stability are crucial to understanding degeneracies in representation learning, with connections to well-known phenomena like oversmoothing in graph neural networks \cite{Li_Han_Wu_2018} and the cut-off phenomenon \cite{avelin2022deep}.

Subhomogeneous and order-preserving models provide a tractable and expressive class of neural networks for which such stability properties can be studied in a principled way. These models have been recently employed in the analysis of deep equilibrium networks \cite{sittonisubhomogeneous, piotrowski2024fixed}, in characterizing the cut-off behavior in deep architectures \cite{avelin2022deep}, and in understanding oversmoothing in graph-based models \cite{zhang2025rethinking}.

Finally, we remark that subhomogeneous and order-preserving systems find application in many other fields, like network optimization problems \cite{cavalcante2019connections, CAVALCANTE2016Elementary}, integral-preserving and matrix scaling problems \cite{lemmens2012nonlinear}, chemical reactions studies \cite{angeli2008translation} and game theory \cite{akian2024competive}.

\subsubsection*{Structure of the paper and main contributions}

The manuscript starts with \Cref{SEC:preliminaries}, which provides background material, including a concise summary of relevant nonlinear Perron–Frobenius results (see also \Cref{Appendix_nonlinear_Perron Frobenius}). The rest of the manuscript is structured as follows:
\begin{itemize}
    \item \Cref{SEC:JSR and the stability} introduces the nonlinear joint spectral radius and shows how it characterizes asymptotic and local stability for switched systems. In particular, we show how order-preservation allows the JSR to bound trajectory divergence via the Thompson metric (\Cref{Theorem_Lipshitz_constant_bound_JSR}).
    
    \item \Cref{SEC:JSR_from_subhomogeneus_to_homogeneous} shows that any subhomogeneous map can be controlloed from above and below by two homogeneous order-preserving maps $f_0$ and $f_\infty$ derived by scaling limits. We then bound the nonlinear JSR using the JSRs of these associated homogeneous families (\Cref{Cor_controlling_JSR_of_subhomogeneous_by_homogeneous}).
    
    \item \Cref{Sec:jsr_homogeneous_case} develops the theory for the homogeneous case. We present a first duality formula for the joint spectral radius using monotone prenorms (\Cref{Thm_duality_for_the_cone_spectral_radius}). Second, we introduce the generalized JSR obtained by looking at the spectral radii of the maps in semigroup generated by the family and discuss situations where equality between the JSR and the generalized JSR holds. We also analyze continuity properties of the JSR and generalized JSR.
    
    \item \Cref{sec:polytopal_algorithm} introduces a numerical algorithm for computing the JSR of homogeneous, order-preserving families. We give sufficient conditions for finite convergence and relate the algorithm to extremal prenorms and the semigroup formulation.
\end{itemize}

Each section is complemented with examples, remarks, and numerical results to illustrate how the nonlinear JSR captures both stability and expressivity properties of switched systems, including applications to neural networks. In the appendix \Cref{subsec:joint_spectral_radii_of_subhomogeneous_maps} we also include a Perron-type result on the existence of continuous curves of eigenpairs for a subhomogeneous map.

\section{Preliminaries}\label{SEC:preliminaries}
In this section, we formalize our setting and recall some basic results and definitions from the literature, in particular from the Nonlinear Perron-Frobenius Theory. We refer to \cite{lemmens2012nonlinear} for a more detailed discussion of these topics. For the reader convenience we also include in \Cref{Appendix_nonlinear_Perron Frobenius} a list of the technical results from \cite{lemmens2012nonlinear} that are used in this paper.

The domain of our dynamical systems will be a cone $\cone\subset V$ in a finite dimensional vector space $V$, to this end we start recalling the definition of a solid closed cone, and of the ordering that it induces.
\begin{definition}[Cone, dual cone and induced ordering]
    $\cone\subset V$ is a solid closed convex cone if it is a closed convex subset such that $\cone \cap -\cone=\empty$, $c x\in \cone$ for any $x\in \cone$ and $c\geq 0$, and $\mathrm{Int}(\cone)\neq 0$ (solid). A subset $C'$ of the cone is a face if whenever $ tx+(1-t)y$ for any $t\in[0,1]$ then both $x,y\in C'$. 
    The dual cone $\cone^*$ of $\cone$ is the set of linear forms nonnegative on $\cone$, i.e.
    $$\cone^*=\{\psi \in V^* \text{ such that } \psi(x)\geq 0 \quad \forall x\in \cone\}.$$
    Finally we say that a point $x\in \cone$ dominates a point $y\in \cone$, $x\cgeq y$, if $x-y\in \cone$. Note that $x\cgeq y$ if and only if $\psi(x)\geq \psi(y)$ for all $\psi\in \cone^*$. 
    Then, based on this notion of ordering, for any $x$ and $y$ in $\cone$ it is possible to introduce:
    \begin{equation}
M(x/y):=\inf \{\beta\:|\; x\leq \beta y\}\,,\qquad m(x/y):=\sup\{\alpha\:|\;\alpha y\leq x\} 
\end{equation}
We say that $x\sim y$ if $m(x/y)>0$ and $M(x/y)<\infty$. It is not difficult to see that $\sim$ is an equivalence relation whose classes of equivalence are the relative interiors of the faces of $\cone$. Moreover we write $x\cgg y$ if $x-y\in \interior(\cone)$.
\end{definition}
In particular if we let $\|\cdot\|$ be a norm on $V$, we know that any closed $\cone$ is normal with respect to $\|\cdot\|$ (see Lemma 1.2.5 \cite{lemmens2012nonlinear}), i.e. there exists some $\delta>0$ such that if $x\cleq y$ then $\|x\|\leq \delta\|y\|$. If $\delta$ is equal to $1$ we say that $\|\cdot\|$ is monotone on $\cone$. An example of monotone norms is given by the action of the linear forms in the interior of the dual cone $\cone^*$, i.e. if we have that for any $x\in \cone$ it holds $\|x\|=\psi(x)$ with $\psi\in \cone^*$, then $\|\cdot\|$ is monotone.

Next we introduce some particular classes of continuous maps on the cone, i.e. homogeneous, subhomogeneous and order preserving maps.

\begin{definition}[Subhomogenous, homogeneous, and order preserving maps]
Let $f\in C(\cone,\cone)$ e a continuous map. We say that $f$ is subhomogeneous if 
    $$f(\lambda x)\cleq \lambda f(x) \qquad \forall x\in \cone,\;\forall \lambda\geq 1.$$
    The last is equivalent to ask $f(\lambda x)\cgeq \lambda f(x)$ for all $x\in \cone$ and $\lambda\leq 1$. Moreover $f$ is strictly subhomogeneous if $f(\lambda x) \cll \lambda f(x)$ for any $x\in \cone\setminus\{0\}$ and $\lambda>1$. Finally we say that $f$ is homogenous if the inequality is an equality for any $\lambda\geq0$.

    In addition $f$ is order preserving if 
   $$ \text{Given } x,y\in \cone \text{ s.t. } x\cgeq y\;, \text{ then } f(x)\cgeq f(y).$$
    
\end{definition}

As we anticipated, the class of continuous subhomogeneous and order-preserving maps are particularly interesting in the context of the nonlinear Perron-Frobenius theory, admitting the study of some spectral properties also in the nonlinear case.
Indeed, it is well known that any subhomogeneous and order-preserving map is nonexpansive with respect to the Thompson metric, which is defined on a cone $\cone$ as follows
\begin{definition}[Thompson metric]
Let $x,y \in \cone$ with $x\sim y$, then 
\begin{equation}
\distT(x,y)=\begin{cases} \log\big(\max\{M(x/y),\;M(y/x)\}\big)\qquad &\text{if } x\sim y\\
\infty \qquad &\text{otherwise}
\end{cases} 
\end{equation}
defines a distance on the relative interior of any face of the cone $\cone$.
\end{definition}
In particular, it is possible to prove that the relative interior of any face of the cone is a complete metric space with respect to the Thompson metric (see Proposition 2.5.2 \cite{lemmens2012nonlinear}).
In addition, as we recalled above, any continuous subhomogeneous and order preserving map $f$ is naturally Thompson non expansive (see Lemma 2.1.7 \cite{lemmens2012nonlinear}) i.e.
\begin{equation}
    \text{if }\quad  x\cgeq y \quad \text{ then }\quad  f(x)\cgeq f(y) \, .
\end{equation}
Furthermore, if $f$ is strictly subhomogeneous, then $f$ is contractive with respect to $\distT$.

Consider now a family $\family:=\{f_i\in C(\cone,\cone)\}_{i\in I}$ such that $f_i$ is subhomogneous and Thompson non-expansive for any $i\in I$. We aim at studying the stability of a discrete switched dynamical system alternating maps from $\family$, 
\begin{equation}\label{sw:dyn.syst}
    x_{k+1}=f_{\sigma(k)}(x_k) \qquad  f_{\sigma(k)}\in \family:=\{f_i\}_{i\in I}, \qquad \sigma: \N\rightarrow I.
\end{equation}
In particular, in \eqref{sw:dyn.syst}, we will assume the switching rule $\sigma$ is unknown, i.e. we do not know the sequence of maps from $\family$ that determines the evolution rule of the system.
Since we do not know $\sigma$, our strategy to study the stability of the system will be to take into account the worst case possible when varying all the possible evolution rules of the system. In particular, the following section is devoted to extending the notion of Joint Spectral Radius (JSR) from the linear case \cite{jungers2009joint} to the nonlinear one and proving that the nonlinear JSR of $\family$ establishes the asymptotic stability of the system and yields information about the local stability, i.e. the local rate of expansion or contraction of the system. 
%

%%
%%
%%
%%
%%
%%
%%%%%%%
%%      JSR: Subhomogeneous case 
%%%%%%%
%%
%%
%%
%%
%%
%%

\section{Nonlinear joint spectral radius}
\label{SEC:JSR and the stability}

In this section, we introduce and discuss the notion of Joint Spectral Radius (JSR) of a family of continuous subhomogeneous maps.
Let $\family:=\{f_i\}_{i\in I}$ be a family of continuous and subhomogeneous maps on a closed cone $\cone$. For any $k\in\N$ consider the set of compositions of maps from $\family$ of length $k$: 
\begin{equation}
\Sigma_k(\family):=\{f_{i_1}\comp \dots \comp f_{i_n}\:|\; i_j\in I\}\qquad \Sigma_0=\{\mathrm{Id}\}\,.
\end{equation}
Then, introduce the semigroup generated by $\family$ as the union of all the finite compositions:
\begin{equation}
\Sigma(\family)=\bigcup_{k\geq 0}\Sigma_k(\family)\,.
\end{equation}
Analogously to the linear case \cite{jungers2009joint}, we can define the cone joint spectral radius of $\family$ as
\begin{equation}\label{Cone_joint_spectral_radius}
\cradius(\family)=\limsup_{k}\sup_{f\in\Sigma_k(\family)}\big(\|f\|_{\cone}\big)^{\frac{1}{k}}
\end{equation}
where, for a subhomogeneous map $f$, we set
\begin{equation}\label{ind_norm_subh}
    \|f\|=\sup_{\{x:\|x\|\leq 1\}} \|f(x)\|.
\end{equation}
It is not difficult to see that, by the equivalence of the norms, the value of $\cradius(\family)$ does not depend on the choice of the norm. Similarly, it does not depend on the ball over which we define the norm, i.e. if we alternatively define $\|f\|_{r}=\sup_{x|\|x\|\leq r}\|f(x)\|$, with some $c>0$ but different from $1$, then $\cradius(\family)$ does not change. To observe this, note that if $0<r_1<r_2$, then trivially $\|f\|_{r_1}\leq \|f\|_{r_2}$. In addition, if $\|x\|\leq r_2$ then by subhomogeneity we have that $f(x)=f (r_1r_2x/r_1r_2)\cleq 
 r_2/r_1f(r_1x/r_2)$  and thus by the normality of the cone there exists some $\delta>0$ such that $\|f\|_{r_1}\leq \|f\|_{r_2}\leq \delta r_2/r_1 \|f\|_{r_1}$. 
\begin{remark}\label{rmk:norm_is_a_norm}
    Note that, if $\|\cdot\|$ is a norm on a finite dimensional vector space $V$ and $\cone$ is a cone in $V$, then $\text{SubH}(\cone):=\{f\in C(\cone,\cone)\;\text{ s.t. } f \text{ subhomogeneous } \}$, i.e. the set of subhomogeneous maps on $\cone$, is a cone in $C(V,V)$. Additionally, the induced norm \eqref{ind_norm_subh} satisfies all the properties of a norm when restricted to $\text{SubH}(\cone)$.

    In fact, it is easy to prove that homogeneity and the triangle inequality with respect to positive weights hold. In order to show that given $f\in \text{SubH}(\cone)$, then $\|f\|=0$ if and only if $f=0$, notice that by definition we know that $f(x)=0$ for any $x\in \cone$ with $\|x\|\leq 1$, on the other hand if $x\in \cone$ but $\|x\|\geq 1$, then $f(x)=f(\|x\| x/\|x\|)\cleq \|x\| f(x/\|x\|)=0$.
\end{remark}
As a consequence of \Cref{rmk:norm_is_a_norm}, we can prove that having a joint spectral radius smaller than one corresponds to having an asymptotically stable family of maps. First, we recall the definition of asymptotic stability:

\begin{definition}
    A family $\family$ is asymptotically stable if for any $U,V$ bounded neighborhoods of zero there exists $\alpha<1$ and $K>0$ such that 
    $$f(x)\in \alpha^k V \qquad \forall\,  x\in U \text{ and } f\in \Sigma_{k\geq K}(\family).$$
\end{definition}
First of all, observe the following 
\begin{proposition}
    Let $\family$ be a bounded family of subhomogeneous maps on a cone $\cone$. Then $\family$ is asymptotically stable if and only if there exist $U,V$ bounded neighborhoods of zero and $\alpha<1$ such that for any $k>0$, $f(x)\in \alpha^k V$ for all $x\in U$ and $f\in \Sigma_k(\family)$.
\end{proposition}
\begin{proof}
    Assume $\family$ asymptotically stable, then for any $U'$ and $V'$ neighborhoods of zero there exists $K$ and $\alpha'$ such that 
    $f(x)\in \alpha'^k V'$ for all $x\in U'$ and $f\in \Sigma_{k\geq K}(\family)$. Then, taking $\alpha=\alpha'$,
    \begin{equation}
        U=U' \quad \text{and} \quad V=V'\cup\Big(\bigcup_{k=0}^K\bigcup_{f\in \Sigma_k(\family)}\frac{f(U)}{\alpha^k} \Big),
    \end{equation}
    we easily prove that $\family$ satisfies the thesis.
    
    On the other hand, assume there exist $U, V$ and $\alpha<1$ as in the statement and let $U', V'$ be two generic bounded neighborhoods of zero. Assume without loss of generality that if $x\in V'$, $\{y:y\cleq x\}\subseteq V'$ (otherwise in place of $V'$ we can consider a subset of $V'$ with this property without afflicting the result).
    Then there exist $0<\lambda\leq 1$ and $\Lambda>1$ such that $\lambda U'\subset U$ and $V\subseteq \Lambda V'$. 
    In particular, given any $f\in \Sigma_k(\family)$, for any $x\in U'$, $f(\lambda x)\in \alpha^k V \subset \Lambda \alpha^k V'$ and, since any $f\in \Sigma(\family)$ is subhomogneous, $\lambda f(x) \big(\cleq f(\lambda x)\big)$ is also contained in $\Lambda \alpha^k V'$. Thus we have
    \begin{equation}
    f(x)\in \frac{\Lambda}{\lambda}\alpha^k V' \qquad \forall x\in U', \;\forall f\in \Sigma_k(\family).       
    \end{equation}
    Now take $K$ as the smallest integer such that $(\Lambda/\lambda)^{1/K}\alpha <1$ and $\alpha':=(\Lambda/\lambda)^{1/K}\alpha$. Since $\Lambda/\lambda\geq 1$, for any $k\geq K$ we have 
    \begin{equation}
     f(x)\in (\Lambda/\lambda) \alpha^k V'\subset (\Lambda/\lambda)^{\frac{k}{K}}\alpha^k V'=(\alpha')^k V' \qquad \forall x\in U',\; f\in \Sigma_k(\family),
    \end{equation}
    i.e., the thesis.
\end{proof}

Next, we show that the value of the joint spectral radius provides information about the asymptotic stability of the system.

\begin{theorem}\label{Theorem_asymptotic_stability_via_JSR}
    Let $\family=\{f_i\}_{i\in I}$ be a bounded family of subhomogeneous maps on a cone $\cone$, then $\family$ is asymptotically stable if and only if $\cradius(\family)<1$.
\end{theorem}
\begin{proof}
Assume that $\cradius(\family)<1$ and consider $\epsilon< 1-\cradius(\family)$, then there exists $M$ sufficiently large such that $\|f\|\leq \big(\cradius(\family)+\epsilon\big)^k $ for any $f\in \Sigma_{k}$ with $k\geq M$. In particular, we can consider 
\begin{equation}
    U=\{x\in\cone | \|x\|\leq 1\} \quad  \text{and} \quad V= \bigcup_{0\leq k\leq M}\bigcup_{f\in \Sigma_k(\family)}\frac{f(U)}{\big(\cradius(\family)+\epsilon\big)^k}
\end{equation}
where $\Sigma_0=\{\mathrm{Id}\}$. Then, since $\family$ is bounded $U$ and $V$ are bounded neighborhoods of zero we can prove that $\family$ is asymptotically stable considering the neighborhoods of zero $U$ and $V$ and the scaling factor $\cradius(\family)+\epsilon$.

On the other hand, the asymptotic stability implies $\cradius(\family)<1$. Indeed, let $r_1,r_2\in \R^+$ be such that $B_{r_1}\subset U$ and $V\subset B_{r_2}$, where $B_r=\{x\in \cone\,|\;\|x\|\leq r\}$. Then, assuming without loss of generality $r_1\leq 1$: 
\begin{equation}
\sup_{\|x\|\leq r_1}\|f(x)\|\leq \alpha^k r_2\qquad \forall f\in \Sigma_k(\family).     
\end{equation}
In addition, since any $f\in \Sigma(\family)$ is subhomogeneous and $r_1\leq 1$, for any $x$ with $\|x\|\leq 1$
we have: 
\begin{equation}
    \|f(x)\|\leq \frac{\|f(r_1 x)\|}{r_1}\leq \alpha^k \frac{r_2}{r_1} \qquad \text{i.e.}\quad  \|f\|_{\cone}\leq \frac{r_2}{r_1}\alpha^k \quad \forall f\in\Sigma_k(\family). 
\end{equation}
Finally the last inequality yields the desired upper bound for the joint spectral radius:
\begin{equation}
\cradius(\family)\leq \alpha < 1\,.
\end{equation}
\end{proof}
In particular, as a direct corollary of \Cref{Theorem_asymptotic_stability_via_JSR},  we have 
\begin{corollary}\label{Corollary_decay_norms_semigroup}
    Let $\family=\{f_i\}_{i\in I}$ be a bounded family of subhomogeneous maps on a cone $\cone$. Then for any $\epsilon>0$ there exists some constant $C_{\epsilon}\geq 1$ such that 
    $$\|f\|\leq C_{\epsilon}\big(\cradius(\family)+\epsilon\big)^k \qquad \forall f\in \Sigma_k(\family),\; \forall k\in \N.$$
\end{corollary}

Next, we deal with the problem of local stability. We assume the maps $f_i\in \family$ to be non-expansive with respect to the Thompson metrics. It is clear that, being non-expensive, the maps are somehow intrinsically stable. However, it is also easy to see that taking two generic points, $x$ and $y$, in the interior of the cone, and any positive scalar factor $c>0$, then $\distT(x,y)=\distT(cx,cy)$. Thus, the most trivial case of nonstable maps (in terms of Lyapunov exponents)  $f(x)=cx$ with $c>1$, is non-expansive with respect to the Thompson metric. This motivates the study of how the nonexpansivity with respect to the Thompson metric is translated in terms of the distance induced by a generic norm of $\R^n$. The problem is deeply discussed in \cite{lemmens2012nonlinear}, where it is proved that the Thompson metric induces the canonical topology on the relative interior of any face of the cone; we refer in particular to Lemmas 2.5.1 and 2.5.5 in \cite{lemmens2012nonlinear}. These results are also reported in \Cref{Appendix_nonlinear_Perron Frobenius} for the reader's convenience. 

In particular we observe that in point $2$ of \Cref{Lemma_thompson_equivalent_to_norm}, as long as $\|x-y\|\leq r$, we have that
\begin{equation}
    \frac{r+\|x-y\|}{r}\leq \frac{r}{r-\|x-y\|}, 
\end{equation}
thus we always have the upper bound $\distT(x,y)\leq \log\big(r/(r-\|x-y\|)\big)$
As a direct consequence, we obtain that the joint spectral radius of the family $\family$ allows us to control the maximal growth rate of the distance between all the possible trajectories, admitted by the dynamical system, that start from two close points. Precisely, we have
\begin{theorem}\label{Theorem_Lipshitz_constant_bound_JSR}
Let $\family$ be a family of continuous and subhomogeneous maps on a closed cone $\cone\subset (V,\|\cdot\|)$ that is non-expansive with respect to the Thompson metric. Then, for any $\epsilon>0$, $x\in \interior(\cone)$ with $B_r(x)\in \interior(\cone)$ and $y\in B_r(x)$, there exists a scalar factor $C(x,y,\epsilon)$ such that for any $f\in \Sigma_k(\family)$ and any $k\geq 0$
$$\|f(x)-f(y)\|\leq C(x,y,\epsilon) (\cradius(\family)+\epsilon)^k \|x-y\|,$$
with $C(x,y,\epsilon)=\delta(1+2\delta) C_\epsilon \max\{1,\|x\|+r\}/(r-\|x-y\|)$.
\end{theorem}
\begin{proof} 
    From Lemma \ref{Lemma_thompson_equivalent_to_norm} we know that for any $y\in B_r(x)$, 
\begin{equation}\label{equicontinuity_eq_1}
    \distT(x,y)\leq \log\left(\frac{r}{r-\|x-y\|}\right).
\end{equation}
In addition, if $f\in \Sigma_k(\family)$ we have $\distT(f(x),f(y))\leq \distT(x,y)$, and moreover if $\|y\|\leq 1$ we have $\|f(y)\|\leq \|f\|$, otherwise 
\begin{equation}
    f(y)= f\!\left(\|y\|\frac{y}{\|y\|}\right) \cleq \|y\|\, f\!\left(\frac{y}{\|y\|}\right)\cleq  (\|x\|+r)\, f\!\left(\frac{y}{\|y\|}\right) 
\end{equation}
where we are assuming $y\in B_r(x)$. Thus, by the normality of the cone, the definition of joint spectral radius and \Cref{Corollary_decay_norms_semigroup}
\begin{equation}
\|f(y)\|\leq \delta\max\{1,\|x\|+r\}\|f\|\leq \delta C_\epsilon \max\{1,\|x\|+r\}\cradius^k(\family).
\end{equation}
Now, defining $C_{\epsilon,x}:=\delta C_\epsilon \max\{1,\|x\|+r\}$, from \Cref{Lemma_thompson_equivalent_to_norm}, we have that 
\begin{equation}\label{eq_equicontinuity_1}
\begin{aligned}
    \|f(x)-f(y)\|
    &\leq  C_{\epsilon,x}(1+2\delta)(e^{\distT(f(x),f(y))}-1) \\
    &\leq C_{\epsilon,x}(1+2\delta) (e^{\distT(x,y)}-1)  \qquad \forall \;y\in B_r(x),\; i\in I.
\end{aligned}
\end{equation}
Next, using \eqref{equicontinuity_eq_1}, we have
\begin{equation}
    \|f(x)-f(y)\|\leq  C_{\epsilon,x}(1+2\delta)\frac{\|x-y\|}{r-\|x-y\|}, 
\end{equation}
In particular, after defining $C(x,y,\epsilon)=\delta(1+2\delta) C_\epsilon \max\{1,\|x\|+r\}/(r-\|x-y\|)$, we have the thesis.
\end{proof}

\Cref{Theorem_Lipshitz_constant_bound_JSR} shows that, when the family $\family$ consists of subhomogeneous maps that are non-expansive with respect to the Thompson metric, the joint spectral radius $\cradius(\family)$ provides an upper bound on the Lipschitz constant of all possible trajectories of the system.

It is worth pointing out that when studying the local expansion rate of trajectories, an arguably more natural object to consider is the exponential growth rate of the norms of the Jacobians of the maps in the semigroup generated by $\family$ (assuming differentiability; otherwise one may consider appropriate generalized derivatives). This leads to the quantity
\begin{equation}
\cradius(J\family)(x) = \limsup_{k\rightarrow\infty} \sup_{f \in \Sigma_k(\family)} \|Jf(x)\|^{\frac{1}{k}},
\end{equation}
where $Jf(x)$ denotes the Jacobian matrix of the map $f$ at the point $x$.

Studying $\cradius(J\family)$ could potentially yield sharper bounds than those obtained via $\cradius(\family)$, especially for assessing local stability properties. However, both the theoretical and computational treatment of $\cradius(J\family)$ may be significantly more involved. Although one moves from a nonlinear to a linear setting by considering the family of Jacobians $D\family = \{Jf_i(x)\}_{i \in I,\, x \in \cone}$, the application of the chain rule implies that not all switching sequences are admissible in the composition of Jacobians. As a consequence, the classical linear joint spectral radius of $D\family$ (see \cite{jungers2009joint}) only provides an upper bound for $\cradius(J\family)$, and a more refined analysis is required to obtain tight estimates. Furthermore, even if the original family $\family$ is finite, the induced Jacobian family $D\family$ becomes infinite, adding another layer of complexity to its analysis.

In the next section, we investigate more in depth the Joint spectral radius in the case of a family of subhomogeneous and order-preserving maps. Before doing that, we discuss below several direct consequences of the results that we have presented in this section in terms of properties of the dynamical system and its orbits. In particular, we consider some common problems related to the stability of neural networks in machine learning.

\subsection*{Implications for artificial neural networks}

Before proceeding to a deeper analysis of the joint spectral radius for subhomogeneous and order-preserving families, we highlight several direct consequences of the results presented in this section for the dynamics of neural networks. Neural networks used in modern deep learning can be naturally interpreted as discrete-time switched systems of the form
\begin{equation}
    x_{k+1} = f(x_k, \theta_k),
\end{equation}
where each $\theta_k$ denotes the parameters of the $k$-th layer and belongs to a constrained parameter space $\theta_k \in \mathcal{M}$. The sequence of parameter selections $\{\theta_k\}$ defines the switching rule, which is implicitly determined by the training process.

In the supervised learning setting, training involves minimizing a loss function $\mathcal{L}$ over a dataset $\mathcal{X}_0 = \{x_0^i\}_{i=1}^N$ with corresponding targets $\{y^i\}_{i=1}^N$. The goal is to find parameters $\Theta = \{\theta_k\}_{k=1}^L$ that minimize the discrepancy between the predicted outputs $x_L^i(\Theta)$ and the targets $y^i$:
\begin{equation}
 \Theta^* \in \argmin_{\Theta = \{\theta_k\} \in \mathcal{M}} \sum_{i=1}^N \mathcal{L}(x_L^i(\Theta), y^i), \qquad \text{where} \quad x_k^i(\Theta) = f(x_{k-1}^i, \theta_k).
\end{equation}
We now discuss how our results can shed light on several key issues in machine learning.

\paragraph{Vulnerability to adversarial attacks.}
One of the most well-known challenges in deep learning is the vulnerability of neural networks to adversarial perturbations \cite{biggio2013evasion, goodfellow2014explaining}. A small, often imperceptible perturbation $x_0^\epsilon$ of the input $x_0$, with $\|x_0 - x_0^\epsilon\| \leq \epsilon$, can cause a significant change in the output $x_L^\epsilon$, undermining both robustness and generalization.

This phenomenon is closely tied to the Lipschitz continuity of the network \cite{goodfellow_intr_properties, parseval, tsuzuku2018lipschitz, hein2017formal}. In fact, a substantial body of research in the machine learning community focuses on designing architectures and training strategies that enforce or encourage low $L^2$-based Lipschitz constants \cite{tsuzuku2018lipschitz, trockmanorthogonalizing, leino2021globally, savostianova2023robust}, with the aim of obtaining certified robustness guarantees \cite{hein2017formal, singlaimproved, meunier2022dynamical}. 

The nonlinear joint spectral radius can shed new light on this problem. Suppose the network is trained so that all layer maps belong to a family of subhomogeneous and order-preserving functions with respect to a cone. Then, as shown in \Cref{Theorem_Lipshitz_constant_bound_JSR}, the joint spectral radius provides a tight upper bound on the Lipschitz constant of the overall network when measured with respect to Thompson’s metric. This bound directly controls the expansion rate of trajectories and thus the sensitivity of outputs to small input perturbations.

Crucially, this nonlinear joint spectral radius offers a sharper and more geometrically meaningful measure of expansivity than the standard $L^2$-based Lipschitz bounds typically used in the literature. As a result, it opens the door to improved robustness certification and training schemes, especially for architectures naturally constrained to positive parameters. Promoting a low joint spectral radius through architecture or regularization design could yield significant gains in adversarial robustness, while preserving expressivity and trainability.

\paragraph{Exploding gradients.}
A fundamental challenge in training deep neural networks is the problem of exploding gradients \cite{bengio1994learning, pascanu2013difficulty}. During backpropagation, the gradients of the loss function with respect to parameters in early layers can grow exponentially with the network's depth, making optimization unstable or even infeasible. For a parameter $\theta_k$ at layer $k$, the gradient takes the form:
\begin{equation}
    \frac{\partial \mathcal{L}(x_L^i, y^i)}{\partial \theta_k} = \frac{\partial \mathcal{L}(x_L^i, y^i)}{\partial x_L^i} \cdot \frac{\partial x_L^i}{\partial x_k^i} \cdot \frac{\partial x_{k+1}^i}{\partial \theta_k},
\end{equation}
where the term $\partial x_L^i / \partial x_k^i$ captures the long-range dependency between early and late layers, and is primarily responsible for potential gradient explosion as the depth $L$ increases.

This issue is exacerbated in deep architectures and recurrent neural networks, where the repeated composition of nonlinear maps makes it difficult to maintain well-behaved gradient norms \cite{pascanu2013difficulty, hochreiter1998vanishing}. Standard mitigation techniques typically assume layerwise independence and are based on the $L^2$ norm of the layers, and include gradient clipping as well as initialization and architectural constraints \cite{glorot2010understanding, arjovsky2016unitary, zhang2019fixup}.

In the setting considered here, where each layer map belongs to a family of subhomogeneous and order-preserving functions, our theoretical results offer a different mechanism to control gradient growth. Specifically, by \Cref{Theorem_Lipshitz_constant_bound_JSR}, the term $\| \partial x_L^i / \partial x_k^i \|$ can be uniformly bounded in terms of the nonlinear joint spectral radius $\cradius(\family)$ of the system. That is,
\begin{equation}
    \left\|\frac{\partial x_L^i}{\partial x_k^i} \right\| \leq C(x_k^i, \epsilon)\, (\cradius(\family) + \epsilon)^{L - k}, \qquad \forall L \geq k,
\end{equation}
for any $\epsilon > 0$, where $C(x_k^i, \epsilon)$ is a locally defined constant.
This inequality provides a direct mechanism for analyzing and mitigating exploding gradients through architectural or training constraints that limit the joint spectral radius rather than the Euclidean norm.

\paragraph{Fixed points and deep equilibrium models.}
An increasingly influential direction in neural network design centers on models whose outputs are defined as fixed points of certain nonlinear transformations. In particular, deep equilibrium models (DEQs) \cite{bai2019deep} replace the traditional finite-depth layer stacking with an implicit definition of the output as the solution to a fixed-point equation. That is, the network computes a feature embedding $z^*$ such that
\begin{equation}
    z^* = f(z^*, \theta),
\end{equation}
for some parameterized nonlinear map $f$. These models have shown considerable promise in various domains, including imaging, inverse problems, and sequence modeling \cite{bai2019deep, gilton2021deep}, due to their ability to model infinitely deep behavior using constant memory during training and inference.

The convergence and well-posedness of such models hinge on the existence and uniqueness of fixed points. Several recent works have studied sufficient conditions for contractivity, particularly in networks with architectural constraints that guarantee subhomogeneity or monotonicity \cite{sittonisubhomogeneous, piotrowski2024fixed}. 

In this context, our framework offers a natural theoretical foundation for analyzing equilibrium dynamics. Specifically, as shown in \Cref{Theorem_asymptotic_stability_via_JSR}, a necessary condition for global convergence to a fixed point is that the joint spectral radius satisfies $\cradius(\family) < 1$, where $\family$ denotes the family of layer-wise maps in the model. This provides a rigorous spectral condition for the existence of attracting fixed points, even when the architecture is highly nonlinear and non-contractive in the Euclidean metric.

Furthermore, recent advances in implicit differentiation for fixed-point models \cite{fung2021fixed, revay2023lipconvnet} rely heavily on the spectral properties of the forward map. In this regard, the nonlinear joint spectral radius offers a more precise and geometrically adapted measure of the system’s contractive behavior than traditional Euclidean Lipschitz-based conditions. 
%%%
%%%
%%%
%%%
%%%

%%
%%
%%

%%%
%%%
%%%
%%%
%%%

\section{The subhomogeneous setting}
\label{SEC:JSR_from_subhomogeneus_to_homogeneous}

In \Cref{SEC:JSR and the stability}, we discussed the relevance of the joint spectral radius of a family of subhomogeneous and order-preserving maps in various application domains arising in machine learning. As emphasized, accurately estimating or computing the JSR is crucial for both theoretical analysis and practical applications. The remainder of this manuscript is therefore dedicated to a deeper investigation of the JSR, including dual reformulations, continuity properties, and the development of an algorithm for computing the JSR under suitable conditions.

This section focuses on the structural properties of the JSR for families $\family := \{f_i\}_{i \in I}$, where each map $f_i$ is continuous, subhomogeneous, and order-preserving. We show that the JSR of such a family can be bounded above and below by the JSRs of two associated families, $\family_0$ and $\family_{\infty}$, each consisting of continuous, homogeneous, and order-preserving maps. These results directly generalize the framework developed in \cite{cavalcante2019connections, auslender2006asymptotic}. Motivated by this reduction, the next section will explore in detail the homogeneous case, which presents a simpler yet foundational scenario.

To analyze the asymptotic behavior, we consider the Alexandrov compactification $\cone^{\infty} := \cone \cup \{\infty\}$ of the cone $\cone$. For a subhomogeneous map $f$ and scalars $0 < c_1 < c_2$, the subhomogeneity property implies:
\begin{equation}
    f(c_2 x) = f\left(\frac{c_2}{c_1} c_1 x \right) \leq \frac{c_2}{c_1} f(c_1 x).
\end{equation}
Thus, for a fixed $x$, the function $c \mapsto f(cx)/c$ is monotonic in $c$, and the limits as $c \to 0$ and $c \to \infty$ exist in $\cone^\infty$.

Although a general cone is not necessarily a lattice (i.e., it may not admit suprema or infima), it satisfies enough order-theoretic properties to guarantee uniqueness of limits along monotone chains. Specifically, let $(x_n)$ be a decreasing sequence in $\cone$ with $x_n \cleq x_{n-1}$ and suppose it has an accumulation point $x^*$. Then $x^* \cleq x_n$ for all $n$. Moreover, if $x_n \cgeq y$ for all $n$, then $x^* \cgeq y$ as the set $\{x \in \cone \mid x \cgeq y\}$ is closed. Hence, any two accumulation points must coincide. A similar argument applies to increasing sequences, which thus either diverge to $\infty$ or converge to a unique accumulation point. 

As a consequence, if $(f_n)$ is a monotone sequence of maps, i.e., $f_n(x) \cleq f_{n-1}(x)$ or $f_n(x) \cgeq f_{n-1}(x)$ for all $x \in \cone$, then the pointwise limit exists. This allows us to define two canonical asymptotic maps associated with any subhomogeneous function $f$:
\begin{equation}\label{Def_f_0 and f_inf}
    f_0(x) := \lim_{c \to 0} \frac{f(c x)}{c}, \qquad \qquad
    f_{\infty}(x) := \lim_{c \to \infty} \frac{f(c x)}{c}.
\end{equation}
For convenience, we also introduce the notation:
\begin{equation}
    f_c(x) := \frac{f(c x)}{c}.
\end{equation}

Finally, we recall from \cite{lemmens2012nonlinear} that if $\cone$ is a solid, closed cone, then its dual $\cone^*$ is also solid and closed. If $\psi \in \interior(\cone^*)$, then $\psi(x) > 0$ for all $x \in \cone \setminus \{0\}$, and the level set
\begin{equation}
    \Xi_{\psi} := \{x \in \cone \mid \psi(x) = 1\}
\end{equation}
is a compact convex subset of the ambient space $V$.

In what follows, we study key properties of the two limit maps $f_0$ and $f_\infty$, which play a central role in bounding the joint spectral radius of the original family $\family$ and in reducing the analysis to the homogeneous case.

\begin{proposition}\label{Prop_limit_homogeneous_functions}
    Let $f$ be a continuous, subhomogeneous, and order-preserving map on a closed cone $\cone$. Then the two maps $f_{0}$ and $f_{\infty}$, defined in \eqref{Def_f_0 and f_inf}, are homogeneous and order-preserving with respect to the cone $\cone$. Moreover for any $\psi \in \cone^*$, $\psi\big(f_0(x)\big)$ and $\psi\big(f_\infty(x)\big)$ are lower and upper semicontinuous, respectively. 
\end{proposition}
\begin{proof}
    To prove the homogeneity, note that for any $x\in \cone$ and $\lambda>0$ we have
    \begin{equation}
       f_{0/\infty}(\lambda x)=\lim_{c}\frac{f(c\lambda x)}{c}=\lambda \lim_{c}\frac{f(c\lambda x)}{c\lambda}=\lambda f_{0/\infty}(x).
    \end{equation}
Similarly, for the order preserving property, if $x\cleq y$ we have:
    \begin{equation}
       f_{0/\infty}(x) =\lim_{c}\frac{f(c x)}{c}\cleq \lim_{c}\frac{f(c y)}{c}= f_{0/\infty}(y).
    \end{equation}
To prove the semicontinuity let $x_n\in \cone$ be a sequence converging to $x^*$. Observe that, for any $\epsilon>0$ there exists $C$ such that for any $c>C$ and $\psi\in \cone^*$ we have:
    \begin{equation}
        \psi\big(f_{c}(x^*)\big)\leq \psi(f_\infty(x^*))+\epsilon,
    \end{equation}
In addition, since $f_{c}$ is continuous, there exists a neighborhood $U$ of $x^*$ such that $x_n\in U$ for all $n\geq M$ and:
\begin{equation}
   \psi\big(f_{\infty}(x_n)\big)\leq \psi\big(f_{c}(x_n)\big)\leq \psi\big(f_{c}(x^*)\big)+\epsilon \leq \psi(f_\infty(x^*))+2\epsilon
\end{equation}
for all $n\geq M$. An analogue discussion yields the lower semicontinuity of $\psi\big(f_0(x)\big)$ for any $\psi\in\cone^*$.
\end{proof}
Next we observe that, if $f_0$ and $f_{\infty}$ are continuous, the convergence of $f_{c}$ is uniform on compacts subsets of the cone.
\begin{proposition}\label{Prop:uniform_convergence_limit_maps}
    Let $f$ be a continuous, subhomogeneous, and order-preserving map on a closed cone $\cone$. ans assume $f_0$ and/or $f_{\infty}$, defined in \eqref{Def_f_0 and f_inf}, to be continuous on $\cone$. Then $f_c$ converges uniformly to $f_{0/\infty}$ on any compact subset of $\cone$ for $c$ going to $0,\infty$, respectively. 
\end{proposition}

\begin{proof}
  Given a compact subset $\Omega$ of $\cone$, $f_{c}$ converge uniformly to $f_\infty$ on $\Omega$ (for $c$ going to $\infty$) if for any $\epsilon>0$ there exists $C$ such that for any $c>C$:
    \begin{equation}
        \sup_{x\in \Omega}\|f_{\infty}(x)-f_{c}(x)\|\leq \epsilon.
    \end{equation}
    Because of the equivalence of the norms we can take the norm induced by some $\psi\in \interior(\cone^*)$, i.e. some norm $\|\cdot\|_*$ such that $\|x\|_*=\psi(x)$ for all $x\in \cone$. Then since $f_c(x)\cleq f_0(x)$ for any $c>0$, by the linearity of $\psi$ on $\cone$, the uniform convergence can be written as
        \begin{equation}
        \sup_{x\in \Omega} \psi\big(f_0(x)\big)- \psi\big(f_{c}(x)\big)\leq  \epsilon \qquad \forall c>C.
    \end{equation}
    Now recall that $\psi(f_{c}(x))$, varying $c$ from $0$ to $\infty$, is a decreasing sequence of real functions and thus by the Dini's theorem it converges uniformly. An analogue discussion holds for $f_0$.
\end{proof}
Next we observe that whenever $f$ is Lipschitz then $f_0$ and $f_{\infty}$ are also Lipschitz and so continuous.
\begin{proposition}\label{Prop_lipschitz_subhomogeneous_yields_continuous_limits}
    Let $f$ be a continuous, subhomogeneous, and order-preserving map on a closed cone $\cone$ and Assume that $f$ is $L$-Lipshitz, i.e. $\sup_{x\in\cone}\|f(x)-f(y)\|\leq L \|x-y\|$, then $f_{0}$ and $f_{\infty}$ are continuous and $L$-Lipschitz
\end{proposition}
\begin{proof}
Since $f_0$ and $f_{\infty}$ are homogeneous it is sufficient to prove that they are continuous when restricted to $\Omega_1:=\{x\in\cone\::\;\psi(x)=1\}$ with $\psi\in \interior(\cone^*)$. 
So consider $x,y\in \Omega_1$, then 
\begin{equation}
\begin{aligned}
\|f_{0/\infty}(x)-f_{0/\infty}(y)\|=\lim_{c} \bigg\|\frac{f(cx)}{c}-\frac{f(cy)}{c}\bigg\| 
 \leq \lim_c \frac{L}{c} \big\|cx-cy\big\|\leq L \|x-y\|.
\end{aligned}
\end{equation}
\end{proof}
In particular, observe that if $f(0)=0$ and $f$ is differentiable in $0$ then $f_0(x)=Jf(0)x$.
On the other hand, as observed in \cite{cavalcante2019connections}, $f_{\infty}$ is continuous whenever the cone is polyhedral.
\begin{proposition}
    If $\cone$ is a closed polyhedral cone and $f$ is a continuous, subhomogeneous and order-preserving map. Then, $f_{\infty}$ is continuous.
\end{proposition}
\begin{proof}
    Take $x_n\in \cone$ with $\lim_{n}x_n=x^*$. Then, since the cone is polyhedral, for any $\alpha\in(0,1)$ there exists some $N$ such that for any $n>N$, $x_n\cgeq \alpha x^*$. In particular, since $f_{\infty}$ is order preserving and homogeneous, for any $\psi\in \cone^
    *$ we have
    \begin{equation}
        \liminf_{n}\psi\big(f_{\infty}(x_n)\big)\geq\psi\big(f_{\infty}(x^*)\big).
    \end{equation}
    The last joint with the upper continuity proved in \Cref{Prop_limit_homogeneous_functions} yields:
\begin{equation}
    \psi\big(f_{\infty}(x)\big)\leq \liminf \psi\big(f_{\infty}(x_n)\big) \leq \limsup \psi\big(f_{\infty}(x_n)\big) \leq \psi\big(f_{\infty}(x)\big).
\end{equation}
In particular the continuity of $\psi(f_{\infty})$ for any $\psi\in\cone^*$, yields the continuity of $f_{\infty}$ in~$x$.
\end{proof}
The same is not true for $f_0$ as we see in the next example
\begin{example}
    Let $\cone=\R_+^2$ and consider
    \begin{equation}
        f(x_1,x_2)=\Big(x_1+\big(x_1x_2\big)^{1/4}, x_2+\big(x_1x_2\big)^{1/4}\Big).
    \end{equation}
   The map $f$ is easily proved to be order-preserving. Moreover $f(\lambda \bar{x})=\Big(\lambda x_1+\lambda^{1/2}\big(x_1x_2\big)^{1/4}, \lambda x_2+\lambda^{1/2}\big(x_1x_2\big)^{1/4}\Big)\leq \lambda f(\bar{x})$ for any $\lambda>1$, i.e. $f$ is subhomogeneous. Now taking the limit of $f(\lambda\bar{x})/\lambda$ for $\lambda$ that goes to $0$, we get:
   \begin{equation}
   f_0(\bar{x})=\begin{cases}
           (x_1,0) \qquad &\text{if } x_2=0\\
           (0,x_2) \qquad &\text{if } x_1=0\\
           \infty \qquad &\text{otherwise}
       \end{cases}
   \end{equation}
\end{example}
Next, we study the asymptotic mappings of a composition of continuous, subhomogeneous and order-preserving maps
\begin{lemma}\label{Lemma_limits_of_compositions}
    Let $f:=f_n\circ \dots \circ f_1$ be the composition of $n$ continuous, subhomogeneous and order-preserving maps with $f_{i,0}(x)\in\cone$ for any $x\in \cone$ and $i=1,\dots, n$. Then $$f_{0/\infty}=f_{n,0/\infty}\circ \dots \circ f_{1,0/\infty}.$$
\end{lemma}
\begin{proof}
    We prove it for $n=2$, i.e. $f=f_2\circ f_1$ and note that by iterating the result can be easily extended to any $n$. Recall that, if $g$ is subhomogeneous and order-preserving, then by definition $g_\infty(x)\cleq g_c(x)=g(cx)/c\cleq g_0(x)$ for any $c>0$. So, taking $c=1$ we have 
    \begin{equation}
      f_{2,\infty}\big( f_{1,\infty}(x)\big)\cleq   f_2\big(f_1(x)\big)\cleq f_{2,0}\big(f_{1,0}(x)\big).
    \end{equation}
    Finally, the homogeneity of $f_{2,\infty}, f_{1,\infty}$ and $f_{2,0}$, $f_{1,0}$ yields:
    \begin{equation}
    \begin{aligned}
      f_{2,\infty}\big( f_{1,\infty}(x)\big)\cleq \lim_{c\rightarrow \infty}   \frac{f_2\big(f_1( c x)\big)}{c}= f_{\infty}(x)\\
      f_{0}(x)=\lim_{c\rightarrow 0}   \frac{f_2\big(f_1( c x)\big)}{c} \cleq f_{2,0}\big(f_{1,0}(x)\big).
    \end{aligned}
    \end{equation}
    Next let $c_1\leq c \leq c_2$, then since $f_1(cx)/c$ is monotone in $c$
    \begin{equation}
      \frac{ f_2\big(c f_{1,\infty}(x)\big) }{c}\cleq \frac{f_2\big(\frac{c}{c_2}f_1(c_2 x)\big)}{c}\cleq \frac{f_2\big(f_1(cx)\big)}{c}\cleq \frac{f_2\big(\frac{c}{c_1}f_1(c_1 x)\big)}{c} \cleq \frac{f_2\big(c f_{1,0}(x)\big)}{c},
    \end{equation}
    finally taking the limits fot $c$ that goes to $\infty$ and $0$ we get the missing inequalities
    \begin{equation}
      \begin{aligned}
      f_{2,\infty}\big(f_{1,\infty}(x)\big)=\lim_{c\rightarrow \infty}\frac{ f_2\big(c f_{1,\infty}(x)\big) }{c} \cleq f_{\infty}(x)\\       f_0(x)\cleq \lim_{c\rightarrow 0}\frac{ f_2\big(c f_{1,\infty}(x)\big) }{c}=f_{2,0}\big( f_{1,0}(x)\big)
      \end{aligned}
    \end{equation}
    
\end{proof}

It follows from \Cref{Lemma_limits_of_compositions} that the joint spectral radius of a family of subhomogeneous and order-preserving maps can be bounded in terms of the JSRs of two associated families consisting of \textbf{homogeneous} and order-preserving maps.

\begin{theorem}\label{Cor_controlling_JSR_of_subhomogeneous_by_homogeneous}
    Let $\family = \{f_i\}_{i \in I}$ be a family of continuous, subhomogeneous, and order-preserving maps on a cone $\cone$, and assume that the corresponding limit maps $f_{i,0}$ and $f_{i,\infty}$ are continuous and satisfy $f_{i,0}(x) \in \cone$ for all $x \in \cone$ and $i \in I$. Then the following inequality holds:
    \[
        \cradius(\family_{\infty}) \leq \cradius(\family) \leq \cradius(\family_0),
    \]
    where $\family_{\infty} := \{f_{i,\infty}\}_{i \in I}$ and $\family_0 := \{f_{i,0}\}_{i \in I}$.
\end{theorem}

\begin{proof}
Let $\|\cdot\|$ be a monotone norm, i.e., a norm such that $\|x\| \geq \|y\|$ whenever $x \cgeq y$. For example, one can consider the norm induced by any $\psi \in \interior(\cone^*)$. By the definitions of $f_{i,\infty}$ and $f_{i,0}$, and the monotonicity of the norm, we obtain:
\[
    \|f_{i,\infty}(x)\| \leq \|f_i(x)\| \leq \|f_{i,0}(x)\| \qquad \forall x \in \cone.
\]
Now let $f = f_{i_1} \circ \dots \circ f_{i_k} \in \Sigma_k(\family)$ be any $k$-fold composition. Then, using the same reasoning,
\[
    \left\|f_{i_1,\infty} \circ \dots \circ f_{i_k,\infty}(x)\right\|
    = \|f_{\infty}(x)\|
    \leq \|f(x)\|
    \leq \|f_0(x)\|
    = \left\|f_{i_1,0} \circ \dots \circ f_{i_k,0}(x)\right\| \qquad \forall x \in \cone.
\]
By taking the appropriate $\limsup$ and using the definition of the joint spectral radius, which is independent of the specific norm used, we obtain the claimed bounds.
\end{proof}

This result motivates the study, undertaken in the next section, of the joint spectral radius for families of continuous, homogeneous, and order-preserving maps. While this reduction simplifies the analysis in many ways, it is important to note that the limit maps $f_0$ and $f_{\infty}$ may, in some cases, be degenerate or uninformative.

For instance, consider a subhomogeneous and order-preserving map $f$ satisfying $f(0) \cgeq 0$. In this case, it is easy to verify that $f_0(x) = \infty$ for all $x \in \cone$. Conversely, if $f(x) \cleq y$ for all $x \in \cone$ and for some fixed $y$, then $f_{\infty}(x) = 0$ identically on $\cone$. These pathologies illustrate that while bounding the JSR via homogeneous approximations is a powerful technique, the interpretation of the limit maps must be handled with care.

We explore these issues further in the next example, which provides a more concrete illustration of the results presented in this section in the context of an artificial neural network.

\begin{example}[Artificial Neural Network]\label{ex:ANN_subhomogeneous}
Many common neural network architectures, including multilayer perceptrons (MLPs),  residual and recurrent networks (ResNets, RNNs), can be expressed in the form
\begin{equation}
    x_{k+1} = f_k(x_k) = A_k x_k + \phi(B_k x_k) + b_k,
\end{equation}
where $x_k \in \R_+^N$ represents the feature vector at the $k$-th layer, $A_k$ and $B_k$ are learnable weight matrices, $b_k$ is a bias vector, and $\phi$ is a nonlinear activation function acting entrywise (i.e., $\phi(x) = (\phi(x_1), \dots, \phi(x_n))^\top$). Popular choices for $\phi$ include \texttt{ReLU}, \texttt{leakyReLU}, \texttt{sigmoid}, \texttt{tanh}, and \texttt{softplus}, all of which are Lipschitz continuous and subhomogeneous.

Assume now that $A_k, B_k$ and $b_k$ have nonnegative entries, and that the activation function $\phi$ is monotone increasing, concave, and satisfies $\phi(x) \geq 0$ for all $x \geq 0$. Under these conditions, the map $f_k$ is both subhomogeneous and order-preserving on the cone $\R_+^N$.

To see subhomogeneity, fix any $x \in \R_+^N$ and $\lambda \in (0,1)$. By concavity and monotonicity of $\phi$, and the nonnegativity of $b_k$ and $\phi(0)$, we obtain:
\begin{equation}
\begin{aligned}
    f_k(\lambda x) &= \lambda A_k x + \phi(\lambda B_k x) + b_k \\
    &\geq \lambda A_k x + \lambda \phi(B_k x) + (1 - \lambda) \phi(0) + \lambda \frac{b_k}{\lambda} \\
    &\geq \lambda A_k x + \lambda \phi(B_k x) + \lambda b_k = \lambda f_k(x),
\end{aligned}
\end{equation}
using the concavity of $\phi$ and the fact that $\phi(0) \geq 0$. Order preservation follows immediately from the monotonicity of $\phi$ and the nonnegativity of all components.

\paragraph{Homogeneous limit maps.}
Under these assumptions, we can compute the asymptotic maps $f_{k,0}, f_{k,\infty} : \R_+^N \to \R_+^N \cup \{\infty\}$, as introduced in \Cref{Def_f_0 and f_inf}. If the activation $\phi$ is Lipschitz, which is true for most standard activations, then both $f_{k,0}$ and $f_{k,\infty}$ are continuous by \Cref{Prop_lipschitz_subhomogeneous_yields_continuous_limits}.

First, consider the behavior near zero. If $\phi(0) \neq 0$ or $b_k \neq 0$, then
\[
f_{k,0}(x) = \lim_{c \to 0} \frac{f_k(cx)}{c} = \infty \qquad \text{for all } x \in \R_+^N,
\]
since the constant terms dominate the scaled input.

If instead $\phi(0) = 0$ and $b_k = 0$, and assuming $\phi$ is differentiable on $\R_+$, we can apply a first-order approximation. Then:
\begin{equation}
    f_{k,0}(x) = Jf_k(0)x = A_k x + \mathrm{diag}(\phi'(0)) B_k x,
\end{equation}
where $\phi'(0)$ is the entrywise derivative of the activation at zero.

Similarly, the limit at infinity is:
\begin{equation}
    f_{k,\infty}(x) = \lim_{c \to \infty} \frac{f_k(c x)}{c} = A_k x + \mathrm{diag}(\phi'(\infty)) B_k x,
\end{equation}
where $\phi'(\infty) := \lim_{x \to \infty} \phi'(x)$ exists due to the concavity of $\phi$. This expression can also be derived via an application of L'Hôpital's rule as discussed in \cite{piotrowski2024fixed}.

\end{example}

\section{The homogeneous setting}
\label{Sec:jsr_homogeneous_case}

Based on the results from the previous section, we now focus on the case in which the family $\family$ consists of continuous, homogeneous, and order-preserving maps. This setting is substantially closer to the classical linear case and enables a deeper theoretical analysis of the joint spectral radius (JSR), including dual formulations that generalize known linear results \cite{jungers2009joint}.

A key advantage of the homogeneous setting lies in the spectral properties of the individual maps. Homogeneity, defined by the identity $f(cx) = c f(x)$ for all $c \geq 0$ and $x \in \cone$, allows us to meaningfully define eigenpairs. Specifically, if $x$ is an eigenvector of $f$ associated with eigenvalue $\lambda$, then so is every positive scalar multiple $cx$ of $x$.

In particular, a foundational result from nonlinear Perron–Frobenius theory asserts that the spectral radius $\cradius(f)$ of a continuous, homogeneous, and order-preserving map $f$ is well-defined via the nonlinear Gelfand formula. Moreover, this quantity coincides with the maximal eigenvalue associated with eigenvectors lying in the cone. As we shall see, this property is central to the development of both analytical and algorithmic approaches to the JSR in the nonlinear setting.
\begin{definition}[Cone Spectral Radius]\label{Def_bonsall_spectral_radius}
Let $\cone$ be a closed cone and $f$ a continuous, homogeneous and order-preserving map. Then the spectral radius $\cradius(f)$ is defined as:
   \begin{equation}
   \cradius(f)=\lim_{m\rightarrow\infty}\|f^m\|_{\cone}^{\frac{1}{m}}=\inf_{m\geq 1}\|f^m\|_{\cone}^{\frac{1}{m}}.
   \end{equation} 
\end{definition}
In particular there exists some eigenvector $x_f\in\cone\setminus\{0\}$ correspoding to the spectral radius, i.e. $f(x_f)=\cradius(f) x_f$.
This fact is a consequence of the nonexpansivity of any homogeneous and order-preserving map with respect to the Hilbert distance $\dist$ induced by the cone. More in detail any continuous, homogeneous and order-preserving map $f$ satisfies
\begin{equation}
    \dist(f(x),f(y))\leq \dist(x,y)\,.
\end{equation}
Moreover, whenever $f$ is strongly order preserving, i.e.:
\begin{equation}
    x\cl y \;\big(x\leq y \,,\; \neq y\big) \quad \Rightarrow \quad f(x)\cll f(v)\; \big(f(x)-f(y)\in \interior(\cone)\big)\,,
\end{equation}
then $f$ is a also a contraction with respect to $\dist$\,.
Below we recall the defintion of the Hilbert distance and some basic facts about it.
\begin{definition}[Hilbert Distance]
\begin{equation}
\dist(x,y)=\begin{cases}
    \log\bigg(\frac{M(x/y)}{m(x/y)}\bigg)\qquad &x,y\in \cone \text{ with } x\sim y\\
\infty \qquad &\text{otherwise}
\end{cases}.    
\end{equation}
\end{definition}
We recall that $\dist$ is only a “projective" distance, meaning that $\dist(c_1x,c_2 y)=d_{H}(x,y)$ for all $c_1,c_2>0$, and $x,y\in\cone$.
However, the relative interior of any face intersected with any slice of the cone, i.e. $P\cap \Xi_\phi$, where $\Xi_\phi=\{x\in\cone \;\text{s.t.}\; \phi(x)=1\}$ with $\phi\in \interior(\cone^*)$ and $P$ one part of the cone, becomes a complete metric spaces when endowed with the Hilbert distance. We refer to \cite{lemmens2012nonlinear} and the \Cref{Appendix_nonlinear_Perron Frobenius} for more details.

One key advantage of homogeneous mappings over subhomogeneous ones is given by the fact that, given a norm $\|\cdot\|$ and a homogeneous map $f$, the corresponding induced operator norm can be equivalently reformulated as $\|f\|=\sup_{\|x\|\leq 1}\|f(x)\|=\sup_{x\in\cone\setminus\{0\}}\|f(x)\|/\|x\|$. As a direct consequence, the induced operator norm is submultiplicative, i.e., if $f=f_1\circ f_2$ with $f_1\in \Sigma_k(\family)$ and $f_2\in \Sigma_h(\family)$, then $\|f\|\leq \|f_1\|\|f_2\|$. Hence, when we consider a family $\family$ of homogeneous maps, we can write
\begin{equation}
    \sup_{f\in \Sigma_{k+h}(\family)}\log \|f\|\leq \sup_{f\in \Sigma_{h}(\family)}\log \|f\|+\sup_{f\in \Sigma_{k}(\family)}\log \|f\|\,,    
\end{equation}
which implies, because of Fekete's lemma, that the sequence $\{(1/k)\sup_{f\in \Sigma_{k}}\log \|f\|\}_{k\in \N}$ admits a limit equal to the infimum of the sequence, i.e.
\begin{equation}\label{Conespectralradius_characterization}
\cradius(\family)=\inf_k\sup_{f\in\Sigma_k(\family)}\big(\|f\|\big)^{\frac{1}{k}}. 
\end{equation}

This will allow us to show in the following that the joint spectral radius of a family of homogeneous and order-preserving maps measures the maximal rate of growth of any internal point of the cone.
Before doing that, we recall that given $x\in \interior{\cone}$, there exists a norm $\|\cdot\|_x$ in $\R^N$, known as the \textit{order-unit norm} induced by $x$ in \cite{lemmens2012nonlinear}, such that 
\begin{equation}\label{order_unit_norm}
    \|y\|_x=M(y/x)  \qquad \forall y\in \cone.   
\end{equation}
\begin{proposition}\label{Prop_family_Power_method_characterizes_JSR}
    Let $\family=\{f_i\}_{i\in I}$ be a family of continuous, subhomogeneous and order-preserving maps on a cone $\cone$. Then for any $x\in \interior(\cone)$
    $$\cradius(\family)=\limsup_{k\rightarrow \infty}\sup_{f\in \Sigma_k(\family)}M(f(x)/x)^{\frac{1}{k}}.$$
Moreover if $f_{i}$ is homogeneous for any $i\in I$ then $\cradius(\family)=\inf_{k}\sup_{f\in \Sigma_k(\family)}M(f(x)/x)^{\frac{1}{k}}.$
\end{proposition}
\begin{proof}
Consider the norm induced by any interior point $x$ as in \eqref{order_unit_norm} for which we have $M(f(x)/x)=\|f(x)\|_x$. 
Then, for any map $f\in \Sigma_k(\family)$ we have $M(f(x),x)\leq \|f\|_x=\sup\{ M(f(y)/x)|y \text{ s.t. }M(y/x)\leq 1\}$, in particular 
\begin{equation}
    \limsup_{k\rightarrow \infty}\sup_{f\in \Sigma_k(\family)}M(f(x)/x)^{\frac{1}{k}}\leq \limsup_{k\rightarrow \infty}\sup_{f\in \Sigma_k(\family)}\|f\|_x^{\frac{1}{k}}= \cradius(\family),
\end{equation}
where we have used the fact that the value of the joint spectral radius does not depend on the norm by means of which it is defined.
On the other hand for any $f\in \Sigma_k(\family)$ we can write 

\begin{equation}
\|f\|_x=\sup\{M(f(y)/x)|\; y \text{ s.t. } M(y/x)\leq 1\},
\end{equation}
But $M(y/x)\leq 1$ means $y\cleq x$ which yields $f(y)\cleq f(x)\cleq M(f(x)/x)x$, i.e. $M(f(y)/x)\leq M(f(x)/x)$. In particular we have $\|f\|_x\leq M(f(x)/x)$, thus we have the opposite inequality:
\begin{equation}
\cradius(\family)=\limsup_{k\rightarrow \infty}\sup_{f\in \Sigma_k(\family)}\|f\|_x^{\frac{1}{k}}\leq \limsup_{k\rightarrow \infty}\sup_{f\in \Sigma_k(\family)}M(f(x)/x)^{\frac{1}{k}}.
\end{equation}
To conclude the proof we only miss to prove that, in the homogeneous case, the sequence $c_k=\sup_{f\in \Sigma_k(\family)}M(f(x)/x)^{\frac{1}{k}}$ admits a limit in $k$ that is equal to the infimum of the sequence.
To this end, observe that given $g_1\in \Sigma_k(\family)$ and $g_2\in \Sigma_h(\family)$
\begin{equation}
   g_2\big(g_1(x)\big)\leq c_k^k g_2(x) \leq c_k^k c_h^h x , 
\end{equation}
i.e. $c_{h+k}^{h+k}\leq c_h^h c_k^k$. In particular taking logarithms we have 
$(h+k)\log(c_{h+k})\leq h \log(c_h)+k\log(c_k)$, thus by the Fekete' Lemma we know that there exists the limit of $\log(c_k)$ and that it is equal to the infimum of the sequence, which conludes the proof. 
\end{proof}
\begin{corollary}\label{Cor:power_method_representation_joint_spectral_radius}
    Let $\family=\{f_i\}_{i\in I}$ be a family of continuous, homogeneous and order-preserving maps. Then, for any $x\in \interior(\cone)$ and norm $\|\cdot\|$
    $$\limsup_{k\rightarrow \infty}\sup_{f\in \Sigma_k(\family)}\|f(x)\|^{\frac{1}{k}}=\cradius(\family).$$
\end{corollary}
\begin{proof}
    The proof is an easy consequence of the last \Cref{Prop_family_Power_method_characterizes_JSR} and the equivalence of the norms. Indeed there exist constants $c,C>0$ such that for any $y\in \cone$ and $x\in\interior(\cone)$ 
    $$cM(y/x) \leq\|y\|\leq  C M(y/x).$$
    Thus plugging such inequality in expression of the JSR in \Cref{Prop_family_Power_method_characterizes_JSR} yields the thesis.
\end{proof}
Note that \Cref{Cor:power_method_representation_joint_spectral_radius} states that the joint spectral radius of a family $\family$ measures the maximal rate of growth of any point in the interior of the cone. An analogue result allows one to express the spectral radius of a homogeneous and order-preserving map $f$ as the rate of growth of the trajectory of any interior point of the cone when iteratively applying $f$ (see e.g. proposition 5.3.6 in \ \cite{lemmens2012nonlinear}).

\begin{remark}\label{reamrk_competitive_spectral_radius}
In \cite{akian2024competive}, Akian, Gaubert, and Marchesini consider a game where the payoff is the escape rate of a switched system. At every stage, the two players of the game play in such a way as to maximize and minimize the payoff, respectively. In particular, the escape rate can be expressed in terms of the competitive spectral radius of the system. If we consider the simplified case where the player that wants to minimize the payoff does not take any action, we see that the joint spectral radius that we have defined in \Cref{SEC:JSR and the stability} can be considered as a particular case of the competitive spectral radius in \cite{akian2024competive}. Indeed, if we assume that the actions taken from the player that wants to maximize the payout are given by the maps in $\family$, then the competitive spectral radius by Akian, Gaubert, and Marchesini can be written as:  
\begin{equation}
    \rho_{comp}(\family)=\limsup_{k\rightarrow \infty} \sup_{f\in \Sigma_k(\family)}\frac{\mathrm{d}_{Funk}(f(x^*),x^*)}{k},
\end{equation}
where $x^+$ is the initial state and $\mathrm{d}_{Funk}(y,x)=\log (M(y/x))$. In particular, from \Cref{Prop_family_Power_method_characterizes_JSR}, we have that if $x^*\in \interior(\cone)$, then
\begin{equation}
    \log\big(\cradius(\family)\big)=\limsup_{k\rightarrow\infty} \sup_{f\in \Sigma_k(\family)}\frac{\log\big(M(f(x^*)/x^*)\big)}{k}=\rho_{comp}(\family).
\end{equation}
Interestingly, Akian, Gaubert, and Marchesini study dual reformulations of the competitive spectral radius. We also study dual reformulations of the joint spectral radius in the next sections \Cref{sec:duality_monotone_prenorms} and \Cref{Sec:gen_joint_spectral_radius}. However, the connection between our and their dual reformulations, if it exists, seems not straightforward.
\end{remark}

\subsection{Duality by monotone prenorms}\label{sec:duality_monotone_prenorms}
The aim of this subsection is to discuss a dual formulation for the joint spectral radius of a homogeneous and order-preserving family. Our dual formulation is inspired by the dual formulation of the joint spectral radius of a family of matrices $\family=\{A_i\}_{i\in I}$, in which case it is possible to write
\begin{equation}\label{Linear_dual_spectral_radius}
\radius(\family):=\inf_{\|\cdot\|}\sup_{A_i\in \family}\|A_i\|\,,
\end{equation}
where the infimum is performed over all the possible norms in $\R^n$. The last dual formulation in the linear case was originally proved by Rota and Strang in \cite{rota1960note} and later discussed by different other authors, see \cite{elsner1995generalized,shih1997asymptotic}.

We start discussing the analogue result in the case of a single homogeneous and order-preserving map $f$. %
In this case, we will observe that the duality representation in \eqref{Linear_dual_spectral_radius} can be extended to the nonlinear setting. In particular, we say that there exists an extremal norm if there exists some norm realizing the minimum in \eqref{Linear_dual_spectral_radius}.
In the proof of the next theorem we make use of a result from \cite{lemmens2012nonlinear} which states that any continuous, homogeneous and order-preserving map and its dominant eigenpair can be approximated by a sequence of continuous, homogeneous and order-preserving maps having the dominant eigenvector in the interior of the cone; We recall this result in details in \Cref{Thm_existence_of_eigenvector}.
\begin{theorem}\label{Theorem_duality_cone_spectral_radius}
Let $f$ be a continuous, homogeneous, and order-preserving map on a closed cone $\cone$. Then
$$
\cradius(f)=\inf_{\|\cdot\|}\|f\|\,.
$$
Moreover, if $\ceigenvector\in\interior(\cone)$, then  there exists an extremal norm given by $\|\cdot\|_{\ceigenvector}$.
\end{theorem}

\begin{proof}
    The inequality 
\begin{equation}
\cradius(f)\leq\inf_{\|\cdot\|}\|f\|\,.
\end{equation}
is easily proved since given any norm $\|\cdot\|$, 
\begin{equation}
\cradius(f)=\limsup_{m\rightarrow\infty}\|f^m\|^{\frac{1}{m}}\leq \|f\|
\end{equation}
To prove the opposite, start by assuming that there exists an eigenfunction associated to the spectral radius in the interior of the cone, $\ceigenvector\in\interior(\cone)$. Then consider the $\ceigenvector$-norm introduced in \eqref{order_unit_norm}.
\begin{equation}
   \|y\|_{\ceigenvector}=M(y/\ceigenvector)\,, \qquad \forall y\in\cone\,.
\end{equation}
In particular, observe that, since $f$ is order preserving $y\cleq M(y/\ceigenvector)\ceigenvector$ yields $f(y)\cleq M(y/\ceigenvector)\cradius\ceigenvector$,
i.e. $\|f(y)\|_{\ceigenvector}\leq \cradius(f) \|y\|_{\ceigenvector}$. Thus 
\begin{equation}
\|f\|_{\ceigenvector}=\sup_{y\in\cone\setminus\{0\}}\frac{\|f(y)\|_{\ceigenvector}}{\|y\|_{\ceigenvector}}\leq \cradius(f)\,.
\end{equation}
Which concludes the proof in the case $\ceigenvector\in\interior(\cone)$.

Next, assume no eigenfunction relative to $\cradius$ is included in $\interior(\cone)$. Then, from \Cref{Thm_existence_of_eigenvector}, there exists a sequence of homogeneous and order-preserving maps $\{f_{n}\}_{n}$, defined as $f_n:=f+\epsilon_n u \psi$ for some $\psi\in\interior(\cone^*)$ and $u\in\interior(\cone)$, and $\big(x_{f_n},\cradius(f_{n})\big)$ such that :
\begin{equation}\label{Proof_Thm_extremal_norm_eq2}
f_{n}(x_{f_n})=\cradius(f_{n})x_{n}\in \interior(\cone)\quad \mathrm{and}\quad \big(x_{f_n},\cradius(f_{n})\big)\rightarrow \big(\ceigenvector,\cradius(f)\big)
\end{equation}
Moreover, since any $x_{f_n}\in\interior(\cone)$, the first part of the proof shows that any $x_{f_n}$ induces a norm $\|\cdot\|_{n}$, such that
\begin{equation}
\sup_{y\in \cone}\frac{M(f_{n}(y)/ x_{f_{n}})}{M(y/ x_{f_{n}})}=\|f_{n}\|_{n}=\cradius(f_{n})\,.
\end{equation}
To conclude observe that since $f(y)\cl f_{n}(y)$ $\forall y\in\cone$: 
\begin{equation}\label{Proof_Thm_extremal_norm_eq3}
M(f(y)/x_{n})\leq M(f_{n}(y)/x_{n})\quad \forall y\in\cone \qquad \mathrm{i.e.}\qquad \|f\|_{n}\leq \|f_{n}\|_{n}=\cradius(f_{n})
\end{equation}
Thus \eqref{Proof_Thm_extremal_norm_eq2} and  \eqref{Proof_Thm_extremal_norm_eq3} yield
\begin{equation}
    \inf_{\|\cdot\|}\|f\|\leq  \cradius\,,
\end{equation}
which concludes the proof
\end{proof}

The following technical result follows as a consequence, and it will be particularly useful in the remainder of the manuscript.

\begin{corollary}\label{Corollary_x-norm}
Let $f$ be a continuous, homogeneous, and order-preserving map on a closed cone $\cone$ and assume $\ceigenvector\in\interior(\cone)$. Then, given $y\in\interior(\cone)$, we have
\begin{equation}
\|f\|_y \leq \cradius(f) \frac{M(\ceigenvector/y)}{m(\ceigenvector/y)}\,.
\end{equation}
\end{corollary}
\begin{proof}
Since both $\ceigenvector$ and $y$ are in the inerior of the cone, there exist $M(\ceigenvector/y)<\infty$ and $m(\ceigenvector/y)>0$ such that
\begin{equation}\label{eq_1_corollary_x-norm}
m(\ceigenvector/y)y\cleq \ceigenvector \cleq M(\ceigenvector/y)y\,.
\end{equation}
Since $f$ is order preserving, we derive 
\begin{equation}\label{eq_2_corollary_x-norm}
m(\ceigenvector/y)f(y)\cleq \cradius(f)\ceigenvector \cleq M(\ceigenvector/y)f(y)\,.
\end{equation}
Now given $z\in\cone$ with $z\cleq M(z/y)y$, using \eqref{eq_1_corollary_x-norm} and \eqref{eq_2_corollary_x-norm} we note:
\begin{equation}\label{eq_3_corollary_x-norm}
f(z)\cleq M(z/y)f(y)\cleq \frac{M(z/y)}{m(\ceigenvector/y)}\cradius(f)\ceigenvector \cleq M(z/y)\frac{M(\ceigenvector/y)}{m(\ceigenvector/y)}\cradius(f)y.
\end{equation} 
In particular, the last inequality \eqref{eq_3_corollary_x-norm} yields 
\begin{equation}
\|f(z)\|_y\leq \cradius(f) \frac{M(\ceigenvector/y)}{m(\ceigenvector/y)} \|z\|_y,
\end{equation}
and thus the thesis.
\end{proof}

Next, we consider the case of a family of multiple maps $\family$. In this case, differently from the linear case, norms are not sufficient to provide a dual representation of the joint spectral radius; we discuss a counterexample later on in this section, see \Cref{ex:not_existence_of_continuous_extr_prenorm}. However, if we consider prenorms, which are more general than norms, we can recover a dual representation of the joint spectral radius of a family $\family$. 
\begin{definition}\label{Def:prenorm}
A prenorm $\Theta$ on a closed cone $\cone$ is a function satisying the following properties (see also \cite{MERIKOSKI1991315}):
\begin{enumerate}
    \item $\Theta$ is continuous 
    \item $\Theta(x)\geq 0$ and $\Theta(x)=0$ if and only if $x=0$
    \item $\Theta(cx)=c\, \Theta(x)$ for all $c\geq 0$, $x\in \cone$.
\end{enumerate} 
\end{definition}
As a norm, also any prenorm on a cone $\cone$ induces a prenorm on the space of maps on $\cone$, i.e.
\begin{equation}
    \Theta(f)=\sup_{x\in \cone\setminus\{0\}} \frac{\Theta(f(x))}{\Theta(x)}
\end{equation}
In particular, observe that, given a continuous, homogeneous, order preserving map $f$ on a cone $\cone$, since there exists $x_f\in\cone$ such that $f(x_f)=\cradius(f) x_f$, then clearly $\Theta(f)\geq \cradius (f)$. 
Moreover, by compactness, given a prenorm $\Theta$ and a norm $\|\cdot\|$, there exist $c,C>0$ such that 
\begin{equation}
    c\|x\|\leq \Theta(x)\leq C\|x\| \qquad \forall x\in \cone.
\end{equation}
Thus, we have the following expression for the cone spectral radius of $f$:
\begin{equation}
    \cradius(f)=\lim_{m\rightarrow \infty}\big(\Theta(f^m)\big)^\frac{1}{m},
\end{equation}
where $\Theta$ is any prenorm. 

\begin{remark}
Note that given a homogeneous nonnegative functional $\Theta$, the 
“equivalence" with the norms is sufficient to guarantee all the above definitions and results stated for prenorms, even if $\Theta$ is not continuous. In particular, we say that $\Theta$ is a discontinuous prenorm if only $2$ and $3$ in \Cref{Def:prenorm} are satisfied. In this case, we say that $\Theta$ is non-degenerate if, given any norm $\|\cdot\|$, there exist $c,C>0$ such that 
\begin{equation}
    c\|x\|\leq \Theta(x)\leq C\|x\| \qquad \forall x\in \cone.
\end{equation}

Note in particular that given any non-degenerate prenorm and $f_1,f_2$ continuous, homogeneous and order-preserving maps, then $\Theta(f_1\circ f_2)\leq \Theta(f_1)\Theta(f_2)$.
As a consequence, if $\family$ is a family of continuous, homogeneous and order-preserving maps, then 
\begin{equation}\label{sub_mult_prenorms}
    \cradius(\family)\leq \sup_{f\in \family}\Theta(f)
\end{equation} 
for any possibly discontinuous non-degenerate prenorm.
\end{remark}

We now consider the set of prenorms that are monotone with respect to the cone $\cone$.
We say that a prenorm $\Theta$ is monotone if $\Theta$ is order preserving as a function from the cone $\cone$, to the cone $\R^+$, i.e. if $x\cgeq y$ then $\Theta(x) \geq \Theta(y)$. We denote by $\MPN(\cone)$ the set of monotone prenorms on a cone $\cone$ and by $\DMPN(\cone)$ the set of non-degenerate monotone and possibly discontinuous prenorms on a cone $\cone$.
Further, consider a family of continuous, order-preserving and homogeneous maps $\family := \{f_i\}_{i\in I}$ on a cone $\cone$. We say that a monotone prenorm $\Theta$, is extremal for the family  $\family$, if $\sup_{f\in\family}\Theta(f)=\cradius(\family)$.

The following duality result for the joint spectral radius holds
%. 
%

%
%
\begin{theorem}\label{Thm_duality_for_the_cone_spectral_radius}
    Let $\family=\{f_i\}_{i\in I}$ be a bounded family of continuous, order-preserving and 1-homogeneous maps on a cone $\cone$ and let $\cradius(\family)$ be the joint cone spectral radius. Then 
    \begin{equation}
        \cradius(\family)=\inf_{\Theta\in \DMPN(\cone)}\sup_{i\in I}\Theta(f_i).
    \end{equation}
    If in addition there exists some $L>0$ such that any $f_i\in \family$ is $L$-Lipschitz then
    \begin{equation}
        \cradius(\family)=\inf_{\Theta\in \MPN(\cone)}\sup_{i\in I}\Theta(f_i).
    \end{equation}
    In particular, if $\Sigma(\family')$ is bounded, then $\family$ admits an extremal, possibly discontinuous, monotone prenorm, where $\family'=\{f_i/\cradius(\family)\}$.   
\end{theorem}
    \begin{proof}
        If $\Theta\in \DMPN(\cone)$, then we trivially observe that since $\Theta$ is non-degenerate 
        \begin{equation}
            \cradius(\family)=\limsup_{k\rightarrow \infty}\sup_{f\in\Sigma_k(\family)}(\Theta(f))^{\frac{1}{k}}.
        \end{equation}
        Moreover, for any $f\in\Sigma_k(\family)$, $f=f_{i_1}\cdots f_{i_k}$
        \begin{equation}
            \Theta(f)\leq \Theta(f_{i_1})\cdots \Theta(f_{i_k}),
        \end{equation}
        yielding the inequality
        \begin{equation}
            \cradius(\family)\leq \sup_{i\in I}\Theta(f_i).
        \end{equation}
        To prove the opposite inequality, let $\psi\in \interior(\cone^*)$, then for any $\epsilon>0$ let $\family_\epsilon=\{f_i/(\cradius(\family)+\epsilon)\}_{i\in I}$ (note that since $\family$ is bounded, $\cradius(\family)$ is finite and $\family_\epsilon$ is non-zero) and define:
        \begin{equation}
            \Theta_{\epsilon}(x)=\sup_{f\in \Sigma(\family_{\epsilon})}\psi(f(x)).
        \end{equation}
        It is trivial to observe that  $\Theta_{\epsilon}$ is $1$-homogeneous, moreover, since $\Sigma_0(\family_\epsilon)=\{Id\}$, then $\Theta_{\epsilon}(x)\geq \psi(x)\geq 0$, with equality if and only if $x=0$. We have proved that $\Theta_\epsilon$ is a possibly discontinuous prenorm, moreover
        since any $f\in \Sigma(\family_\epsilon)$ is order preserving and $\psi(x)\geq \psi(y)$ whenever $x\cgeq y$, then $\Theta_\epsilon$ is monotone. We miss to prove that $\Theta_{\epsilon}$ is non-degenerate, first observe that:
        \begin{equation}
            \sup_x \frac{\|x\|}{\Theta_{\epsilon}(x)}\leq \sup_x \frac{\|x\|}{\psi(x)}=\frac{1}{c} <\infty
        \end{equation}
        for some $c>0$, where we have used that $\psi\in \interior(\cone^*)$.
        On the other hand, by definition of the cone joint spectral radius \eqref{Cone_joint_spectral_radius}, for any $\delta>0$ there exists $N_0$ sufficiently large such that for any $k>N_0$, 
        \begin{equation}        \sup_{f\in\Sigma_k(\family_\epsilon)}\|f\|^{\frac{1}{k}}\leq \frac{\cradius(\family)}{\cradius(\family)+\epsilon}+\delta.
        \end{equation}
        Thus $\Sigma(\family_\epsilon)$ is bounded, i.e. there exists $\alpha>0$ such that for any $x\in\cone$, $\sup_{f\in\family_{\epsilon}}\|f(x)\|\leq \alpha\|x\|$. In particular, we have 
        \begin{equation}
            \Theta_{\epsilon}(x)=\sup_{f\in \Sigma(\family_\epsilon)}\psi\big(f(x)\big)\leq \beta \sup_{f\in \Sigma(\family_\epsilon)}\|f(x)\|\leq \alpha \beta \|x\|,
        \end{equation}
        for some $\beta>0$, which yields 
        \begin{equation}
            \sup_{x\in\cone}\frac{\Theta_{\epsilon}(x)}{\|x\|}\leq \alpha\beta=:C<\infty,
        \end{equation}
        concluding the proof of the non-degeneracy.
        To conclude the theorem, observe that $
            \Theta_{\epsilon}(f(x))\leq \Theta_{\epsilon}(x)$, for any $f\in \family_{\epsilon}$ and $x\in\cone$, i.e.:
            \begin{equation}
                \Theta_\epsilon(f_i) \leq \cradius(\family)+\epsilon \qquad \forall i\in I.
            \end{equation}
          Observe also that, if $\Sigma(\family_0)=\Sigma(\family')$ is bounded, the same argument above yields the last claim of the theorem, i.e. the existence of an extremal prenorm.
          We miss to prove the Lipschitz case. In this case, recall from \Cref{Theorem_asymptotic_stability_via_JSR}, that since $\cradius(\family_\epsilon)<1$, the family $\family_\epsilon$ is asymptotically stable. In particular, for any $\epsilon>0$ there exists $N_\epsilon>0$ such that for any $k\geq N_\epsilon$ and any $f\in \Sigma_k(\family_\epsilon)$, $\psi(f(x))\leq \psi(x)$. As a consequence, we can write 
          \begin{equation}
            \Theta_{\epsilon}(x)=\sup_{f\in \cup_{k=1}^{N_\epsilon}\Sigma_k(\family_{\epsilon})}\psi(f(x)).
        \end{equation}
          Hover the last is the sup of a family of equi-Lipschitz maps and thus $\Theta_\epsilon$ is Lipschitz i.e. continuous, concluding the proof.          
    \end{proof}
We want to discuss briefly the differences in the dual representation of the joint spectral radius between the linear and the nonlinear case. One key distinction is that, in the nonlinear case, monotonicity takes the place of convexity. In the linear setting, the auxiliary functions $\Theta_\epsilon$ arising in the proof of the dual formulation are automatically convex due to the linear structure of the maps. However, this property does not extend to the nonlinear case, where there is no inherent reason for such functions to remain convex.

A second key distinction concerns the continuity of the prenorms used in the dual representation. In the linear setting, a bounded family of linear maps is necessarily equi-Lipschitz, which guarantees the continuity of associated functionals. This regularity is generally lost in the nonlinear setting, where boundedness does not imply any form of equi-continuity. As the next example illustrates, this can lead to prenorms that are not continuous even when the family of maps is uniformly bounded.
\begin{example}\label{Example_uniform_convergence}
    Let $\cone=\R^2_+$ and consider the family of maps $\family:=\{g_n\}$, where:    $$\big(g_n(x)\big)_{i}=\big(t^{\frac{1}{n}}{x_{i}}^{\frac{1}{n}}+(1-t^{\frac{1}{n}})\min\{x_1,x_2\}^{\frac{1}{n}}\big)^n \qquad i=1,2$$
    with $0\leq t\leq1$ that is fixed.
    Observe that, if $x_1\leq x_2$ then $g_n(x)_1=x_1$. In addition, if $0=x_1 \leq x_2$ then $g_n(x)=[0,t x_2]$. Now we claim that for $0<x_1\leq x_2$, then $g_n(x)\rightarrow_{n\rightarrow\infty} [x_1,x_2]$. This proves that the limit of $g_n$ is given by:
    \begin{equation}
    g(x)=\begin{cases}
    x & \qquad \text{if }\; x_1,x_2\neq 0\\
    (0,tx_2) & \qquad \text{if }\; x_1=0\\
    (tx_1,0) & \qquad \text{if }\; x_2=0
    \end{cases}.
    \end{equation}
    Thus $g$ is not continuous on the boundary of $\cone$.
    To prove the claim assume without loss of generality that $0< x_1\leq x_2$ and write
    \begin{equation}
    g_n(x)=\Big[x_1\,,\; x_2\Big(t^{\frac{1}{n}}+(1-t^{\frac{1}{n}})(x_1/x_2)^{\frac{1}{n}}\Big)^n \;\Big]
    \end{equation}
    Now, observe that using Taylor expansions of the exponential and the logarithm:
    \begin{equation}
    \begin{aligned}        
    &\lim_{z\rightarrow 0}\Big(t^z+(1-t^z)\alpha^z \Big)^{\frac{1}{z}}=\lim_{z\rightarrow 0} \Big(1+z\log(t)+1+z\log(\alpha)-1-z\log(\alpha t)+o(z)\Big)^{\frac{1}{z}}\\
     =&\lim_{z\rightarrow 0}\Big(1+ o(z)\Big)^\frac{1}{z}=\lim_{z\rightarrow 0} exp\bigg(\frac{\log(1+ o(z))}{\log(1+z)} \cdot \frac{\log(1+z)}{z}\bigg)\\
     =&\lim_{z\rightarrow 0} exp\bigg(\frac{o(z)}{z+o(z)} \cdot \frac{\log(1+z)}{z}\bigg)=1,
     \end{aligned}
     \end{equation}
     where $0<t,\alpha<1$.
     Thus
    \begin{equation}
     \lim_{n\rightarrow \infty}g_n(x)=\Big[x_1\,,\; x_2 \Big], \qquad \forall x_1,x_2>0
    \end{equation}
    i.e. the claim.
\end{example}
The previous example highlights an important phenomenon:  in the non-linear case, the limit of a bounded sequence of continuous, homogeneous, and order-preserving maps, if it exists, may fail to be continuous. However, we know that continuous, homogeneous, and order-preserving maps are not expansive with respect to the Thompson metric; in particular, this means that they are equicontinuous in the interior of any face of the cone. Thus, we can apply the Ascoli-Arzelà theorem for noncompact spaces to obtain the following result, providing us with sufficient conditions for the existence of a continuous limit map on the interior of the faces.

\begin{proposition}[Ascoli-Arzelà]\label{Prop_ascoli_arzela}
    Let $g_n:\cone\rightarrow \cone'$ be a bounded sequence of continuous, homogeneous and order-preserving maps from a closed cone $\cone$ to a closed cone $\cone'$ and assume that for any $x\in \interior(\cone)$ there exists some compact set $D(x)\in \interior(\cone')$ such that $g_i(x)\in D$ for all $i\in I$. Then there exists a subsequence $g_{n_k}$ and a limit continous, homogeneous and order-preserving map $g*:\interior(\cone)\rightarrow \cone'$ such that $g_{n_k}$ converges uniformly to $g*$ on compact sets.
\end{proposition}
\begin{proof}
    The proof is a direct consequence of the Ascoli-Arzelà theorem on metric non compact spaces \cite{bourbakigeneral}. This states that if
 $X$ is a topological space and $Y$ a Hausdorff uniform (any metric space is uniform)  space, and $H$ is a equicontinuous set of continuous functions such that $H(x)$ is relatively compact in $Y$ for any $x\in X$. Then $H$ is relatively compact in the space of continuous functions equipped with the topology of compact convergence. In particular, since in our case $\interior(\cone)$ is not a compact metric space when endowed with Thompson metric, in order to apply the Ascoli-Arzelà theorem we need the compactness of the trajectory of any point with respect to the Thompson topology, i.e. we need the trajectory not to go to the boundary of the cone. This is guaranteed for any $x\in \interior\cone$ by the existence of the compact set $D(x)\in \interior(\cone)$.
\end{proof}

\subsubsection{On the existence of extremal monotone prenorms}

We now investigate sufficient conditions for the existence of an extremal monotone prenorm, that is, a minimal prenorm in the expression of $\cradius(\family)$ as characterized in \Cref{Thm_duality_for_the_cone_spectral_radius}. To this end, we focus on the case where the cone $\cone$ is polyhedral. Polyhedral cones are particularly advantageous because they satisfy a regularity property known as the \textbf{G}-condition, which ensures that any convergent sequence remains comparable with its limit.

\begin{definition}[Condition \textbf{G}]\label{Def_G_condition}
    A cone $\cone$ satisfies condition \textbf{G} if and only if for every sequence $\{x_k\}_k \subseteq \cone$ with $\lim_k x_k = x$ and for each $0 < \lambda < 1$, there exists $m \in \mathbb{N}$ such that $\lambda x \cleq x_k$ for all $k \geq m$.
\end{definition}

It is known that a cone $\cone$ satisfies condition \textbf{G} if and only if it is polyhedral; see Lemma 5.1.4 in \cite{lemmens2012nonlinear}.

Moreover, when the \textbf{G}-condition holds, any continuous, homogeneous, and order-preserving map $f$ defined on the interior of $\cone$ with values in another cone $\cone'$ admits a continuous extension to the closure of $\cone$, which remains homogeneous and order-preserving. For details, see Theorems 5.1.2 and 5.1.5 in \cite{lemmens2012nonlinear}, and \Cref{Appendix_nonlinear_Perron Frobenius}, particularly \Cref{Thm_extension_of_order_preserving_functions}.

With these tools in place, we are ready to study the existence of extremal monotone prenorms. We start by observing in \Cref{Proposition_nonlinear_reducibility} that a necessary condition for the non-existence of a, possibly discontinuous, extremal prenorm is the invariance of certain families of faces of the cone under all maps in the family. This mirrors the situation in the linear case, where the absence of an extremal norm is associated with the invariance of proper linear subspaces under the action of the family, i.e., to the reducibility of the family of matrices; see \cite{elsner1995generalized}. In particular,  \Cref{Proposition_nonlinear_reducibility} provides sufficient conditions for the existence of a, possibly discontinuous, monotone prenorm; see \Cref{Corollary_irreducibility_yields_boundedness}. Finally, in \Cref{Prop_existence_of_continuous_extremal_prenorm} and \Cref{corollary_existence_cont_extr_prenorm} we provide sufficient conditions for the existence of an extremal prenorm that is also continuous. In particular, we observe that if the mappings of the family map the interior of the cone in itself, a monotone extremal prenorm defined only on the interior of the cone can be extended with continuity to the boundary.
\begin{proposition}\label{Proposition_nonlinear_reducibility}
    Let $\family=\{f_i\}_{i\in I}$ be a bounded family of continuous, homogeneous, and order-preserving maps on a polyhedral cone $\cone$. If $\family$ does not admit a, possibly discontinuous, non-degenerate extremal prenorm, then there exists $P=\bigcup_{j=1}^h P_j$, where any $P_j$ is a face of the boundary of the cone, such that 
    \begin{equation}
        f_i(x)\in P \qquad \forall i\in I,\; x\in P.
    \end{equation}
\end{proposition}
\begin{proof}
Assume $\Theta_{\epsilon_n}$ to be a minimizing sequence of possibly-discontinuous non-degenerate monotone prenorms on $\cone$ such that $\sup_i \Theta_{\epsilon_n}(f_i)\leq \cradius(\family)+\epsilon_n$ .
Observe that, since the cone is polyhedral, any such prenorm is lower semi-continuous. Indeed, if $\cone$ is polyhedral it satisfies condition \textbf{G}, (see \ref{Def_G_condition}).
Thus, since any $\Theta_{\epsilon}$ is monotone, given $x_n\rightarrow x$ we know that for any $\lambda<1$, there exists $m>0$ such that $\lambda x\cleq x_k$ for all $k\geq m$, yielding 
\begin{equation}
    \liminf_{n\rightarrow \infty} \Theta_\epsilon(x_n)\geq \lambda\Theta_\epsilon(x) \qquad \forall \,0<\lambda<1,
\end{equation}
i.e., the lower semicontinuity.
Now let $u\in \interior(\cone)$ and w.l.o.g. assume that the prenorms $\Theta_{\epsilon_n}$ are all such that 
\begin{equation}
    \Theta_\epsilon(u)=\|u\|_u=1 \qquad \forall \epsilon>0.
\end{equation}
Then for any $x\in\interior(\cone)$ using the monotonicity of $\Theta_{\epsilon}$ we have $m(x/u)\leq \Theta_{\epsilon_n}(x)\leq M(x/u)$, i.e. we can apply the Ascolì-Arzelà theorem (see \Cref{Prop_ascoli_arzela}) to the maps $\Theta_{\epsilon_n}:(\interior(\cone), \distT)\rightarrow (\R_{>0},\distT)$ and derive the existence of some continuous $\Theta=\lim_{n\rightarrow \infty}\Theta_{\epsilon_n}$, where the limit has to be understood up to some subsequence.
Analogously, if $P$ is a face of the cone and $x\in P$, then $\Theta_{\epsilon_n}(x)\leq M(x/u)$, i.e. all the trajectories stay bounded. 
Thus, if there exists some $y\in \interior(P)$ and subsequence $\Theta_{\epsilon_{n_k}}$ such that $\Theta_{\epsilon_{n_k}}(x)>c>0$, using Ascoli-Arzelà we can extract also a limit for $\Theta_{\epsilon_{n_k}}:(\interior(P), \distT)\rightarrow (\R_{>0},\distT)$, otherwise any subsequence is such that $\lim_{k\rightarrow \infty}\Theta_{\epsilon_{n_k}}(x)=0$. In both the cases we can find a poitwise limit of $\Theta_{\epsilon_n}$ on the interior of $P$, up to subsequences. 
Finally, by the finiteness of the number of faces, we can derive the existence of a punctual limit (up to some subsequence) $\Theta=\lim_{n\rightarrow \infty}\Theta_{\epsilon_n}$ such that:
\begin{equation}    \Theta\big(f_i(x)\big)=\lim_{n\rightarrow \infty}\Theta_{\epsilon_n}\big(f_i(x)\big)\leq \lim_{\epsilon_n\rightarrow 0}(\cradius(\family)+\epsilon_n) \Theta_{\epsilon_n}(x)=\cradius(\family)\Theta(x) \qquad \forall x\in \cone,\; i\in I.
\end{equation}
Then $\Theta$ is still $1$-homogeneous, monotone, and lower-semicontinuous.
Moreover $\Theta(u)=1$, thus 
\begin{equation}
    \sup_{x\in\cone}\frac{\Theta(x)}{\|x\|_u}=\sup_{\|x\|_u=1}\Theta(x)<1,
\end{equation}
indeed given $B=\{x\in\cone| \|x\|_u=1 \}$ it is compact and  given $x\in B$ we have $x\cleq  u$, yielding $\Theta(x)\leq \Theta(u)=1$. Thus, since $\family$ does not admit an extremal prenorm, we necessarily have:
\begin{equation}
    \inf_{x\in\cone\setminus\{0\}}\frac{\Theta(x)}{\|x\|_u}=\inf_{\|x\|_u=1}\Theta(x)=0.
\end{equation}
But then, by the lower semicontinuity of $\Theta$ and the compactness of $B$, there exists at least one point $x^*\in B$ such that $\Theta(x^*)=0$.
By the monotonicity of $\Theta$, we immediately observe $\Theta(y)=0$ for any $y\leq x^*$, thus $x^*$ belongs to the boundary of the cone, otherwise $\Theta(u)$ would be zero, and $\Theta$ is zero on any face dominated by $x^*$. Name $\{P_j\}_{j=1}^h$ all and only the faces of the cone such that $\Theta(x)=0$ for all $x\in P_j$, then since 
\begin{equation}
    \Theta\big(f_i(x)\big)\leq \cradius(\family) \Theta(x) \qquad \forall x\in\cone,
\end{equation}
$f(x)\in \bigcup_{j=1}^h P_j$ for all $x\in \bigcup_{j=1}^h P_j$, i.e. the thesis.
\end{proof}
In analogy with the linear case, we say that a family of maps is \textit{reducible} if it has an invariant family of boundary faces, as in \Cref{Proposition_nonlinear_reducibility}. The corollary below follows straightforwardly
\begin{corollary}\label{Corollary_irreducibility_yields_boundedness}
    Let $\family=\{f_i\}_{i\in I}$ be an irreducible bounded family of continuous, homogeneous, and order-preserving maps on a polyedral cone $\cone$. Then, $\family$ admits a, possibly discontinuos, non-degenerate extremal prenorm $\Theta$ and $\Sigma(\family')$ is bounded, where $\family'=\{f_i/\cradius(\family)\}_{i\in I}$.
\end{corollary}

Next, we discuss an example where an extremal prenorm, possibly discontinuous, may not exist.

\begin{example}\label{ex:not_existence_of_continuous_extr_prenorm}
    Consider the map
    \begin{equation}
        M_{-1}(x_1, x_2) =
        \begin{cases}
            \left(x_1^{-1} + x_2^{-1}\right)^{-1}, & \text{if } x_1, x_2 > 0, \\
            0, & \text{otherwise}.
        \end{cases}
    \end{equation}
    It is known that $M_{-1}$ is continuous and monotone; see, for example, Section 4.1 in \cite{lemmens2012nonlinear}. Now, define the map
    \begin{equation}
        \begin{aligned}
            f: \mathbb{R}^2_+ &\longrightarrow \mathbb{R}^2_+, \\
            \underline{x} &\mapsto \left(x_1 + M_{-1}(\underline{x}),\ x_2\right).
        \end{aligned}
    \end{equation}
    By construction, $f$ is continuous, order-preserving, and homogeneous. Furthermore, observe that every boundary point of the cone is an eigenvector of $f$ associated with eigenvalue $1$, i.e., the entire boundary of the cone is invariant under $f$.

    Note also that the second component of $\underline{x}$ remains unchanged under the action of $f$. Therefore, the spectral radius of $f$ is exactly equal to $1$. However, consider an interior point $\underline{x} \in \interior(\cone)$ such that $x_1 = \alpha x_2$ for some $\alpha > 0$. Then we compute:
    \begin{equation}
        f(\underline{x}) = \left( \left( \alpha + \frac{\alpha}{\alpha + 1} \right) x_2,\ x_2 \right).
    \end{equation}
    Since the function $\alpha \mapsto \alpha/(1 + \alpha)$ is strictly increasing for $\alpha > 0$, and noting that
    \[
        \alpha + \frac{\alpha}{\alpha + 1} > \alpha,
    \]
    it follows by induction that
    \begin{equation}
        f^k(\underline{x}) \cgeq \left( \left( \alpha + k \cdot \frac{\alpha}{\alpha + 1} \right) x_2,\ x_2 \right).
    \end{equation}
    Therefore, for any interior point, the sequence $\{f^k(\underline{x})\}_{k \in \mathbb{N}}$ diverges, i.e., the orbit of $\underline{x}$ under $f$ is unbounded.

    As a consequence, there can not exist any non-degenerate extremal prenorm for $f$ despite the map being homogeneous, continuous, and order-preserving with spectral radius equal to $1$.
\end{example}
Finally, we investigate sufficient conditions for the existence of a continuous extremal monotone prenorm. In particular, we observe that whenever the family admits an extremal prenorm on the interior of the cone and leaves it invariant, then the extremal prenorm can be extended with continuity to the boundary of the cone.
\begin{proposition}\label{Prop_existence_of_continuous_extremal_prenorm}
    Let $\family=\{f_i\}_{i\in I}$ be a bounded family of continuous, homogeneous and order-preserving maps on a polyhedral cone $\cone$. If $\family$ admits a, possibly discontinuous, non-degenerate extremal prenorm and $f_i(x)\in \interior(\cone)$ for any $x\in\interior(\cone)$ and $i\in I$. Then $\family$ admits a continuous and monotone extremal prenorm. 
\end{proposition}
\begin{proof}
    Let $\Theta$ be an extremal prenorm for $\family$, i.e.:
\begin{equation}
    \Theta\big(f_i(x)\big)\leq \cradius(\family)\Theta(x)\qquad \forall x\in \cone,\;i\in I.
\end{equation}
 
 Then $\Theta|_{\interior(\cone)}$ is continuous, monotone and $1$-homogeneous. In particular, since $\cone$ is a polyhedral cone, from \Cref{Thm_extension_of_order_preserving_functions} we know that $\Theta|_{\interior(\cone)}$ admits a continuous, monotone and homogeneous extension to the boundary of $\cone$, say $\Theta^{ext}$.
    Now, let $y$ be a point on the boundary of $\cone$ and consider $\{x_n\}_{n\in \N}\in \interior (\cone)$ such that $x_n\rightarrow y$, $x_n\cgeq y$. 
    Then, for any $i\in I$:
    \begin{equation}
        \Theta^{ext}\big(f_i(y)\big)\leq  \Theta\big(f_i(x_n)\big)\leq \cradius(\family)\Theta(x_n)\rightarrow \cradius(\family)\Theta^{ext}(y).
    \end{equation}
   Hence $\sup_i \Theta^{ext}(f_i)\leq \cradius(\family)$. On the other hand, we know that the opposite inequality holds for any non-degenerate and possibly discontinuous prenorm because of the submultiplicativity of the induced operator prenorm \eqref{sub_mult_prenorms}. Finally, since $\Theta$ is non-degenerate, $\Theta^{ext}$ is easily proved to be non-degenerate as well, concluding the proof.
\end{proof}

\begin{corollary}\label{corollary_existence_cont_extr_prenorm}
    Let $\family=\{f_i\}_{i\in I}$ be a family of continuous, homogeneous, and order-preserving maps on a polyhedral cone $\cone$. If $f_i(\cone\setminus\{0\})\subset \interior(\cone)$ for all $i\in I$. Then $\family$ admits an extremal and continuous monotone prenorm.
\end{corollary}
\begin{proof}
    Since no family of boundary faces is left invariant by $\family$, Proposition \ref{Proposition_nonlinear_reducibility} yields the existence of an extremal monotone prenorm. Then, Proposition \ref{Prop_existence_of_continuous_extremal_prenorm} yields the continuity.
\end{proof}
Observe, however, that it is not always the case of the existence of an extremal prenorm. The same example shows also that taking the infimum in \Cref{Thm_duality_for_the_cone_spectral_radius} over the only continuous prenorms (or norms), generally does not yield the joint spectral radius.
\begin{example}
Let $\cone=\R^2_+$ and consider the family of maps:
$\family=\{f_n\}_{n\in \N}$ where $f_n(x)=(\|g_n(x)\|_1,0)^T$ and $g_n(x)$ that is defined as in \Cref{Example_uniform_convergence} with $t=1/2$. 
We first claim that 
\begin{equation}
    \|f(x)\|_1\leq (1/2)^{k-1}\|x\|_1  \text{ for all } f\in \Sigma_k(\family),
\end{equation} 
indeed $\|g_n(x)\|_1\leq \|x\|_1$ for any $n$ and $\|g_n(x)\|_1=(1/2)\|g_n(x)\|_1$ whenever $x_2=0$.
Then, the claim follows observing that, if $f\in \Sigma_k(\family)$ with $f=f_{i_k}\circ \dots, f_{i_1}$, then for any point $x$ in the cone, $f(x)=f_{i_k} \odot \dots \odot f_{i_{2}}(f_{i_1}(x))$ with $(f_{i_1}(x))_2=0$. 
Second we consider $x_0=(1,1)$ and we observe that $\|f_n(x_0)\|_1=\|x_0\|_1$ (indeed $g_n(x_0)=x_0$). In particular, we have the equality $\|f(x_0)\|_1=(1/2)^{k-1}\|x_0\|_1$ for any $f\in \Sigma_k(\family)$. 
As a direct consequence:
 \begin{equation}
    \sup_{f\in \Sigma_k(\family)}\|f\|_1=\frac{1}{2^{k-1}} \qquad \mathrm{i.e.} \qquad \cradius(\family)=\frac{1}{2}
\end{equation}
Now let $\epsilon<\frac{1}{2}$ and assume by contradiction that there exists a continuous prenorm $\Theta$ such that 
\begin{equation}
    \Theta(f_n(x))\leq \Big(\cradius(\family)+\epsilon\Big)\Theta(x) \qquad \forall n\in \N.
\end{equation}
Let $\{x_n\}\in \interior(\cone)$ a sequence such that $x_n\rightarrow (1,0)$ and $\|x_n\|_1=1$ for all $n\in \N$.
Since $\lim_{m\rightarrow \infty} g_m(x)=x$ for all $m \in \interior(\cone)$, $\lim_m f_m(x)=(1,0)$ for any $x$ such that $\|x\|_1=1$. Thus, if $\Theta$, as above, exists, we should have 
\begin{equation}\label{eq_1_countereample_continuity}
    \Theta\big((1,0)\big)\leq \Big(\frac{1}{2}+\epsilon\Big)\Theta(x_n) \qquad \forall n\in \N.
\end{equation}
However since $\Theta$ is continuous we also have $\lim_n\big(\Theta(x_n)\big)=\Theta\big((1,0)\big)$ which contradicts \eqref{eq_1_countereample_continuity}. 
We have proved that, for any $\epsilon<1/2$, there can not exist a continuous prenorm such that $\Theta(f_n)\leq \cradius(\family)+\epsilon$.

On the other hand, if we consider the discontinuous prenorm:
\begin{equation}
    \Theta*=\begin{cases}
        \|x\|_1 \quad &\text{if } x_1,x_2>0\\
        \displaystyle{\frac 1 2}\|x\|_1 &\text{otherwise} 
    \end{cases}
\end{equation}
then we observe that $\Theta*\big(f_n(x)\big)\leq \cradius(\family)\Theta*(x)$ for any $n\in \N$ and $x\in \cone$, i.e. $\Theta*$ is an extremal prenorm for the family $\family$. 
\end{example}

\subsection{Generalized cone joint spectral radius}
\label{Sec:gen_joint_spectral_radius}

In this subsection, we introduce and investigate an alternative formulation of the joint spectral radius: the generalized cone joint spectral radius. Inspired by the analogous construction in the linear setting \cite{BERGER199221}, this generalization replaces the norm of the composed maps in the semigroup in \eqref{Cone_joint_spectral_radius} with their spectral radii.

Let $\family$ be a family of continuous, homogeneous, and order-preserving maps. We define the generalized cone joint spectral radius as
\begin{equation}\label{Generalized_cone_joint_spectral_radius}
    \cgradius(\family) = \limsup_{k \to \infty} \sup_{f \in \Sigma_k(\family)} \left( \cradius(f) \right)^{\frac{1}{k}},
\end{equation}
where $\cradius(f)$ denotes the cone spectral radius of the map $f$, defined via the Gelfand formula in \Cref{Def_bonsall_spectral_radius}.

As a direct consequence of the properties of the cone spectral radius presented in \Cref{Prop_properties_cone_spectral_radius}, we observe that the $\limsup$ in \eqref{Generalized_cone_joint_spectral_radius} can be replaced by a supremum. That is,
\begin{equation}
    \cgradius(\family) = \sup_{k} \sup_{f \in \Sigma_k(\family)} \left( \cradius(f) \right)^{\frac{1}{k}}.
\end{equation}

Moreover, for any norm $\|\cdot\|$ and any homogeneous, order-preserving map $f$, it holds that $\|f\| \geq \cradius(f)$. This inequality implies the following general relationship:
\begin{equation}\label{inequalities_spectral_radii}
    \cgradius(\family) \leq \cradius(\family).
\end{equation}

In the linear setting, equality in \eqref{inequalities_spectral_radii} holds whenever the family is bounded \cite{BERGER199221, shih1997asymptotic, elsner1995generalized}. However, in the nonlinear setting, boundedness alone is not sufficient to ensure this equality. In particular, in \Cref{Example_spectral_radius_and_generalized_sr_can_be_different}, we present a counterexample involving a bounded family of nonlinear maps that are not uniformly Lipschitz. 

This naturally leads to the question of whether {equicontinuity} might be a sufficient condition to guarantee equality between $\cgradius(\family)$ and $\cradius(\family)$. Indeed, equicontinuity in the nonlinear case plays a role analogous to boundedness in the linear one. However, this question remains open and will be the subject of future investigations.

Before analyzing further the comparison of the classical and generalized joint spectral radii, we introduce a characterization of $\cgradius(\family)$ that closely mirrors the one given in \Cref{Prop_family_Power_method_characterizes_JSR} for $\cradius(\family)$. This alternative formulation shows that the generalized joint spectral radius can also be computed by leveraging the order-preserving structure of the maps.

In this case, however, we replace the cone-induced norms used in the classical setting with \emph{antinorms}, as introduced in \cite{MERIKOSKI1991315}. These antinorms offer a natural dual perspective and are particularly well-suited for the order-preserving setting.

\begin{proposition}
    Let $\family:=\{f_i\}_{i\in I}$ be a family of continuous, homogeneous, and order-preserving maps. Then 
$$\cgradius(\family)=\sup_{x\in \cone} \sup_{k}\sup_{f\in \Sigma_k} m(f(x)/x)^{\frac{1}{k}}$$ 
\end{proposition}
\begin{proof}
To prove it, first of all we {claim} that for any $x\in \cone$ and $f$ continuous, homogeneous and order-preserving then $\cradius(f)=\max_{x\in\cone} m(f(x)/x)$. Indeed, using the order preserving property and the inequality $f(x)\cgeq m(f(x)/x)x $, it is easily proved that 
\begin{equation}\label{eq_1_antinorms_generalized_spectral_radius}
\cradius(f)\geq \limsup_k \|f^k(x)\|^{\frac{1}{k}}\geq m(f(x)/x).
\end{equation}
On the other hand, if $x$ is an eigenvector relative to $\cradius(f)$, then $m(f(x)/x)=\cradius(f)$, proving the claim. 
In particular, from \eqref{eq_1_antinorms_generalized_spectral_radius}, we have the following inequality:
\begin{equation}
    \limsup_{k}\sup_{f\in \Sigma_k} m(f(x)/x)^{\frac{1}{k}}\leq \limsup_{k}\sup_{f\in \Sigma_k} \cradius(f)^{\frac{1}{k}}=\cgradius(\family) \qquad \forall x\in\cone
\end{equation}
Next, let $f_1\in \Sigma_k$ and $f_2\in \Sigma_h$, then 
\begin{equation}
    f_1\circ f_2(x)\cgeq m(f_2(x)/x) f_1(x)\cgeq m(f_2(x)/x) m(f_1(x)/x) x.
\end{equation}
In particular $\sup_{f\in \Sigma_{h+k}}m(f(x)/x)\geq \big(\sup_{f\in \Sigma_{h}}m(f(x)/x)\big) \big(\sup_{f\in \Sigma_{k}}m(f(x)/x)\big)$. Taking the logarithm we observe that the sequence $\{\log(\sup_{f\in \Sigma_{k}}m(f(x)/x))\}_{k\in \mathbb{N}}$ is super additive, in particular by the Fekete's Lemma $\log(\sup_{f\in \Sigma_{k}}m(f(x)/x))/k$ admits limit in $k$ and it is equal to the sup of the sequence, i.e.: 
\begin{equation}
    \limsup_{k}\sup_{f\in \Sigma_k} m(f(x)/x)^{\frac{1}{k}}=
    \sup_{k}\sup_{f\in \Sigma_k} m(f(x)/x)^{\frac{1}{k}}=
    \lim_{k}\sup_{f\in \Sigma_k} m(f(x)/x)^{\frac{1}{k}}
\end{equation}
Last, observe that for any $\epsilon>0$ there exists $k_0$ and $f\in \Sigma_{k_0}$ such that $\cradius(f^*)\geq (\cgradius(\family)-\epsilon)^{k_0}$. In particular, there exists $x^*\in \cone$ such that $m(f^*(x^*)/x^*)\geq (\cgradius(\family)-\epsilon)^{k_0}$ and we have 
\begin{equation}
\sup_k\sup_{f\in\Sigma_k}m\big(f(x^*)/x^*\big)^\frac{1}{k}\geq m\big(f^*(x^*)/x^*\big)^{\frac{1}{k_0}}\geq \cgradius(\family)-\epsilon.    
\end{equation}
Thus, we have found the inequality that we missed
\begin{equation}
\cgradius(\family)\leq\sup_{x\in \cone} \sup_{k}\sup_{f\in \Sigma_k} m(f(x)/x)^{\frac{1}{k}}
\end{equation}
and this concludes the proof.
\end{proof}

Next, we investigate sufficient conditions to guarantee the equality between the joint spectral radius and the generalized joint spectral radius. In particular, we assume that there exists a subcone $\cone'$ contained in the interior of $\cone$ such that any map $f\in\Sigma(\family)$ admits some eigenvector $\ceigenvector\in \cone'$ relative to the spectral radius $\cradius(f)$.
Under this condition, we prove that having $\cgradius<1$ corresponds to having $\family$ asymptotically stable. Before discussing the proof, we recall a dual characterization of $M(y/x)$ and $m(y/x)$ from \cite{lemmens2012nonlinear}.
Given a solid closed cone $\cone$ and $u\in\interior(\cone)$, we can define
\begin{equation}
    \Xi_u^*=\{\psi\in\cone^*\:|\;\psi(u)=1\}\,.
\end{equation}
$\Xi_u^*$ is a compact convex set and we denote by $\mathcal{E}_u^*$ the set of its extreme points. Then, it is possible to prove that 
\begin{equation}\label{dual_expression_M_m}
M(x/y)=\sup_{\psi\in\mathcal{E}_u^*}\frac{\psi(x)}{\psi(y)}\qquad \mathrm{and}\qquad m(x/y)=\inf_{\psi\in\mathcal{E}_u^*}\frac{\psi(x)}{\psi(y)}\,.
\end{equation}
%%%
%%%

\begin{theorem}\label{Thm_Equivalence_cone_spectral_radii_asymptotic_stability}
Let $\cone$ be a closed cone and $\family=\{f_i\}_{i\in I}$ a family of continuous, homogeneous, and order-preserving maps. Assume that there exists $\cone'\in \interior(\cone)$ such that any $f\in\Sigma(\family)$ admits some $\ceigenvector\in \cone'$. Then, if $\cgradius(\family)<1$, the family  $\family$ is asymptotically stable.
\end{theorem}
\begin{proof}
Start considering $\psi\in \interior(\cone^*)$ and define $\Xi=\{x\in \cone'\,|\; \psi(x)=1\}$\,. Then define $U=\{x\in\cone\:|\; \psi(x)\leq 1\}$.
Next, for any $x\in\Xi$ let
\begin{equation}
M(U/x)=\max_{y\in U}M(y/x)<\infty\,,
\end{equation} 
and define
\begin{equation}
C_x:=\{y\in \cone\,|\;y\cleq M(U,x)x\}\qquad \mathrm{and}\qquad V:=\bigcup_{x\in\Xi} C_x\,.
\end{equation}
By construction $U\subset C_x$ for any $x\in\Xi$ and $f(x)\in \cgradius^k(\family)V$ for any $x\in U$ and $f\in \Sigma_k(\family)$.
Indeed, given $f\in \Sigma_k(\family)$ and $x\in U$, we have that $x\in C_{\ceigenvector}$ where $\ceigenvector\in \Xi$ is the eigenvector of $f$ relative to $\cradius(f)$.  Then, from Theorem \ref{Theorem_duality_cone_spectral_radius} and \eqref{Conespectralradius_characterization}:
\begin{equation}
f(x)\in \cradius(f) C_{\ceigenvector}\subseteq \cgradius^k(\family)V\,\,.
\end{equation}
To conclude the proof we only miss to show that $V$ is bounded which is equivalent to have $\sup_{x\in\Xi}M(U/x)<\infty$. Note that using \eqref{dual_expression_M_m}
\begin{equation}
\sup_{x\in\Xi}M(U/x)=\sup_{x\in\Xi}\sup_{y\in U}\sup_{\phi\in \mathcal{E}_{u}^*}\frac{\phi(y)}{\phi(x)}
\end{equation}
where $u\in\interior(\cone)$, $\Xi_u^*=\{\phi\in\cone^*\:|\;\phi(u)=1\}$ and  $\mathcal{E}_u^*$ is the set of extreme points of $\Xi_u^*$. Then, note that since $\mathcal{E}_u^*$ is compact and $U$ is bounded, there exists $L\in \R^+$ s.t.
\begin{equation}
    \sup_{y\in U}\phi(v)\leq L \qquad \forall \phi\in \mathcal{E}_u^*\,.
\end{equation}
Moreover since $\Xi$ is compact and included in $\interior(\cone)$, $\inf_{x\in\Xi} \phi(x)>S_{\phi}>0$ and by compactness of $\mathcal{E}_u^*$, $\inf_{\phi\in\mathcal{E}_u^*} S_{\phi}>S>0$. Otherwise there should exist some $x_0\in \Xi$ and $\phi_0\in \Xi_u^*$ such that $\phi_0(x_0)=0$, contradicting the hypothesis $x_0\in \interior(\cone)$. In particular we have
\begin{equation}
\sup_{x\in\Xi}M(U/x)\leq \sup_{\phi\in \mathcal{E}_{u}^*} \sup_{x\in\Xi}\frac{L}{\phi(x)}\leq \frac{L}{S},
\end{equation}
concluding the proof.
\end{proof}

As a corollary, observe that whenever the hypotheses of \Cref{Thm_Equivalence_cone_spectral_radii_asymptotic_stability}
are satisfied by a family $\family$, the cone spectral radius and the generalized cone spectral radius are equal.

\begin{corollary}\label{Corollary_equality_of_generalized_and_classic_joint_spectral_radii_embedded_subcone}
Let $\cone$ be a closed cone and $\family=\{f_i\}_{i\in I}$ a family of continuous, homogeneous, and order-preserving maps. Assume that there exists $\cone'\in \interior(\cone)$ such that $\ceigenvector\in \cone'$ for any $f\in\Sigma(\family)$. Then 
\begin{equation}
\cradius(\family)=\cgradius(\family)\,.
\end{equation}
\end{corollary}
\begin{proof}
We already know that $\cradius(\family)\geq\cgradius(\family)$. Thus if $\cgradius(\family)=\infty$, there is nothing to prove, otherwise let 
\begin{equation}
\family_{\epsilon}=\bigg\{\frac{f_i}{\cgradius(\family)+\epsilon}\bigg| f_i\in\family \bigg\},
\end{equation}
then $\cgradius(\family_{\epsilon})=\cgradius(\family)/\cgradius(\family)+\epsilon<1$ and thus, from \Cref{Thm_Equivalence_cone_spectral_radii_asymptotic_stability} and \Cref{Theorem_asymptotic_stability_via_JSR}, $\cradius(\family)<\cgradius(\family)+\epsilon$ for any $\epsilon>0$, which concludes the proof.
\end{proof}

\begin{remark}\label{Remark_internal_eigenvectors_corresponds_to_boundedness}
        Observe that in the hypotheses of Theorem \ref{Thm_Equivalence_cone_spectral_radii_asymptotic_stability}, if $\cgradius(\family)<1$, then $\Sigma(\family)$ is equibounded. Indeed let $\psi\in \interior(\cone^*)$ and $u\in\cone'$ with $\psi(u)=1$ and define $\Xi=\{x\in \cone'\,|\; \psi(x)=1\}$\,. By the compactness of $\Xi$ and the inclusion $\cone'\subset\interior(\cone)$, with $\cone'$ closed, there exist
\begin{equation}
M=\max_{x\in \Xi}M(x/u)<\infty \qquad \mathrm{and}\qquad m=\min_{x\in \Xi}m(x/u)>0\,.
\end{equation}
Finally, by hypothesis $\cradius(f)\leq 1\;\forall f\in\Sigma(\family)$, and for any $f\in\Sigma_k$, there exists $\ceigenvector\in\Xi$, thus Corollary \ref{Corollary_x-norm} yields 
\begin{equation}
\|f\|_u\leq \frac{M}{m}\qquad \forall f\in\Sigma(\family)\,.
\end{equation}
\end{remark}

    It is thus natural to wonder if the hypothesis of a subcone that contains all the eigenvectors can be weakened by simply assuming the boundedness of $\Sigma(\family)$, indeed this is the same hypothesis needed in the linear case to have the equality between the joint spectral radius and the generalized joint spectral radius. The next example provides a negative answer to this question.

\begin{example}\label{Example_spectral_radius_and_generalized_sr_can_be_different}
As in the previous \Cref{Example_uniform_convergence}, we let $\cone=\R^2_+$ and consider the same maps $g_n$ with $t=1/2$. Consider also the linear map $Ax=(x_1/2,x_1/2+x_2)$ and consider the family of continuous, subhomogeneous and order-preserving maps $\family:=\{f_n\}_{n\in \N}$ with $f_n(x)=g_n(Ax)$. Then, since it always holds $x/2\leq x/2+y$, replacing $\min\{x_1/2,x_1/2+x_2\}$ by $x_1/2$ in the expression of $g_n$ we get to the following equality: 
\begin{equation}
\begin{aligned}
    \big(f_n(x)\big)_1=\bigg(\Big(\frac{1}{2}\Big)^{\frac{1}{n}}\Big(\frac{x_1}{2}\Big)^{\frac{1}{n}}+\Big(\frac{x_1}{2}\Big)^{\frac{1}{n}}\Big(1-\Big(\frac{1}{2}\Big)^{\frac{1}{n}}\Big) \bigg)^n=\frac{x_1}{2}\\
    \big(f_n(x)\big)_2=\bigg(\Big(\frac{1}{2}\Big)^{\frac{1}{n}}\Big(\frac{x_1}{2}+x_2\Big)^{\frac{1}{n}}+\Big(\frac{x_1}{2}\Big)^{\frac{1}{n}}\Big(1-\Big(\frac{1}{2}\Big)^{\frac{1}{n}}\Big) \bigg)^n
\end{aligned}
\end{equation}
In particular for any $n$ we can conclude that $\cradius(f_n)=\frac{1}{2}$ which is realized by taking $x=(0,1)$, indeed for any $n\in \N$, $f_n(0,x_2)=(0,x_2/2)$. Iterating it is easy to observe that for any $f\in \Sigma_k(\family)$, $\cradius(f)=1/2^k$, yielding:
\begin{equation}
    \cgradius(\family)=\frac{1}{2}
\end{equation}
On the other hand, as a consequence of 
\Cref{Example_uniform_convergence}: 
\begin{equation}
    \lim_n f_n(x)=\Big(\frac{x_1}{2},\frac{x_1}{2}+x_2\Big)\qquad \forall x \text{ s.t. } x_1\neq 0.
\end{equation}
Observe now that 
\begin{equation}
    \Big\|\Big(\frac{x_1}{2},\frac{x_1}{2}+x_2\Big)\Big\|_1=\|x\|_1.
\end{equation}
Then, given a point $x\in \interior(\cone)$ with $\|x\|_1=1$ it is not difficult to prove that for any $k\in \N$, $\sup_{f\in\Sigma_k}\|f(x)\|_1=1$. Indeed any $\tilde{f}\in \Sigma_{k-1}(\family)$ is such that 
\begin{equation}
\big(\tilde{f}(x)\big)_1=x_1/2^{k-1}  \quad \text{ and }\quad  x_1/2^{k-1}   \leq \big(\tilde{f}(x)\big)_2\leq   x_2+x_1\sum_{i=1}^{k-1} \Big(\frac{1}{2}\Big)^i.
\end{equation}
To prove it, observe that $(x_1/2,x_1/2)\cleq f_n(x)\cleq (x_1/2,x_2+x_1/2)$ for any $n\in \N$ and proceed by induction, where we have used the concavity of the function $x^{1/n}$ to derive the upper bound.
In particular, it is easy to see that for any $x\in \interior(\cone)$ $\{\tilde{f}(x)| \tilde{f}\in \Sigma_{k-1}(\family)\}$ is contained in a subset in the interior of $\cone$ that is compact with respect to the Thompson metrics, i.e. arbitrarily far from the boundary of the cone.
Thus, by Ascoli-Arzelà (see \Cref{Prop_ascoli_arzela}), any sequence of maps in $\Sigma_k(\family)$ admits a subsequence converging uniformly on compact subsets of the interior of the cone.
In particular, by \Cref{Example_uniform_convergence} we observe:
\begin{equation}
    \lim_n f_n^{(k)}(x)\rightarrow \bigg(x_1\Big(\frac{1}{2}\Big)^k , x_2+x_1\sum_{i=1}^k \Big(\frac{1}{2}\Big)^i \bigg),
\end{equation}
where $f_n^{(k)}=f_n\circ \dots \circ f_n$ $k$-times.
In particular $\sup_{f\in\Sigma_k(\family)}\|f(x)\|_1=\|x\|_1$, i.e. $\sup_{f\in\Sigma_k(\family)}\|f\|^{\frac{1}{k}}=1$ and 
\begin{equation}
    \cradius(\family)=1,
\end{equation}
concluding the proof.
    
\end{example}

Next, we investigate the condition of having a subcone $\cone'$ of $\cone$ such that $\ceigenvector\in \cone'$ for any $f\in\Sigma(\family)$. In particular, we discuss examples and investigate sufficient conditions to guarantee the existence of such an embedded subcone.

\begin{example}\label{ex_perturbed_families_satisfy_JSR=GJSR}
    We start by discussing an easy example. Let $\family=\{f_i\}_{i\in I}$ be a bounded family of continuous, homogeneous and order-preserving maps, $\psi\in \interior(\cone^*)$ and $u\in\interior(\cone)$ with $\psi(u)=1$. Then, for any $\epsilon>0$ consider the perturbed family $\family_\epsilon=\{f_{i,\epsilon}\}_{i\in I}$, where $f_{i,\epsilon}:=f_i(x)+\epsilon\psi(x)u$.
    Note that, since $\family$  is bounded and $u\in \interior(\cone)$, there exists $C>0$ such that $f_i(x)\cleq Cu$
    for any $x$ such that $\psi(x)=1$ and any $i\in I$. In particular
    \begin{equation}
       \epsilon u \cleq f_{i,\epsilon}(x)\cleq (C+\epsilon)u \qquad \forall i\in I,\; \forall x \text{ s.t } \psi(x)=1.
    \end{equation}
    Hence, since $f_{i,\epsilon}$ are homogeneous and the Hilbert distance is scale invariant, $d_H(x,u)\leq \log\big((C+\epsilon)/\epsilon)$ for all $x\in \cone\setminus
    \{0\}$, i.e. all maps in the semigroup generated by $\family$ have image, and thus dominant eigenvector, in the embedded subcone $\cone'=\{x\in \cone\,:\; \dist(x,u)\leq \log\big((C+\epsilon)/\epsilon)\}$. As a consequence of \Cref{Corollary_equality_of_generalized_and_classic_joint_spectral_radii_embedded_subcone}, $\cgradius(\family_{\epsilon})=\cradius(\family)$ for all $\epsilon>0$. 
\end{example}

\begin{theorem}\label{Thm_conctractivity_yields_invariant_subcone}
Let $\family=\{f_i\}_{i\in I}$ be a family of continuous, homogeneous, and order-preserving maps on a closed solid cone $\cone$.
Assume that $f_i$ is Hilbert contractive of parameter $\beta<1$ for any $i\in I$, i.e.
\begin{equation}
    \dist\big(f(x),f(y)\big)\leq \beta \dist\big(x,y\big) \quad \forall x\sim y\,,
\end{equation}
and that there exists a closed subcone $\cone'\subset \interior(\cone)$ such that any $f_i$, $i\in I$ admits a $\ceigenvector[f_i]\in \cone'$. Then, there exists a closed subcone $\cone'' \subset \interior(\cone)$ such that $f(x)\in \cone''$ for any $f\in \Sigma(\family)$ and $x\in \cone''$. In particular, any function $f\in \Sigma(\family)$ admits some $\ceigenvector\in \cone''$.
\end{theorem}
%%%
%%%
\begin{proof}
For any $i\in I$, let $\ceigenvector[i]:=\ceigenvector[f_i]\in \cone'$ be the eigenvector of $f_i$ in the subcone $\cone'$. 
Consider 
$f\in \Sigma(\family)$ with:
\begin{equation}
    f=f_{i_1}\comp f_{i_2}\comp \dots \comp f_{i_n}\,.
\end{equation}
Then we show that any eigenvector of $f$ in the interior of $\cone'$ is contained in a subcone $\cone''$.
Indeed given $x\in\interior(\cone)$, we claim that 
\begin{equation}
    \min_{i\in I}\dist\big(f(x),\ceigenvector[i]\big)\leq \beta^n \sup_{i\in I}\dist\big(x,\ceigenvector[i]\big)+\sup_{i,j\in I}\dist(\ceigenvector[i],\ceigenvector[j])\sum_{h=1}^{n-1} \beta^h
\end{equation}
We prove it by induction. Let $n=1$, i.e. $f=f_{i_1}$, since $\dist(x,\alpha_1 y)=\dist(x,y)$, a simple computation leads the thesis:
\begin{equation}
\dist\big(f(x),\ceigenvector[i_1]\big)=\dist\big(f(x),\cradius(f_{i_1})\ceigenvector[i_1]\big)= \dist\big(f(x),f(\ceigenvector[i_1])\big)\leq \beta \dist\big(x,\ceigenvector[i_1]\big)\,. 
\end{equation}
Next, assume the thesis holds up to $n-1$ and observe the following:
\begin{equation}
    \begin{aligned}
        \dist\big(f(x),\ceigenvector[i_1]\big)& \leq \beta \:\dist\big((f_{i_2}\comp\dots \comp f_{i_n})(x), \ceigenvector[i_1]\big)
        \\
        & \leq \beta \Big(\dist\big((f_{i_2}\comp\dots \comp f_{i_n})(x), \ceigenvector[i_2]\big)+\dist(\ceigenvector[i_1],\ceigenvector[i_2])\Big)\\
       & \leq \beta \Big( \beta^{n-1} \sup_{i\in I}\dist\big(x, \ceigenvector[i]\big)+\sup_{i,j}\dist(\ceigenvector[i],\ceigenvector[j])\sum_{h=1}^{n-2}\beta^h+\dist(\ceigenvector[i_1],\ceigenvector[i_2])\Big)\\
       & \leq \beta^{n}\sup_{i\in I}\dist\big(x, \ceigenvector[i]\big)+\sup_{i,j}\dist(\ceigenvector[i],\ceigenvector[j])\sum_{h=1}^{n-1}\beta^h\,.
    \end{aligned}
\end{equation}
In particular, we have
\begin{equation}
\begin{aligned}
     \min_{i\in I}\dist\big(f(x),\ceigenvector[i]\big)&
     \leq \beta^n \sup_{i\in I}\dist\big(x,\ceigenvector[i]\big)+\sup_{i,j\in I}\dist(\ceigenvector[i],\ceigenvector[j])\sum_{h=1}^{n-1} \beta^h\\
    &\leq \beta \sup_{i\in I}d_H(x,x_i)+\mathrm{diam}(\cone')\frac{\beta}{1-\beta}
\end{aligned}
\end{equation}
Finally, the last equation yields the following for any $f\in \Sigma(\family)$:
\begin{equation}\label{THM_family_has_invariant_cone}
\dist(f(x),\cone')\leq 
    \beta \dist(x,\cone')+\mathrm{diam}(\cone')\Big(\frac{\beta}{1-\beta}+\beta\Big).
\end{equation}
Hence, if we define $\cone''=\{x\in\cone|\dist(x,\cone')\leq\beta(2-\beta)\mathrm{diam}(\cone')/(1-\beta)^2 \}\cup \{0\}$, it is possible to observe any $f\in \Sigma(\family)$ satisfies $f(x)\in \cone''$ for all $x\in \cone''$.
We have proved that there exists a subcone $\cone''$ that is invariant for any map in the semigroup generated by $\family$. In particular, from the Schouder fixed point theorem any $f\in \Sigma(\family)$ admits an eigenvector $x_f$ in $\cone''$. 
Indeed, given $z\in \interior(\cone^*)$, the map $x\rightarrow f(x)/\langle z,f(x)\rangle$ maps the closed convex set $\{x\in \cone''| \langle z,x \rangle=1\}$ into itself, and hence admits a fixed point.
Lastly, we point out that since $x_f\in \interior(\cone)$, then the corresponding eigenvalue is the spectral radius of $f$, see \Cref{Lemma_lemmens_spectral_radius_2}, which concludes the proof. 

\end{proof}
%%%

    Observe that if the functions are not contractive, the thesis of \Cref{Thm_conctractivity_yields_invariant_subcone} may fail, as the next example shows.

\begin{example}
    Let $f_1(x)=A_1x$ and $f_2(x)=A_2x$ with 
    \begin{equation}
        A_1=\begin{pmatrix}
            0 & 2\\
            1/2 & 0
        \end{pmatrix}\qquad A_2=\begin{pmatrix}
            0 & 1/2\\
            2 & 0
        \end{pmatrix}
    \end{equation}
Then, both $f_1$ and $f_2$ admit an eigenvector relative to $\cradius(f_1),\, \cradius(f_2)$ in $\interior(\cone)$, i.e. $x_1=(2,1)$ and $x_2=(1,2)$. However, it is not hard to prove that given some generic $x\in \interior(\cone)$, $\dist(f_1(x),f_1(x_1))=\dist(x,x_1)$ and analogously for $f_2$. In particular, if we consider 
$f=f_1\circ f_2$ we have $f(x)=\Pi x$ with 
\begin{equation}
    \Pi=\begin{pmatrix}
        4 & 0 \\
        0 & 1/4
    \end{pmatrix},
\end{equation}
 that clearly has no eigenvector in the interior of the cone and does not preserve any subcone embedded in $\interior(\cone)$, indeed it maps everything closer and closer to the $x$-axis.  
\end{example}

%%%
In addition to the equivalence between the joint spectral radius and the generalized joint spectral radius, the existence of an embedded subcone $\cone'\subset \mathrm{int}(\cone)$ containing the dominant eigenvector of any map in the semigroup generated by the family $\family$ has other interesting implications. Indeed, in this case, we can also bound the joint spectral radius of $\family$ in terms of the spectral radii of the maps in $\family$ and the diameter of the subcone $\cone'$. 
\begin{proposition}
    Let $\cone$ be a solid closed cone and $\family=\{f_i\}_{i\in I}$ a family of continuous, homogeneous and order-preserving maps. Assume that there exists a closed subcone $\cone'\subset \interior(\cone)$ such that any $f_i$ admits an eigenvector $\ceigenvector(f_i)\in\cone'$ relative to $\cradius(f_i)$. Then 
    \begin{equation}
        \cgradius(\family)\leq \sup_{i\in I}\cradius(f_i)e^{\mathrm{diam}(\cone')}\,.
    \end{equation}
\end{proposition}
\begin{proof}
    Let $\psi\in\interior(\cone^*)$ and denote by $\ceigenvector[i]:=\ceigenvector[f_i]\in\cone'$ with $\psi(\ceigenvector[i])=1$, $\cradius^i:=\cradius(f_i)$ and $M_i^j:=M(\ceigenvector[i]/\ceigenvector[j])$, $m_i^j:=m(\ceigenvector[i]/\ceigenvector[j])$\,. Then, let $f=f_{i_1}\comp\dots\comp f_{i_n}\in \Sigma(\family)$. We claim that 
    \begin{equation}
        f(\ceigenvector[i_n])\cleq \bigg(\prod_{j=1}^n \cradius^{i_j}\bigg) \bigg(\prod_{j=1}^{n-1} M_{i_{j+1}}^{i_j}\bigg)\ceigenvector[i_1]\,. 
    \end{equation}
    Prove it by induction on $n$; if $n=1$, the thesis is trivial. Otherwise, by inductive step, we have:
    \begin{equation} f(\ceigenvector[i_n])=f_{i_1}\comp\big((f_{i_2}\comp\cdot\comp f_{i_n}(\ceigenvector[i_n])\big)\cleq \bigg(\prod_{j=2}^n \cradius^{i_j}\bigg) \bigg(\prod_{j=2}^{n-1} M_{i_{j+1}}^{i_j}\bigg) f_{i_1}(\ceigenvector[i_2])
    \end{equation}
    Then, since $\ceigenvector[i_2]\cleq M_{i_2}^{i_1}x_{i_1}$ and $f_{i_1}$ is order preserving:
\begin{equation}\label{eq_1_bound_joint_spectral_radius} f(\ceigenvector[i_n])\cleq \bigg(\prod_{j=2}^n \cradius^{i_j}\bigg) \bigg(\prod_{j=2}^{n-1} M_{i_{j+1}}^{i_j}\bigg) M_{i_2}^{i_1}
f_{i_1}(\ceigenvector[i_1])= \bigg(\prod_{j=1}^n \cradius^{i_j}\bigg) \bigg(\prod_{j=1}^{n-1} M_{i_{j+1}}^{i_j}\bigg) \ceigenvector[i_1]
\end{equation}
%%%
%%%
Now recall that $m_{i_n}^{i_1} \ceigenvector[i_1]\cleq \ceigenvector[i_n]$ and $m_{i_n}^{i_1}\geq \prod_{j=1}^{n-1} m_{i_{j+1}}^{i_j}$, thus replacing in \eqref{eq_1_bound_joint_spectral_radius} we obtain:
\begin{equation}\label{eq_2_bound_joint_spectral_radius} f(\ceigenvector[i_n])\cleq \bigg(\prod_{j=1}^n \cradius^{i_j}\bigg) \bigg(\prod_{j=1}^{n-1} \frac{M_{i_{j+1}}^{i_j}}{m_{i_{j+1}}^{i_j}}\bigg) \ceigenvector[i_n]
\end{equation}
The last equation \eqref{eq_2_bound_joint_spectral_radius} joint with \Cref{Lemma_comparison_eigenvalues} imply that:
\begin{equation}
    \cradius(\ceigenvector)\leq  \bigg(\prod_{j=1}^n \cradius^{i_j}\bigg) \bigg(\prod_{j=1}^{n-1} \frac{M_{i_{j+1}}^{i_j}}{m_{i_{j+1}}^{i_j}}\bigg)\,.
\end{equation}
In particular 
\begin{equation}
    \cradius^{\frac{1}{n}}(\ceigenvector)\leq \sup_{i\in I}\cradius(f_i)e^{\frac{n-1}{n}\mathrm{diam}(\cone')}
\end{equation}
where $\mathrm{diam}(\cone')=\sup_{x,y\in \cone'}\dist(x,y)=\sup_{x,y\in \cone'}\log(\frac{M(x/y)}{m(x/y)})$\,.
Thus, by definition of $\cgradius(\family)$ we derive the thesis:
\begin{equation}
    \cgradius(\family)\leq \sup_{i\in I}\cradius(f_i)e^{\mathrm{diam}(\cone')}\,.
\end{equation}
\end{proof}

 \begin{remark}
     We recall that the main scope of this article is discussing the stability of discrete systems on a cone $\cone$ that alternate subhomogeneous and order-preserving maps from a family $\family$. To this end, in \Cref{SEC:JSR and the stability} we have introduced and studied the joint spectral radius of the family $\family$. Then in \Cref{SEC:JSR_from_subhomogeneus_to_homogeneous} we have shown that it is possible to bound from above and below the joint spectral radius of the subhomogeneous family $\family$ in terms of the joint spectral radii of two homogeneous families $\family_0$ and $\family_{\infty}$, which are obtained by taking the proper limits to zero and infinity of the maps in $\family$. Lastly, in this section, we have studied the joint spectral radius of homogeneous families.  In particular, we have observed that, using the nonlinear Perron-Frobenius theory, we can study the joint spectral radius of a homogeneous family by looking at the spectral radii of the maps in the semigroup generated by the family. 
    Reconnecting to the discussion about subhomogeneous families, we point out that, given a subhomogeneous map $f$ and any $\psi\in\interior(\cone^*)$, then there exists a continuous curve of eigenpairs $c\rightarrow(x_c,\cradius^c(f))$ such that $\psi(x_c)=c$ and $\cradius^c(f)$ corresponds to the maximum eigenvalue on the slice $\Omega_c=\{x\in \cone|\psi(x)=c\}$ (see \Cref{Prop_characterization_of_subhomog_spectral_radius} and \Cref{prop_limits_of_subhomogeneous_eigenpairs}).
    In particular, under mild hypotheses, the limits of this curve to zero and infinity converge to the spectral radii and associated eigenvectors of the two limit homogeneous maps $f_0$ and $f_{\infty}$ introduced in \Cref{SEC:JSR_from_subhomogeneus_to_homogeneous}. This fact is discussed in \cite{cavalcante2019connections, nuzman2007contraction} in the case of strongly subhomogeneous maps on the positive cone of $\R^N$. We could not find anywhere in the literature a proof of this fact for subhomogeneous maps on a generic cone. For this reason, we devoted a section to discuss this property of subhomogeneous maps in the appendix; see \Cref{subsec:joint_spectral_radii_of_subhomogeneous_maps}.

\end{remark}

\subsection{About the continuity of the joint spectral radius}\label{subsec:continuity_JSR}

Finally, in this subsection, we investigate the continuity of the joint spectral radius. As a result we prove that for any family $\family$ there always exists some perturbed families $\family_{\epsilon}$ with $\mathrm{dist}(\family,\family_\epsilon)\leq \epsilon$ such that $\cradius(\family_\epsilon)\rightarrow \cradius(\family)$ for $\epsilon$ that goes to zero with $\cradius(\family_\epsilon)=\cgradius(\family_\epsilon)$ for any $\epsilon>0$.

In the linear case, it is well known that the spectral radius is continuous with respect to the family of maps \cite{jungers2009joint,heil1995continuity,wirth2005generalized}. In the nonlinear case, the problem of the continuity of the joint spectral radii is strictly connected to the problem of the equivalence between the joint and the generalized spectral radius. 
Given two families of maps, indexed on the same set $I$, $\mathcal{G}=\{g_i\}_{i\in I}$ and $\mathcal{F}=\{f_i\}_{i\in I}$, we write 
\begin{equation}
    \mathrm{dist}(\mathcal{G},\mathcal{F})=\sup_{i\in I}\sup_{x\in \cone}\frac{\|f_i(x)-g_i(x)\|}{\|x\|}.    
\end{equation}
Moreover we say that $\{\family_n\}_{n\in \N}$ converges to $\family$ from above if $\mathrm{dist}(\family_n,\family
    )\rightarrow 0$ and $f_i^{(n)}(x)\geq f_i(x)$ for any $i\in I$, $x\in\cone$ and $n\in \N$. In particular, we introduce the following definition.

\begin{definition}\label{Definition_continuity_spectral_radius}
     The spectral radius $\cradius(\cdot)$ and the generalized spectral radius $\cgradius(\cdot)$ are continuous (from above) in $\family$ if for any $\{\family_n\}_{n\in \N}$ converging to $\family$ (from above),  $\cradius(\family_n)$ and $\cgradius(\family_n)$ converge to $\cradius(\family)$ and $\cgradius(\family)$ respectively.
\end{definition}

We then start by studying the continuity of the joint spectral radius.

\begin{theorem}\label{Theorem_continuity_of_spectral_radius}
    Let $\family=\{f_i\}_{i\in I}$ be a bounded and equicontinuous family of continuous, homogeneous, and order-preserving maps on a closed cone $\cone$. Then $\cradius(\family)$ is continuous from above, i.e. for any $\family_n:=\{f_i^{(n)}\}_{i\in I}$ converging to $\family$ from above $\cradius(\family_n)$ converges to $\cradius(\family)$. Moreover, $\cgradius(\family)$ is lower semicontinuous. 
\end{theorem}
\begin{proof}
    Assume that $\family_n$ converges to $\family$ from above and note that since $\family$ is bounded and equicontinuous then $\Sigma_k(\family_n)$ converges to $\Sigma_k(\family)$ for any $k$. Moreover $f_i^{(n)}(x) \cgeq f_i(x)$ for any $x\in \cone$,  $n\in \N$ and $i\in I$ and since any $f_i^{(n)}$ is order preserving, then also $f_{i_1}^{(n)}\circ\dots \circ f_{i_k}^{(n)} (x)\cgeq f_{i_1}\circ\dots \circ f_{i_k} (x)$. This shows that also $\Sigma_k(\family_n)$ converges to $\Sigma_k(\family)$ from above for any $k\in \N$.
Now, remember that $\cradius(\family)=\lim_k\sup_{f\in\Sigma_k(\family)}\big\|f\big\|^{\frac{1}{k}}=\inf_k\sup_{f\in\Sigma_k(\family)}\big\|f\big\|^{\frac{1}{k}}$.
Then, for any $\epsilon>0$ there exists $K$ such that for any $k\geq K$,  $\sup_{f\in \Sigma_k(\family)}\|f\|^{\frac{1}{k}}\leq \cradius(\family)+\epsilon$. 
On the other hand, fixed the index $K$ as above, there exists $N$ such that for any $n\geq N$, and $f'\in \Sigma_K(\family_n)$, there exists some $f\in \Sigma_K(\family)$ such that  $|\|f\|^{\frac{1}{K}}-\|f'\|^{\frac{1}{K}}|\leq \epsilon$.
In particular, for any $n\geq N$:
\begin{equation}
\cradius(\family_n)=\inf_k\sup_{f\in\Sigma_k(\family_n)}\|f\|^{\frac{1}{k}}\leq \sup_{f\in \Sigma_K(\family_n)}\|f\|^{\frac{1}{K}}\leq \sup_{f\in \Sigma_K(\family)}(\|f\|)^{\frac{1}{K}}+\epsilon\leq \cradius(\family)+2\epsilon.
\end{equation}
This proves that $\cradius(\cdot)$ is always upper semicontinuous.
Next assume w.l.o.g that $\|\cdot\|$ is an order preserving norm, then we trivially observe that if $f^{(n)}\in \Sigma_k(\family_n)$ converges to $f\in \Sigma_k(\family)$, $\|f^{(n)}\|\geq \|f\|$.
In particular, since $\Sigma_k(\family_n)$ converges to $\Sigma_k(\family)$ for any $k$, 
\begin{equation}
    \sup_{f^{(n)}\in \Sigma_k(\family_n)}\|f^{(n)}\|^{\frac{1}{k}}\geq \sup_{f\in \Sigma_k(\family)}\|f\|^{\frac{1}{k}}. 
\end{equation}
I.e. $\cradius(\family_n)\geq \cradius(\family)$ for any $n$, i.e. the lower semicontinuity.
Next we consider the generalized joint spectral radius. 
We start by recalling that $\cgradius(\family)=\limsup_k \sup_{f\in\Sigma_k(\family)} \cradius(f)^{\frac{1}{k}}=\sup_k \sup_{f\in \Sigma_k(\family)} \cradius(f)^{\frac{1}{k}}$. 
Thus, for any $\epsilon>0$ there exist $K\in \N$ and $\bar{f}\in \Sigma_K(\family)$ such that $\cradius(f)^{\frac{1}{K}}\geq \cgradius(\family)-\epsilon$.
Since $\Sigma_K(\family_n)$ converges to $\Sigma_K(\family)$ from above, from \Cref{Continuity_of_spectral_radius_from_above} there exists $M$ such that for any $n\geq M$ there exists $\bar{f}^{(n)}\in \Sigma_K(\family_n)$ such that $\cradius(\bar{f}^{(n)})^{\frac{1}{k}}\geq \cradius(\bar{f})^{\frac{1}{k}}-\epsilon$, i.e.:
\begin{equation}
 \cgradius(\family_n)=\sup_{k}\sup_{f\in \Sigma_k(\family_n)}\cradius(f)^{\frac{1}{k}}\geq    \cradius(\bar{f}^{(n)})^{\frac{1}{k}}\geq \cradius(\bar{f})^{\frac{1}{k}}-\epsilon\geq \cgradius(\family)-2\epsilon. 
\end{equation}
This proves that $\cgradius(\family)$ is lower semicontinuous. 
\end{proof}
As a consequence of \Cref{Theorem_continuity_of_spectral_radius}, we can approximate the joint spectral radius of any bounded and equicontinuous family, say $\family$, of continuous, homogeneous and order-preserving maps by the generalized joint spectral radius of a sequence of continuous, homogeneous and order-preserving families converging to $\family$ from above.
\begin{corollary}
    Let $\family:=\{f_i\}_{i\in I}$ be a bounded and equicontinuous family of continuous, homogeneous and order-preserving maps on a closed cone $\cone$. Then there exists a sequence $\family_n$ of continuous, homogeneous and order-preserving families converging to $\family$ from above such that $\cgradius(\family_n)\to \cradius(\family)$ as $n\to\infty$. 
\end{corollary}
\begin{proof}
    For any $n\in \N$ consider the family of continuous, homogeneous and order-preserving maps:
\begin{equation}
    \family_n:=\{f_i^{(n)}\}_{i\in I} \qquad \text{with} \qquad f_i^{(n)}(x)=f_i(x)+\frac{\phi(x)}{n}u,
\end{equation}
with $u\in \interior(\cone)$ and $\phi(\cdot)\in \interior(\cone^*)$.
It is easy to observe the $\{\family_n\}_{n\in \N}$ converge to $\family$ from above, yielding as a consequence of \Cref{Theorem_continuity_of_spectral_radius} 
\begin{equation}
    \lim_{n\rightarrow \infty}\cradius(\family_n)=\cradius(\family).
\end{equation} 
Moreover, from \Cref{ex_perturbed_families_satisfy_JSR=GJSR} we know that $\cgradius(\family_n)=\cradius(\family_n)$ for any $n\in \N$, concluding the proof. \end{proof}

Finally we point out the following equivalence between the continuity from above of the generalized joint spectral radius and the equality between $\cradius(\family)$ and $\cgradius(\family)$.
\begin{corollary}\label{cor_continuity_of_spectral_radius}
    Let $\mathcal{F}=\{f_i\}_{i\in I}$ be a bounded and equicontinuous family of continuous, homogeneous and order-preserving maps on a closed cone $\cone$. Then $\cgradius(\cdot)$ is continuous from above in $\family$ if and only if $\cgradius(\family)=\cradius(\family)$. 
\end{corollary}
\begin{proof}
Under our hypotheses, from \Cref{Theorem_continuity_of_spectral_radius} we know that $\cradius(\family)$ is continuous and $\cgradius(\family)$ is lower semicontinuous. Thus if we assume $\cgradius(\family)=\cradius(\family)$ and let $\family_n$ converge to $\family$ from above, we have
\begin{equation}
\cgradius(\family)\leq \lim_n\cgradius(\family_n)\leq \lim_n\cradius(\family_n) = \cradius(\family)=\cgradius(\family),
\end{equation}
i.e., $\cgradius(\family)$ is continuous from above.
Now assume that, besides $\cradius(\family)$, also $\cgradius(\family)$ is not only semicontinuous but actually continuous in $\family$ with respect to families converging to it from above.
Then, for any $n\in \N$ consider the family of continuous, homogeneous and order-preserving maps:
\begin{equation}
    f_i^{(n)}(x)=f_i(x)+\frac{\phi(x)}{n}u
\end{equation}
    with $u\in \interior(\cone)$ and $\phi(\cdot)\in \interior(\cone^*)$. Then $\family_n=\{f_i^{(n)}\}_{i\in I}$ converge to $\family$ from above. Next we recall from \Cref{ex_perturbed_families_satisfy_JSR=GJSR}, that $\cradius(\family_n)=\cgradius(\family_n)$ for any $n\in \N$ which, by the hypothesis of continuity in $\family$ yields $\cradius(\family)=\cgradius(\family)$. 
\end{proof}

\section{Polytopal Algorithm}\label{sec:polytopal_algorithm}

In this section, we discuss a numerical method to compute or bound the joint spectral radius of a family $\family$ of continuous, homogeneous, and order-preserving maps. The resulting method is inspired by the polytopal algorithm discussed in \cite{guglielmi2005complex, guglielmi2013exact} and used to compute the joint spectral radius of a family of matrices. The main goal is to explicitly compute an extremal monotone prenorm for the family $\family$, or approximate it. To this end, we use the different reformulations of the joint spectral radius discussed in \Cref{sec:duality_monotone_prenorms} and \Cref{Sec:gen_joint_spectral_radius}.

First, we introduce the class of finitely generated monotone prenorms that will be used in the algorithm. % in the following.

\begin{definition}\label{Def_finite_prenorm}
  We say that $\Theta\in MPN(\cone)$ is a finitely generated prenorm if there exist $\{z_i\in \interior(\cone)\}_{i=1}^m$ such that 
$$\Theta(x)=\min_{i=1,\dots,m}M(x/z_i)$$
\end{definition}
Any $\Theta$ as above is easily proved to be a continuous, nondegenerate, monotone prenorm. In particular, $\Theta$ corresponds to the Minkowski functional of the subset of the cone that includes all the points dominated some point $z_i$, $i=1,\dots,N$. Moreover, note that without loss of generality we can assume $M(z_i/z_j)>1$ for any $i,j$, indeed a point $z_i$ dominated by another point $z_j$ can be neglected without changing the induced prenorm.
In the following discussion we show that the finite prenorms play, in the nonlinear case, the same role played by polytopal norms in the linear case \cite{guglielmi2005complex,guglielmi2013exact}. Indeed, the existence of an extremal finite prenorm guarantees the exact computation of the joint spectral radius of a family of continuous, homogeneous and order-preserving maps.
\begin{lemma}\label{Lemma_finite_prenorm_gives_spectral_radius}
    Let $\family=\{f_i\}_{i\in I}$ be a family of continuous, homogeneous and order-preserving maps on a cone $\cone$. Assume that there exist a finite set of points $\{z_j\}_{j=1}^m\in \interior(\cone)$, with $M(z_j/z_h)\geq 1$ for all $j,h$, such that
    \begin{enumerate}
        \item $\min_{h=1,\dots,m} M\big(f_i(z_j)/z_h\big)\leq \alpha \qquad \forall i\in I,\;j=1,\dots,m$
        \item $\exists f\in \Sigma_k(\family)$ and $j_0$ such that $f(z_{j_0})=\alpha^k z_{j_0}$
    \end{enumerate}
    Then $\alpha=\cradius(\family)$ and the prenorm finitely generated by $\{z_j\}_{j=1}^m$ is extremal for $\family$.
\end{lemma}
\begin{proof}
    Let $\Theta$ be the prenorm induced by the $\{z_j\}_j$ and assume $\Theta(x)= 1$, then there exists some $z_j$ with $M(x/z_j)= 1$. Given $f_i\in \family$, by hypothesis we know that $\Theta(f_i(z_j))\leq \alpha$ and $M(f_i(x)/f_i(z_j))\leq 1$ because $f_i$ is order preserving. In particualar, for any $i\in I$ we have: 
    \begin{equation}
    \Theta\big(f_i(x)\big)\leq M\big(f_i(x)/f_i(z_j)\big)\Theta\big(f_i(z_j)\big)\leq \alpha 
    \end{equation}
 i.e. $\Theta(f_i)\leq \alpha$ for all $i\in I$, which in particular yields $\cradius(\family)\leq \alpha$.
 On the other hand, by the second hypothesis there exists $f\in \Sigma_k(\family)$ such that $\|f^t\|^{\frac{1}{tk}}\geq \alpha$ $\forall t\in \N$, yielding $\cradius(\family)\geq \alpha$ and concluding the proof.
    \end{proof}
Considered the last \Cref{Lemma_finite_prenorm_gives_spectral_radius}, we can now investigate when the hypothesis of the Theorem are satisfied, in particular we provide sufficient conditions similar to the ones discussed in \cite{guglielmi2013exact} for the linear case. We restrict our study to finite families of maps $\family=\{f_i\}_{i=1}^N$ such that $f_i(\cone)\subset \interior(\cone)$ for any $i\in I$.
Observe that any map $f$ mapping the cone in its interior admits necessarily an eigenfunction in the interior of the cone.
Moreover, any $f_i$ maps compact sets in compact sets and it is $1$-homogeneous, thus there exists 
some closed cone $\cone_i\subset\interior(\cone)$ such that $f_i:\cone\rightarrow \cone_i$. In particular there exists a subcone $\cone'\subset\interior(\cone)$ mapped into itself by any map $f\in \Sigma(\family)$, i.e. any $f\in \Sigma(\family)$ admits an eigenvector in  
$\cone'$. This property yields $\cradius(\family)=\cgradius(\family)$ because of \Cref{Corollary_equality_of_generalized_and_classic_joint_spectral_radii_embedded_subcone}.  
Moreover, we look for families that admit a dominant spectrum maximizing map i.e.:
\begin{definition}
 A family $\family$ admits a dominant spectrum maximizing map if there exists $f^*\in \Sigma_k(\family)$ for some $k\in \N$ such that 
 \begin{enumerate}
     \item $\cradius(f^*)=\cradius(\family)^\frac{1}{k}$ .
     \item Exists $\epsilon>0$ such that for any $h\in \N$ and $f\in \Sigma_h(\family)$ that is not a cyclic permutation of $f^*$ or one of their power, it holds $\cradius(f)\leq (1-\epsilon) \big(\cradius(\family)\big)^h$. 
 \end{enumerate} 
where if $f^*=\Pi_{j=i_1}^{i_k}f_{j}$, the cyclic permutations of $f^*$ are the maps $f=\Pi_{j=1}^N f_{i_\sigma(j)}$ varying $\sigma$ among the cyclic permutations of $\{1,\dots,k\}$.
\end{definition}
We can prove that any finite family admitting a dominant spectrum maximizing map admits also a finite extremal prenorm.
\begin{theorem}\label{Thm_Finite_convergence_of_algorithm}
    Let $\family=\{f_i\}_{i=1}^N$ be a finite family of continuous, homogeneous and order-preserving maps on a cone $\cone$. Assume that any $f_i$ admits an internal dominant eigenvector, exists $\beta<1$ such that any $f_i:\interior(\cone)\rightarrow \interior(\cone)$ is contractive of parameter $\beta$ and exists a dominant spectrum maximizing product $f^*$.
    Then, $\family$ admits a finite extremal prenorm whose vertices are a subset of $\{f(x^*)\}_{f\in\Sigma(f')}$ where $\family'=\{f_i/\cradius(\family)\}_{i=1}^N$ and $x^*$ is the dominant eigenvector of $f^*$.
\end{theorem}
\begin{proof}
First of all note that by the contractivity property and  \Cref{Thm_conctractivity_yields_invariant_subcone} there exists a subcone $\cone'\subset \interior(\cone)$ which is invariant for any $f\in \Sigma(\family)$. Moreover, by the contractivity any $f\in \Sigma(\family)$ necessarily has a unique dominant eigenvector in $\cone'$.
    Then, assume without loss of generality that $\cradius(\family)=1$ and that  $f^*=f^*_{i_L}\circ\dots \circ f_{i_1}^*$ is the spectrum maximizing product with $x^*\in \interior(\cone)$ the dominant eigenvector of $f^*$. 

    Before going on with the proof we enunciate three claims whose proof is postponed to the end of the proof.

\paragraph{Claim 1} %%
Given $f^*$ the spectrum maximizing product with $f^*=f^*_{i_L}\circ\dots \circ f_{i_1}^*$, any cyclic permutation of $f^*$,  $f_h^*=f_{i_h}^*\circ\dots\circ f_{i_1}^*\circ f_{i_L}^*\circ\dots \circ f_{i_{h+1}^*}$ with $L<k$ has unique dominant eigenvector in the interior of the cone:  
    \begin{equation}
x_{h}^*=f_{i_h}^*\circ \dots \circ f_{i_1}^*(x^*). 
\end{equation}
\paragraph{Claim 2} Assume that there exists $y\in \interior(\cone)$ and $g_n\in \Sigma(\family)$ such that $g_n(y)\rightarrow y$ for $n$ going to infinity, then $y=\alpha x_{h}^*$ for some $0\leq h\leq L-1$ and $\alpha\geq 0$ 
\paragraph{Claim 3} Let $x^*=x^*_0\in \interior(\cone)$ be the dominant eigenvector of the sepctrum maximizing product $f^*$ and, for any $h=1,\dots,L-1$, let $x_h^*$ be the dominant eigenvectors of the permutations of $f^*$ defined as in Claim 1. Then, if there exist a sequence $g_{n_k}\in \Sigma_{n_k}(\family)$ such that $g_{n_k}(x^*)\rightarrow \alpha x_h^*$ for $n_k$ going to infinity, then $\alpha\leq 1$. If moreover, $\min_{g\in \cup_{h=1}^{n_k-1}\Sigma_h(\family')}M\big(g_{n_k}(x^*)/g(x^*)\big)> 1$, then $\alpha=1$. \\ 
   For any $g\in \Sigma(\family)$ let $y_g=g(x^*)$. 
    Because of \Cref{Lemma_finite_prenorm_gives_spectral_radius}, the thesis holds if exists $k>0$ such that for any $g'\in \Sigma_k(\family)$ exists some $g''\in \cup_{j=1}^{k-1}\Sigma_j(\family)$ with  $M(y_{g'}/y_{g''})\leq 1$.
     So, assume by contradiction that this is not true. Then, there exists a sequence $f_{n_k}\in \family$ and $y_k=f_{n_k}(y_{n_k-1}) \in \Sigma_{k}(\family)$ such that 
     \begin{equation}\label{eq:1_extremal_polytope}
         \min_{g\in \cup_{j=1}^{k-1}\Sigma_j(\family')} M(y_{k}/ y_g)>1 \qquad \forall k.
     \end{equation} 
     where $y_0=x^*$. We write $y_k:=y_{g_k}=g_k(x^*)$ where $g_k=f_{n_k}\circ \dots f_{n_1}$.
   Now, note that from \Cref{Remark_internal_eigenvectors_corresponds_to_boundedness}, $\{y_k\}$ is bounded and thus admits some convergent subsequence:
   \begin{equation}
       y_{n_k}=\tilde{f}_{n_k}(y_{n_k-1})=g_{n_k}(x^*)\longrightarrow y^*,
   \end{equation}
   where $\tilde{f}_{n_k}\in \Sigma(\family)$.
Since $\family$ admits an invariant subcone containing $x^*$, $y^*$ is in the interior of the cone and for any $\lambda<1$ there exists $M_{\lambda}$ sufficiently large such that 
\begin{equation}
    \frac{1}{\lambda}y^*\cgeq y_{n_{k}}\cgeq \lambda y^*  \qquad \forall n_k> M_{\lambda}.
\end{equation}
Last equation, in particular yields 
\begin{equation}
\begin{aligned}
     & \tilde{f}_{n_k}(y^*)\cgeq \lambda \tilde{f}_{n_k}(y_{n_k-1}) = \lambda y_{n_k} \cgeq \lambda^2 y^* \\
    & \tilde{f}_{n_k}(y^*)\cleq \frac{1}{\lambda} \tilde{f}_{n_k}(y_{n_k-1}) = \frac{1}{\lambda} y_{n_k} \cleq \frac{1}{\lambda^2} y^* 
\end{aligned}
\end{equation}
for all $n_k>M_\lambda+1$ sufficiently large.
   In particular $\tilde{f}_{n_k}(y^*)$ converges to $y^*$ and Claims 2 and 3 lead to the equality $y^*=x_{h}^*$ for some $0<h<L-1$.
In addition note that from Claim 1, for any $1\leq h\leq L$ there exists $\bar{g}_h\in \Sigma(\family)$ such that $\bar{g}_h(x_h^*)=x^*$ and similarly we can write $y_{n_k+m}=\hat{g}_{n_k}^m(y_{n_k})$ for some $\hat{g}_{n_k}^m\in \Sigma(\family)$.  In particular,  for any $n_k$, we obtain $\bar{g}_h\circ\hat{g}_{n_k}^m(y_{n_k})\rightarrow x^*$ for $m\rightarrow \infty$. 
Hence, composing with $g_{n_k}$ leads to 
\begin{equation}
g_{n_k}\circ \bar{g}_h\circ\hat{g}_{n_k}^m(y_{n_k})\longrightarrow_{m\rightarrow \infty} y_{n_k},
\end{equation}
which, as a consequence of Claim 2, proves that $y_{n_k}=\alpha_k x_{h_k}^*$ for some $h_k\in\{0,\dots,L-1\}$ and $\alpha_k\geq 0$. Thus, $g_{n_k}(x^*)=\alpha_k x_{h_k}$ for any $n_k$. In particular, the same argument of Claim 3 joint with \eqref{eq:1_extremal_polytope} yield $\alpha_k=1$ for any $n_k$ larger than $L$. 
Finally, since the cardinality of the family $\{x_h^*\}$ is finite, so should be also the number of elements in the sequence $\{y_{n_k}\}$ which however has no repetitions. We found a contradiction and thus we can conclude the proof.
\end{proof}
\begin{proof}[Proof of Claim 1]
   A direct computation proves that $x_{h}^*$ is an eigenvector of $f_h^*$ and the contractivity yields the uniqueness.
\end{proof}
\begin{proof}[Proof of Claim 2]
    To prove the claim observe that, since $y\in \interior(\cone)$, for any $\lambda\in(0,1)$ we have 
\begin{equation}
    \lambda y\cleq g_n(y) \cleq \frac{1}{\lambda} y
\end{equation}
for any $n$ sufficiently large. The last inequality implies $\cradius(g_n)\rightarrow 1$. In particular, for $n$ sufficiently large $\cradius(g_n)\geq 1-\epsilon$ but then necessarily for any $n$ sufficiently large there exists a permutation of $\{1,\dots,k\}$ and some $m_n\in \N$ such that 
\begin{equation}
    g_n=\big( f_h^* \big)^{m_n} \quad 0\leq h\leq L-1.
\end{equation} 
Now let $x_{h}^*$ be the dominant eigenvector of $f_{h}^*$, then since $\dist(x,(f^*_h)^m(x_h^*))\cleq \beta \dist(x,x_h^*)$, a simple application of the triangular inequality yields:
\begin{equation}
    \dist((f_h^*)^m(x),x)\geq (1-\beta)\dist(x,x_{h}^*)
\end{equation}
for any $m>1$ and any $x\in \interior(\cone)$. In particular, since $\dist\big((f_{\sigma_i}^*)^m_i(y),y\big)\rightarrow 0$, $y$ has to be a fixed point of some permuted spectrum maximizing product $f_{h}^*$, i.e. $y=\alpha x_{h}^*$ for some $\alpha\geq 0$.
\end{proof}

\begin{proof}[Proof of Claim 3]
Note that, because of Claim 1, there exists a map $\bar{g}_h\in \Sigma(\family)$ such that $\bar{g}_h(x_h^*)=x^*$. 
Thus, $\bar{g}_h\circ g_{n_k}(x*)$ converges to $\alpha x^*$. 
In particular, since $x^*\in\interior(\cone)$, for any $\delta>0$ and $n_k$ sufficiently large $\bar{g}_h\circ g_{n_k}(x^*)\cgeq (\alpha-\delta)x^*$. 
This proves that $\bar{g}_h\circ g_{n_k}$ has spectral radius greater than $\alpha-\delta$, which in turn implies that $\alpha\leq 1$.
Otherwise for $k$
sufficiently large $\bar{g}_h \circ g_{n_k}$ had spectral radius larger than $1$, contradicting the hypotheses.

To prove the second part note that, since $x_h^*\in \interior(\cone)$, for any $\epsilon>0$ it holds $g_{n_k}(x^*)\cleq (\alpha+\epsilon) x_h^* $ for any $n_k$ sufficiently large, yielding $\alpha\geq 1$.
Indeed $\{x_{h}^*\}_{h=1}^{L-1}\subset \{g(x^*)\,|\,g \in \cup_{h=1}^{L-1}\family_h \}$, and so if $\alpha<1$ then  $\min_{g\in \cup_{h=1}^{n_k-1}\family_h} M(g_{n_k}(x^*)/g(x^*))<1$ for $n_k$ sufficiently large which contradicts the hypothesis.
\end{proof}

In particular the discussion above proves that the polytopal algorithm aimed at computing the linear joint spectral radius can be generalized to the nonlinear case. Before discussing the pseudo algorithm we provide a simple criteria to choose or update a conjectured spectrum maximizing product.

\begin{proposition}\label{prop_Finite_convergence_of_algorithm_spectrum_max_update}
    Let $\family=\{f_i\}_{i=1}^N$ be a finite family of continuous, homogeneous and order-preserving maps on a cone $\cone$. Assume that $\family$ admits a dominant spectrum maximizing product $f^*$ with internal dominant eigenvector. If $c< \cradius(\family)$ and $\family'=\{f_i/c\}$, then for any $x\in \interior(\cone)$ there exists $f\in \Sigma(\family')$ such that $m(f(x)/x)> 1$.    
\end{proposition}
\begin{proof}
    Let $f^*$ be the spectrum maximizing product of the family $\family'$ and let $x^*$ be its dominant eigenvector and $\cradius(f^*)>1$ its sepctral radius. Then it is trivial to observe that
    \begin{equation}
       (f^*)^k(x)\cgeq m(x/x^*)m(x^*/x)\cradius^k(f^*)x,
    \end{equation}
    yielding the thesis for some $k$ sufficiently large.
\end{proof}

In particular we can consider the following pseudo code of the polytope algorithm.

\begin{algorithm}
    \caption{}\label{your_label}
    \begin{algorithmic}
        \STATE Initialize $\family=\{f_i\}_{i=1}^N$ family of continuous, homogeneous and order-preserving maps
        \STATE  Set $f^*=f\in \Sigma_k(\family)$ hypothetical spectrum maximizing map
        \STATE {\textbf{do*}}
        \STATE  Compute $(\cradius(f^*), x_{f^*})$
        \STATE  Define ${\family}'=\{f_i':=f_i/\big(\cradius(f)\big)^{\frac{1}{k}}\}$
        \STATE Define $\mathcal{V}=\mathcal{V}_0=\{x_f\}$
        \WHILE{$\mathcal{V}\neq \emptyset$}
        \STATE $\mathcal{V}_{\text{new}}= \emptyset$
        \FOR{$\xi_0\in \mathcal{V}$}
            \FOR{$i=1,\dots,N$}
                \IF{$ m\big(f_i'(\xi_0)/\xi_0)>1$}    
                   \STATE Set $f^*=f_i$ and restart from {\textbf{do*}}
                \ENDIF
                \IF{$\min_{\xi\in \mathcal{V}_0}M\big(f_i'(\xi_0)/\xi)>1$}
                   \STATE $\mathcal{V}_{\text{new}}=\mathcal{V}_{\text{new}}\cup \xi_0$
                   \STATE $\mathcal{V}_{0}=\mathcal{V}_{0}\cup \xi_0$
                \ENDIF
            \ENDFOR
        \ENDFOR
        \STATE $\mathcal{V}=\mathcal{V}_{\text{new}}$
        \ENDWHILE
    \end{algorithmic}
\end{algorithm}

The algorithm ends in a finite number of steps if after a finite number of iterations 
\begin{equation}
    \max_{\substack{\xi_0\in \mathcal{V}_0\\i=1,\dots,N}}\min_{\xi\in \mathcal{V}_0}M\big(f_i'(\xi_0),\xi)\leq 1.
\end{equation}
We recall that the spectral radius and the corresponding eigenvector of a continuous, homogeneous and order-preserving map, generally can be computed very efficiently using the power method. We refer to \cite{lemmens2012nonlinear} for a discussion about the convergence of the power method in the nonlinear case.

In particular, because of \Cref{Lemma_finite_prenorm_gives_spectral_radius}, the algorithm ending in a finite number of steps is a sufficient condition to have $\cradius(\family)=\cradius(f)$.
Moreovoer, \Cref{Thm_Finite_convergence_of_algorithm}, \Cref{prop_Finite_convergence_of_algorithm_spectrum_max_update} and \Cref{Lemma_lemmens_spectral_radius_2} provide sufficient conditions on the family $\family$ to have convergence of the algorithm in a finite time.
We point out that in the linear case the polytopal algorithm creates an extremal norm for the family $\family$ such that the corresponding unit ball is a finite politope. In the nonlinear case, instead, the algortihm generates an extremal prenorm whose unit ball is the union of a finte number of subcones, where a subcone controlled by a point $x\in\cone$ is given by the set $\{y\in \cone| y\cleq x\}$.
In particular, such unit ball is not a polytope if the cone is not polyhedral.

below we present some numerical tests about the convergence of the algorithm.
\begin{example}
In the first example we consider a family composed by two maps $\family=\{f_1,f_2\}$ defined by:
\begin{equation}
f_i(x)=\Big(B_i*\big((A_i*x)^{\alpha_i}\big)\Big)^\frac{1}{\alpha_i} \quad i=1,2.
\end{equation}
In particular we set:
\begin{equation}
    A_1=\begin{pmatrix}
   4/5 & 1/10 \\ 1/10 & 4/5
\end{pmatrix} \quad A_2=\begin{pmatrix}
   4/5 & 0 \\ 1/10 & 4/5
\end{pmatrix}
\end{equation}
\begin{equation}
B_1=\begin{pmatrix}
   4/5 & 1/5 \\ 0 & 4/5
\end{pmatrix} \quad B_2=\begin{pmatrix}
   9/10 & 0 \\ 4/5 & 7/10
\end{pmatrix}    
\end{equation}
and $\alpha_1=3/10$, $\alpha_2=1/2$.
We consider $f=f_1\circ f_2$ as the conjectured spectrum maximizing map and run the algorithm. The algorithm ends after 4 iterations proving that $\cradius(f)^\frac{1}{2}=\cradius(\family)$. In \Cref{Fig-Algorithm-Illustrative-Example} we provide the illustrative explanation of the steps performed by the algorithm.

\begin{figure}[h!]
\begin{center}
\begin{minipage}{.4\textwidth}
\includegraphics[width=\textwidth]{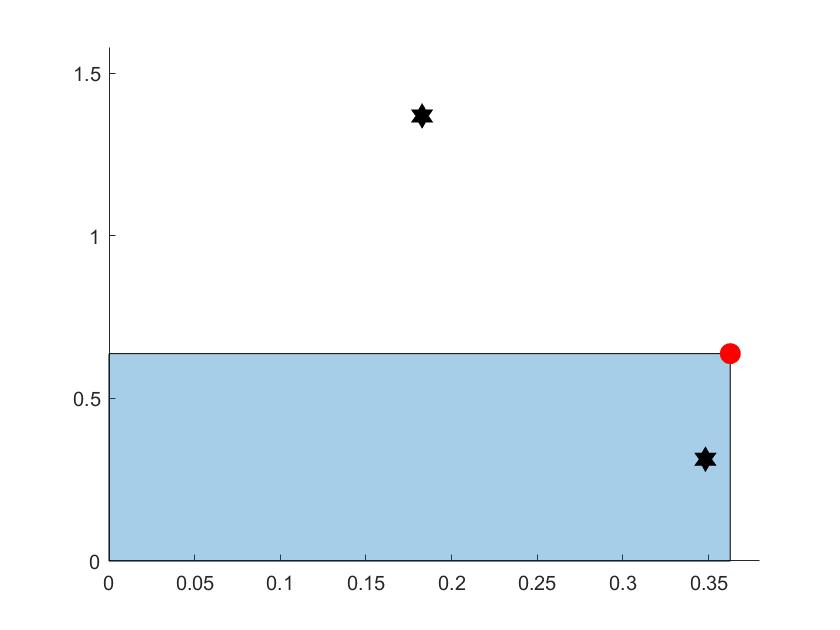}
\includegraphics[width=\textwidth]{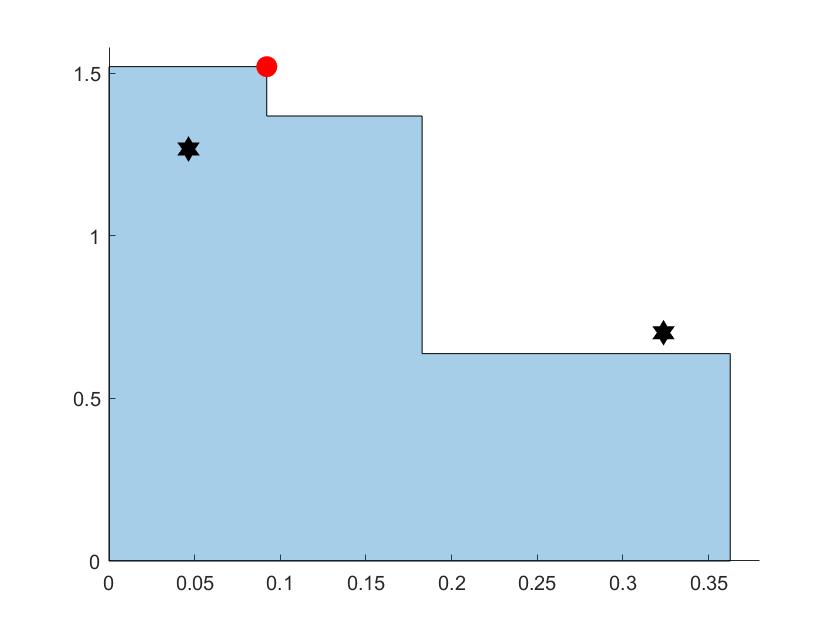}
\end{minipage}
\qquad
\begin{minipage}{.4\textwidth}
\includegraphics[width=\textwidth]{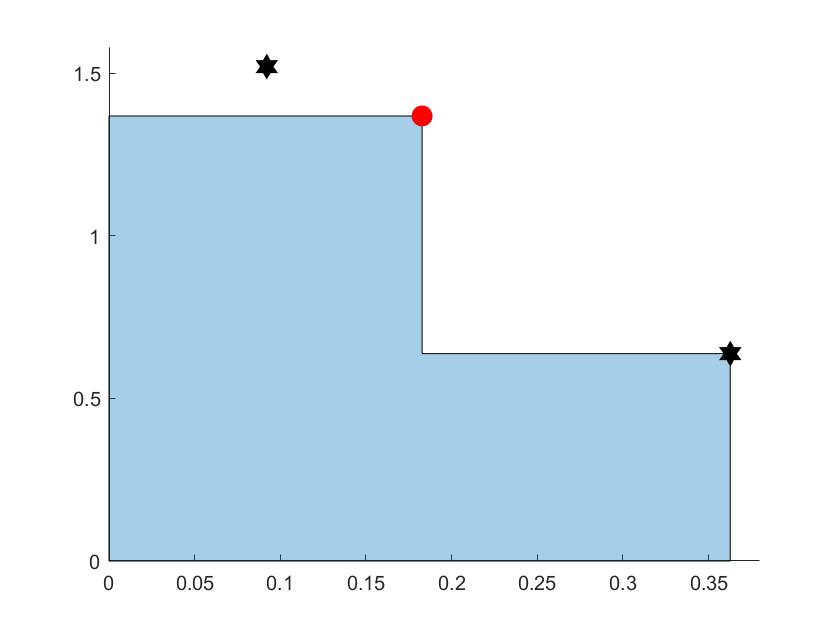}
\includegraphics[width=\textwidth]{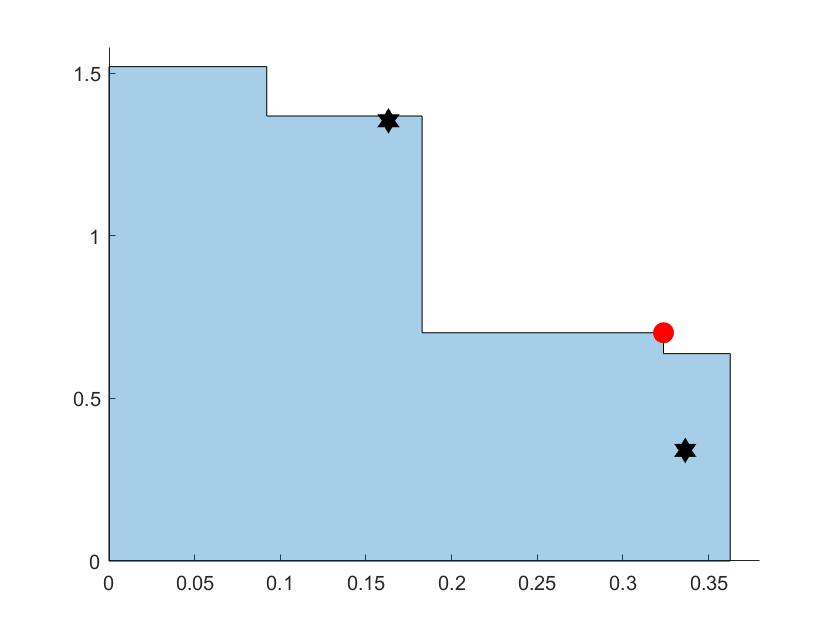}
\end{minipage}
\end{center}
\caption{Left to right, top to bottom, the algorithm starts with $\mathcal{V}=\{x_f\}$, the dominant eigenvector of $f$. At each step, the algorithm selects a vertex $x$ from $\mathcal{V}$ (the red point in each figure) and computes $f_1'(x)$ and $f_2'(x)$ (the two black star points in each figure). A star point is added to the set of vertices $\mathcal{V}_0$, if it is not dominated by some other vertex in $\mathcal{V}_0$. After 4 iterations the images, through $f_1'$ and $f_2'$, of all the vertices in $\mathcal{V}_0$ are dominated by some other vertex in $\mathcal{V}_0$ and the algorithm stops.} \label{Fig-Algorithm-Illustrative-Example}
\end{figure} 

\end{example}

\begin{example}
    In this second example we want to highlight the differences from the politopal algorithm proposed in \cite{guglielmi2001asymptotic,guglielmi2013exact} to compute the linear joint spectral radius. Indeed, our algorithm is clearly suited to compute also the joint spectral radius of a family of matrices that preserve a cone ($\R^2_+$ in this particular example).
    We consider the same example proposed in \cite{guglielmi2013exact} and analyze the differences of our algorithm from the polytopal one.
    The family $\family$ is given by two matrices:
    \begin{equation}
        A_1=\begin{pmatrix} 1 & 1 \\
                            0 & 1       
        \end{pmatrix},
        \qquad 
        A_2=\begin{pmatrix} 9/10 & 0 \\
                            9/10 & 9/10
            \end{pmatrix},
    \end{equation}
    and $\Pi=A_1 A_2$ is the spectrum maximizing product (our spectrum maximizing map). In \Cref{Fig-linear_vs_nonlinear-Example} we report the extremal polytopal norm and extremal prenorm generated by the polytopal algorithm and our algorithm. The vector $v_1$ is the eigenfunction of $\Pi$ relative to $\cradius(\Pi)$. Note that in the linear case one step of the algorithm is sufficient to generate all the vertices of polyhedral norm, i.e. $v_2=A_2'v_1$ and $v_3=A_1'v_1$.
    In the nonlinear case, instead, one more step is required to add the extremal point $v_3=A_2 A_1(x)$ and terminate the algorithm.

\begin{figure}[h!]
\begin{center}
\includegraphics[width=.4\textwidth]{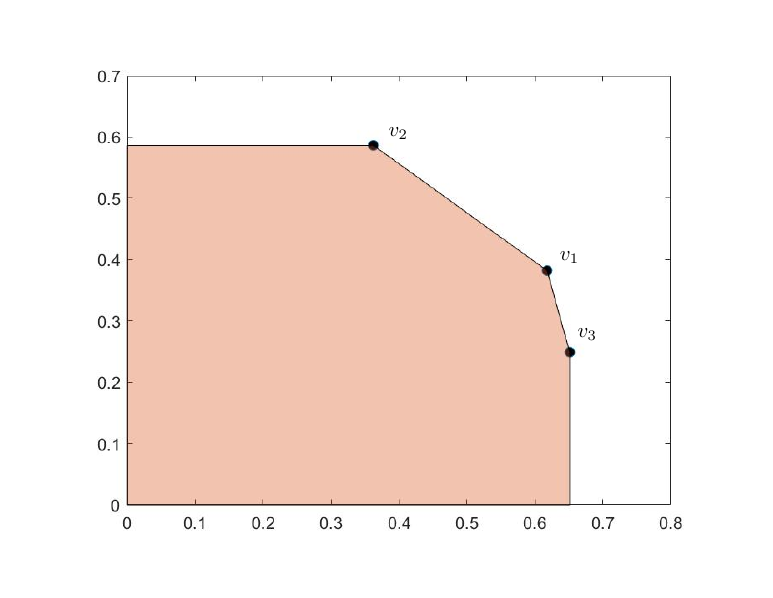}
\includegraphics[width=.4\textwidth]{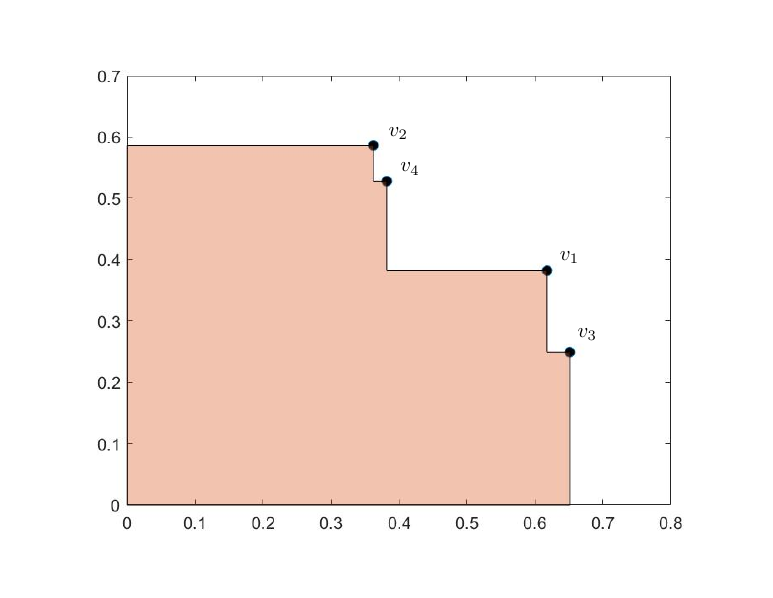}
\end{center}
\caption{Extremal polytope norm (left figure) and prenorm (right figure) generated respectively by the polytopal algorithm \cite{guglielmi2013exact} and our proposed Algorithm} \label{Fig-linear_vs_nonlinear-Example}
\end{figure} 
\end{example}

\section{Acknowledgements}
PD,NG,FT are members of the Gruppo Nazionale Calcolo Scientifico - Istituto Nazionale di Alta Matematica (GNCS-INdAM). PD was partially supported by the MUR-PRO3 grant "STANDS" and the INdAM-GNCS project “NLA4ML—Numerical Linear Algebra Techniques for Machine Learning.” 
NG acknowledges that his research was supported by funds from the Italian  MUR (Ministero dell'Universit\`a e della Ricerca) within the 
PRIN 2022 Project ``Advanced numerical methods for time dependent parametric partial differential equations with applications'' and the PRIN-PNRR Project ``FIN4GEO''.
FT is partially funded by the PRIN-MUR project MOLE code 2022ZK5ME7, and by PRIN-PNRR project FIN4GEO within the European Union’s Next Generation EU framework, Mission 4, Component 2, CUP P2022BNB97.

\bibliographystyle{abbrv}
\bibliography{references.bib}

%%%%
%%%%
%%%%
%%%%
%%%%
%%%%
% APPENDIX
%%%%
%%%%
%%%%
%%%%
%%%%
%%%%

\appendix

\section{Technical results from Nonlinear Perron-Frobenius Theory}\label{Appendix_nonlinear_Perron Frobenius}
\begin{lemma}[Lemma 2.5.1 and Lemma 2.5.5 \cite{lemmens2012nonlinear}]\label{Lemma_thompson_equivalent_to_norm}
    Let $\cone$ be a solid closed cone in $(V,\|\cdot\|)$ and let $\delta>0$ be the normality constant of $\cone$. Then:
    \begin{itemize}
        \item For each $x,y\in\cone$ with $\max\{\|x\|,\|y\|\}\leq b$, we have that 
    $$\|x-y\|\leq b(1+2\delta)(e^{\distT(x,y)}-1).$$
        \item If $x\in\interior(\cone)$ and $r>0$ are such that the closed ball $B_r(x)=\{z\in V:\|x-z\|\leq r\}$ is contained in $\interior(\cone)$, then 
    $$\distT(x,y)\leq \log \max\Big\{ \frac{r+\|x-y\|}{r},\frac{r}{r-\|x-y\|}  \Big\}\qquad \forall y\in B_r(x).$$
    \end{itemize}
\end{lemma}

\begin{lemma}[Lemma 2.1.7 \cite{lemmens2012nonlinear}]
    Let $\cone$ be a solid closed cone and let $f$ be an order-preserving map. Then $f$ is subhomogeneous if and only if $f$ is non-expansive with respect to the Thompson's metric on each part of $\cone$. In addition if $f:\interior(\cone)\rightarrow \interior(\cone)$ is strictly subhomogeneous, then $f$ is contractive with respect to $\distT$.
\end{lemma}

\begin{proposition}[Prop 2.5.4 \cite{lemmens2012nonlinear}]
Let $K$ be a solid closed cone and let $P$ be a part of $K$ and $\psi\in \interior(\cone^*)$. If $P\neq \{0\}$ and $\Xi_{\psi}=\{x\in\cone\:|\;\psi(x)=1\}$, then $\big(\Xi_{\psi}\cap P, \dist\big)$ is a complete metric space.
\end{proposition}

\begin{theorem}[Theorem 5.4.1 and Corollary 5.4.2 \cite{lemmens2012nonlinear}]\label{Thm_existence_of_eigenvector}
Let $\cone$ be a closed cone with nonempty interior and let $f:\cone\rightarrow \cone$ be a continuous, homogeneous, and order-preserving map. Suppose $\psi\in\interior(\cone^*)$ and $u\in\interior(\cone)$ and define
$$
f_{\epsilon}(y)=f(y)+\epsilon\psi(y) u\qquad \forall y\in\cone\,.
$$
Then:
\begin{itemize}
\item $f_{\epsilon}$ is continuous, homogeneous and order-preseving.
\item $\exists x_{\epsilon}\in\interior(\cone)$, with $\psi(x_{\epsilon})=1$, such that $f_{\epsilon}(x_{\epsilon})=\cradius(f_{\epsilon})x_{\epsilon}$
\item If $\eta<\epsilon$, then $\cradius(f_{\eta})<\cradius(f_{\epsilon})$ and hence $\exists \lim_{\epsilon\rightarrow 0}\cradius(f_{\epsilon})=\cradius(f)$
\item There exists a convergent sequence $(x_{\epsilon_k})_k$ with $x_{\epsilon_k}\rightarrow \ceigenvector$ and $\epsilon_k\rightarrow 0$ as $k\rightarrow \infty$, such that and $f(\ceigenvector)=\cradius(f)\ceigenvector$
\end{itemize}
\end{theorem}

\begin{lemma}[Proposition 5.3.6 \cite{lemmens2012nonlinear}]\label{Lemma_lemmens_spectral_radius_2} Let $\cone$ be a closed cone and $f$ a continuous, homogeneous and order-preserving map.
If there exist $x  \in \cone\setminus\{0\}$, $a\in \R^+$ such that 
\begin{equation}
    ax\cleq f(x),
\end{equation}
then $a\leq \cradius(f)$\,.
Moreover for any $x\in\interior(\cone)$ it is possible to prove that 
    \begin{equation}
        \cradius(f)=\lim_{k\rightarrow \infty}\|f^k(x)\|_{\cone}^{\frac{1}{k}}
    \end{equation}
\end{lemma}

\begin{lemma}[Lemma 5.2.1 \cite{lemmens2012nonlinear}]\label{Lemma_comparison_eigenvalues}
Assume that $f$ is a continuous, homogeneous, and order-preserving map on a closed cone $\cone$ and $x,y\in\cone$ are such that $M(x/y)<\infty$. If $\lambda x\cleq f(x)$ and $f(y)\cleq \eta y$, then $\lambda\leq \eta$. 
\end{lemma}
In particular, whenever there exist two eigenvectors in the same face $x\sim y$, they correspond to the same eigenvalue.

\begin{proposition}[Proposition 5.3.2 \cite{lemmens2012nonlinear}]\label{Prop_properties_cone_spectral_radius}
If $K$ is a closed cone and $f$ is a continuous, homogeneous, and order-preserving map, then: 
\begin{itemize}
    \item $\forall k\geq 1$, $\cradius(f^k)=\cradius^k(f)$
    \item  If $f^k(x)=\lambda^k x$ for some $x\in\cone$, $k\geq 1$, then $\lambda\leq \cradius(f)$.
    \item If $f_1$ and $f_2$ are continuous, homogeneous and order-preserving maps, $\cradius(f_1\circ f_2)=\cradius(f_2\circ f_1)$.
\end{itemize}
\end{proposition}

\begin{lemma}[Lemma 5.5.6 \cite{lemmens2012nonlinear}]\label{Continuity_of_spectral_radius_from_above}
    If $f$ is a continuous, homogeneous and order-preserving map on a solid closed cone $\cone$ and $\{f_k\}_{k}$ is a sequence of continuous homogeneous order-preserving maps on $\cone$ converging to $f$ from above, i.e. $f(x)\cleq f_k(x)$ for all $x\in \cone$ and $k\geq 1$. Then $\lim_{k}\cradius(f_k)=\cradius(f)$.
\end{lemma}

\begin{theorem}[Theorems 5.1.2 and 5.1.5]\label{Thm_extension_of_order_preserving_functions}
Let $\cone$ and $\cone'$ be closed cones, and suppose that $\cone$ is polyhedral. If $f:\interior(\cone)\rightarrow \cone'$ is a continuous, homogeneous, and order-preserving map, then there exists $F:\cone\rightarrow\cone'$ a continuous, order-preserving, and homogeneous extension of $f$ to the boundary of the cone.
\end{theorem}

\section{Spectral radii of subhomogeneous functions}\label{subsec:joint_spectral_radii_of_subhomogeneous_maps}

Homogeneous (or subhomogeneous) and order-preserving maps are often studied because of their well-understood spectral properties, i.e., existence of a spectral radius and convergence of the power method toward some eigenvector. In this section, extending other results from \cite{nuzman2007contraction, cavalcante2019connections}, we prove that any subhomogeneous and order-preserving map $f$ admits a continuous curve of eigenpairs such that on any slice of the cone there exists a unique eigenvector of the curve, and the corresponding eigenvalue is the maximal one on that slice. Furthermore, we study sufficient conditions guaranteeing that the limits to $0$ and $\infty$ of this curve correspond to the spectral radius and the corresponding eigenvector of the homogeneous limit functions $f_0$ and $f_{\infty}$.

Besides the scope of this manuscript, these results could find applications in the study of dynamical systems and neural networks where the existence of some fixed point or the convergence of the power method toward some eigenvector is desirable or can shed light on the learned task, see e.g. deep equilibrium models \cite{bai2019deep, piotrowski2024fixed} and denoising networks \cite{bungert2021nonlinear}. 

Let $f$ be a continuous, subhomogeneous, and order-preserving map, then we say that $(x,\lambda)$ is an eigenpair of $f$ iff $f(x)=\lambda x$. Now let $\psi\in \interior(\cone^*)$ and, for any $c\in \R_+$, define the \textbf{cone slice}: 
\begin{equation}\label{Def_cone_slice}
    \Omega_c=\{x\in \cone\;|\; \psi(x)=c\}.
\end{equation} 
Then we can prove that any slice contains some eigenvector of $f$.

\begin{proposition}
    Let $f$ be a continuous, subhomogeneous, and order-preserving map on a closed cone $\cone$ and let $\Omega_c$ be a slice of $\cone$ with $c>0$. Then, there exist some $(x,\lambda)$ with $x\in \Omega_c$ such that $f(x)=\lambda x$.
\end{proposition}
\begin{proof}
    If $f(x)=0$ for some $x\in \Omega_c$ then taking $\lambda=0$ we have the thesis. Otherwise, we can consider the map
    \begin{equation}
        g_c(x)=c\frac{f(x)}{\psi(f(x))},
    \end{equation}
    which maps $\Omega_c$ into itself. Then since $\Omega_c$ is a compact convex set $g_c$ admits a fixed point by the Brouwer fixed point theorem. 
\end{proof}
In addition, the cone spectrum relative to a single slice can be trivially proved to be a compact set 
\begin{proposition}
    Let $\Omega_c$ be a slice of a closed cone $\cone$ and $f$ a continuous, subhomogeneous, and order-preserving map. Denote by $S_c=\{(x,\lambda) \text{ s.t. } x\in \Omega_c, \; f(x)=\lambda x \}$, then $S_c$ is a compact set.
\end{proposition}
In particular it is well defined the spectral radius of $f$ on the slice $\Omega_c$ by:
\begin{equation}
    \cradius^c(f):=\max\{\lambda\;|\; \exists\: x\in\Omega_c \text{ with } f(x)=\lambda x\}.
\end{equation}
Next, we prove a variational characterization of $\cradius^c(f)$ and study how the eigenpairs vary when varying the slice.
To this end, we introduce the following generalized spectral radii:
\begin{equation}
\cgradius^c(f):=\inf_{y\in \interior(\Omega_c)}\sup_{x\in \Omega_c}\frac{M(f(x)/y)}{M(x/y)}
\qquad
\cdradius^c(f):=\sup_{y\in\Omega_c}\inf_{x\in \interior(\Omega_c)}\frac{m(f(x)/y)}{m(x/y)}.
\end{equation}
\begin{theorem}\label{Prop_characterization_of_subhomog_spectral_radius}
Let $f$ be a continuous, subhomogeneous and order-preserving map on a closed cone $\cone$, $\psi\in \interior(\cone^*)$ and $\Omega_c:=\{x\in \cone \text{ s.t. } \psi(x)=c\}$. Then:
\begin{enumerate}
    \item For any $c>0$, $\cradius^c(f)=\cgradius^c(f)=\cdradius^c(f)$.
    \item There exists a continuous curve $c\in(0,\infty)\rightarrow (x_c,\cradius^c(f))$, where $x_c\in \Omega_c$ is some eigenvector associated to $\cradius^c(f)$, such that: 
    \begin{equation}
        x_{c_1}\cleq x_{c_2}\,,  \qquad\qquad \cradius^{c_2}(f_{\epsilon}) \leq \cradius^{c_1}(f_{\epsilon})\leq \frac{c_2}{c_1}\cradius^{c_2}(f_{\epsilon}),
    \end{equation}
    for any $0<c_1< c_2$.
\end{enumerate}
\end{theorem}
\begin{proof}
Start considering $x_c\in \Omega_c$ such that $f(x_c)=\cradius^c(f) x_c$. Then for any $y\in \interior(\cone)\cap \Omega_c$, we have $M(f(x_c)/y)=\cradius^c(f)M(x/y)$ meaning
\begin{equation}\label{eq1_subhom_SR}
    \sup_{x\in \Omega_c}\frac{M(f(x)/y)}{M(x/y)}\geq \cradius^c(f) \qquad \text{i.e.} \qquad \cgradius^c(f)\geq \cradius^c(f).
\end{equation}
On the other hand, \textbf{assuming that $x_c\in \interior(\cone)$}, then $M(x/x_c)\geq \psi(x)/\psi(x_c)=1$ for all $x\in \Omega_c$, which yields
\begin{equation}
    f(x) \cleq f(M(x/x_c)x_c) \cleq M(x/x_c)f(x_c)= M(x/x_c)\cradius^c(f)x_c.
\end{equation}
The last inequality shows that $M(f(x)/x_c)\leq \cradius^c(f) M(x/x_c)$, thus we have:
\begin{equation}
    \cgradius^c(f)\leq \sup_{x\in \Omega_c}\frac{M(f(x)/x_c)}{M(x/x_c))}\leq \cradius^c(f),
\end{equation}
i.e., the thesis.
Analogously taking $y=x_c$, for any $x\in\interior(\cone)\cap \Omega_c$ we have that $m(x/x_c)\leq 1$, yielding
\begin{equation}
    f(x)\cgeq f(m(x/x_c)x_c) \cgeq m(x/x_c)f(x_c) = m(x/x_c) \cradius^c(f) x_c,
\end{equation}
i.e. $m(f(x)/x_c)\geq m(x/x_c)\cradius^c $. So we have 
\begin{equation}\label{eq2_subhom_SR}
    \cdradius^c(f)\geq \inf_{x\in\interior(\Omega_c)}\frac{m(f(x)/x_c)}{m(x/x_c)}\geq \cradius^c(f).
\end{equation}
On the other hand, \textbf{assuming $x_c\in \interior(\cone)$}, then for any $y\in \cone$, we have $m(f(x_c)/y)=\cradius^c(f) m(x_c/y)$ which yields $\cdradius^c(f)\leq \cradius^c(f)$.
We have proved that if $x_c$ belongs to the interior of the cone, the first part of the thesis holds. Next, we address the case $x_c\in \partial\cone$ by perturbation.
Let $u\in \interior(\cone)$ with $\psi(u)=1$, $\psi\in\interior(\cone^*)$, for some $\beta>0$ we have $f(u)\cleq \beta u$. Then for any $\epsilon>0$ consider $f_{\epsilon}$ the map defined by
\begin{equation}
 f_{\epsilon}(x)=f(x)+\epsilon u.   
\end{equation}
Observe that, for any $x\in \cone$, $f_{\epsilon}(x)$ is in the interior of the cone and $f_\epsilon(\lambda x)\cll \lambda f_\epsilon(x)$ for any $\lambda>1$ and $x\in \cone$.
Moreover, for any $c\geq 0$ there exists some $\alpha>0$ such that $u\geq \alpha x$ for all $x\in \Omega_c$ . In particular, since $f(u/\alpha)\leq f(u)$ if $\alpha>1$ and $f(u/\alpha)\leq f(u)/\alpha$ if $\alpha<1$, we have: 
\begin{equation}
     \epsilon u \cleq f_{\epsilon}(x) \cleq \Big(\max\Big\{1,\frac{1}{\alpha}\Big\}\beta + \epsilon\Big) u \qquad \forall x\in \Omega_c.
\end{equation}
The last inequality yields the following bound
\begin{equation}
c \epsilon C_{\alpha,\beta}^{-1} u \cleq
\frac{c f_{\epsilon}(x)}{\psi(f_{\epsilon}(x))}\cleq \frac{c}{\epsilon} C_{\alpha,\beta} u ,
\end{equation}
where we have defined $C_{\alpha,\beta}=\Big(\max\Big\{1,\frac{1}{\alpha}\Big\}\beta + \epsilon\Big)$.
Thus $cf_{\epsilon}(x)/\psi\big(f_{\epsilon}(x)\big)$ maps $\Omega_c$ in a convex subset of $\Omega_c$ that is contained in the interior of $\cone$. In particular, there exists some eigenvector $x_c^{\epsilon}$ of $f_{\epsilon}$ in $\Omega_c$ with corresponding eigenvalue $\lambda^c_{\epsilon}=\psi\big(f_{\epsilon}(x_c^{\epsilon})/c\big)$. 
On the other hand, we can trivially observe that 
if $x,y\in \Omega_c\cap \interior(\cone)$, then $m(x/y) y\cleq x\cleq M(x/y) y$ where necessarily $m(x/y)<1$ and $M(x/y)>1$ (indeed $\psi(x-m(x/y) y)>0$ and $\psi(M(x/y) y-x)>0$). 
Thus, 
\begin{equation}
f_{\epsilon}(x)=f(x)+\epsilon u \cleq M(x/y)\Big( f(y)+\frac{\epsilon u}{M(x/y)} \Big)\cll M(x/y)\Big( f(y)+\epsilon  u\Big)=M(x/y)f_{\epsilon}(y)
\end{equation}
and analogously $m(x/y)f_{\epsilon}(y)\cll f_{\epsilon}(x)$. In particular $\dist(f_\epsilon(x),f_{\epsilon}(y))< d_{H}(x,y)$ and since $\dist$ is scale invariant we observe that $x_{c}^{\epsilon}$ is unique and $\lambda^c_{\epsilon}=\cradius^c(f_{\epsilon})$.
Next we \textbf{claim} that for any $c>0$ and $\epsilon_1>\epsilon_2>0$, 
\begin{equation}
\cradius^c(f_{\epsilon_1})=\cgradius^c(f_{\epsilon_1})=\cdradius^c(f_{\epsilon_1})\geq \cdradius^c(f_{\epsilon_2})=\cgradius^c(f_{\epsilon_2})=\cradius^c(f_{\epsilon_2}).
\end{equation}
Indeed, for any $\epsilon>0$, $x_{c}^\epsilon$ is in the interior of the cone. Moreover, given $\epsilon_1>\epsilon_2$, $f_{\epsilon_1}(x)\cgg f_{\epsilon_2}(x)$ for all $x\in\cone$, in particular $M(f_{\epsilon_1}(x)/y)\geq M(f_{\epsilon_2}(x)/y)$ for any $y\in \interior(\cone)$ and $m(f_{\epsilon_1}(x)/y)\geq m(f_{\epsilon_2}(x)/y)$ for any $y\in \cone$ and $x\in\interior(\cone)$. 
In particular we have that for any $c>0$, $\cradius^c(f_\epsilon)$ decreases as $\epsilon$ goes to zero. Thus we can take the limit, say $\lambda^c$, for $\epsilon$ that goes to zero of $\cradius^{c}(f_{\epsilon})$ and, eventually considering a subsequence, also of the corresponding eigenvectors $x_{c}^{\epsilon}$, say $x_c$. By the continuity of the eigenvalue equation and the uniform convergence of $f_\epsilon$, we easily see that the pair $(x_c, \lambda^c)$ satisfies:
\begin{equation}
   f(x_{c_i})=\lambda^{c_i}x_{c_i} \qquad i=1,2 
\end{equation}
In particular, for any $c>0$ the limit eigenvalue obtained in this way satisfies $\lambda^c\leq \cradius^c(f)\leq \min\{\cdradius^c(f),\cgradius^c(f)\}$, where we have used \eqref{eq1_subhom_SR} and \eqref{eq2_subhom_SR}. On the other hand, as in the claim above, it is easy to observe that, for any $\epsilon>0$, $\max\{\cgradius^c(f),\cdradius^c(f)\}\leq \cradius^c(f_{\epsilon})$.
Thus 
\begin{equation}
    \max\{\cdradius^c(f),\cgradius^c(f)\} \leq \lambda^c = \lim_{\epsilon\rightarrow 0}\cradius^c(f_{\epsilon})\leq \cradius^c(f)\leq \min\{\cgradius^c(f),\cdradius^c(f)\}, 
\end{equation}
proving the equality and concluding the proof of point 1.
To prove point 2, let $c_2>c_1$ and, for a fixed $\epsilon$, consider $\big(x_{c_1}^\epsilon,\cradius^{c_1}(f_{\epsilon})\big)$ and $\big(x_{c_2}^\epsilon,\cradius^{c_2}(f_{\epsilon})\big)$.
Assume by contradiction that $\cradius^{c_2}(f_{\epsilon})\geq \cradius^{c_1}(f_{\epsilon})$ and consider $m_{12}=m(x_{c_1}^\epsilon/x_{c_2}^\epsilon)$.
Since $m_{12}\leq \psi(x_{c_1}^\epsilon)/\psi(x_{c_2}^\epsilon)=c_1/c_2<1$ we have the following set of inequalities:
\begin{equation}
m_{12} \cradius^{c_2}(f_{\epsilon}) x_{c_2}^\epsilon \cll f_{\epsilon}(m_{12} x_{c_2}^\epsilon) \cleq f_{\epsilon}(x_{c_1}^\epsilon) = \cradius^{c_1}(f_{\epsilon})x_{c_1}^\epsilon \cleq \cradius^{c_2}(f_{\epsilon})x_{c_1}^\epsilon  
\end{equation}
The last proves that $m_{12}$ is strictly smaller than $m(x_{c_2}^{\epsilon}/x_{c_1}^{\epsilon})$ yielding a contradiction. In particular we have proved that $\cradius^{c_2}(f_{\epsilon})< \cradius^{c_1}(f_{\epsilon})$ for $c_1\leq c_2$
Analogously we can prove that $x_{c_2}^\epsilon\cgeq x_{c_1}^\epsilon$. Note that the last inequality is equivalent to have $M_{12}:=M(x_{c_1}^\epsilon/x_{c_2}^\epsilon)\leq 1$. So assume the opposite (i.e. $=M(x_{c_1}^\epsilon/x_{c_2}^\epsilon)>1$). Then, since we have proved that $\cradius^{c_1}(f_{\epsilon})>\cradius^{c_2}(f_{\epsilon})$, and $x^{\epsilon}_{c_1}\in \interior(\cone)
)$, we have the following set of inequalities:
\begin{equation}
    \cradius^{c_2}x^{\epsilon}_{c_1}\cll \cradius^{c_1}x^{\epsilon}_{c_1}=f_{\epsilon}(x_{c_1}^{\epsilon})\cleq f(M_{12}x_{c_2}^{\epsilon})\cll M_{12}\cradius^{c_2}x_{c_2}^\epsilon
\end{equation}
the last in particular yields a contradiction as proves that $M_{12}>M(x_{c_1}^\epsilon/x_{c_2}^\epsilon)$.
As a direct consequence of the last inequality we can derive the following simple inequality
\begin{equation}
  \cradius^{c_1}(f_\epsilon)x_{c_1}^\epsilon= f_\epsilon(x_{c_1}^\epsilon)\cleq f_\epsilon(x_{c_2}^\epsilon)= \cradius^{c_2}(f_\epsilon) x_{c_2}^\epsilon.
\end{equation}
In particular, applying $\psi$ on both the left and right hand side of the last inequality yields the following bounds for $\cradius^{c_1}(f_\epsilon)$ in terms of $\cradius^{c_2}(f_\epsilon)$
\begin{equation}
   \cradius^{c_2}(f_\epsilon) < \cradius^{c_1}(f_\epsilon)\leq \frac{c_2}{c_1}\cradius^{c_2}(f_\epsilon)
\end{equation}
where we have recalled the lower bound from above.
As a direct consequence, we have that the limit of the function $c\rightarrow \cradius^c(f_\epsilon)$, letting $\epsilon$ go to zero, exists and is continuous. Moreover the map $ c\rightarrow x_c^\epsilon $ is one Lipschitz, as a map from $\R^+\setminus\{0\}$ to $(\cone,\psi)$, where $\psi$ denotes some norm equal to $\psi$ on the cone, indeed
\begin{equation}
\begin{aligned}
    \mathrm{dist}_{\psi}(x_{c_1},x_{c_2})= \psi(x_{c_2}^\epsilon-x_{c_1}^\epsilon) = |c_2-c_1|\,.
\end{aligned}
\end{equation}

In particular, the Ascoli-Arzelà theorem guarantees that (possibly up to subsequences) the maps $\{c\rightarrow \big(x_c^\epsilon,\cradius^c(f_\epsilon)\big) \}_{\epsilon>0}$ converge to the continuous limit map that associates $c$ to $\big(x_c,\cradius^c(f)\big)$ concluding also the proof of point 2.
\end{proof}

\begin{remark}
Note that if in the definition of $\cgradius^c(f)$ we restrict the supremum to the interior of the cone, the value does not change, i.e. 
\begin{equation}
\cgradius^c(f)=\inf_{y\in\interior(\Omega_c)}\sup_{x\in \interior(\Omega_c)}\frac{M\big(f(x)/y\big)}{M\big(x/y\big)}
\end{equation}
Indeed assume that $y\in \interior(\Omega_c)$ and $x\in\Omega_c$ are such that $x\cleq (\cgradius^c+\epsilon)M(x/y)y$. Then consider $x_t=tx+(1-t)y\in \interior(\Omega_c)$ for all $t\in (0,1)$, because of the continuity of $f$ for any $\delta>0$ there exists some $t_0$ such that for $t>t_0$ $x_t\cleq (\cgradius^c+\epsilon+\delta)M(x/y)y$. Moreover for any $t$ $x\cleq \big(M(x_t/y)-(1-t)\big)/t y$ which means $M(x/y)\leq (M(x_t/y)-(1-t)\big)/t$. So if we let $t$ going to $1$ we verify that 
\begin{equation}
    \cgradius^c(f)+\epsilon\leq \sup_{x\in \interior(\Omega_c)}\frac{M(f(x)/y)}{M(x/y)}
\end{equation}
while the reverse inequality is trivial. In particular, by means of $\cgradius(f)$, the spectral radius definition can be extended to maps defined only on the interior of the cone. For such maps, when also homogeneous, the spectral radius is usually defined by means of the Gelfand formula.
\end{remark}

\begin{corollary}\label{Corollary_Collatz-Wielandt_for_cshop_functions}
    Let $f$ be a continuous, subhomogeneous and order-preserving map on a closed cone $\cone$, $\psi\in\interior(\cone^*)$ and assume $\Omega_c:=\{x\in\cone\;|\;\psi(x)=c\}$. Then 
    \begin{enumerate}
        \item If there exists $x\in\Omega_c$ with $f(x)\cgeq \alpha x$, then $\cradius^c(f)\geq \alpha$. In particular: 
        \begin{equation}\label{eq:1_corollary_subhom_spectral_radius}
            \cradius^c(f)=\sup_{x\in\Omega_c}m(f(x)/x).
        \end{equation}
        \item If there exists $x\in\interior(\cone)\cap\Omega_c$ such that $f(x)\cleq \beta x$, then $\cradius^c(f)\leq \beta$. In particular: 
        \begin{equation}\label{eq:2_corollary_subhom_spectral_radius}         \cradius^c(f)=\inf_{x\in\interior(\cone)\cap\Omega_c}M(f(x)/x).
        \end{equation}
    \end{enumerate}
\end{corollary}
\begin{proof}
We start by proving the first thesis. Assume $x$ as in the hypothesis and consider $y\in \interior(\cone)$, then $M(f(x)/y)\geq M(\alpha x/y)=\alpha M(x/y)$.
    In particular, 
    \begin{equation}
        \sup_{x\in \Omega_c}\frac{M(f(x)/y)}{M(x/y)}\geq \alpha.
    \end{equation} 
Taking the infimum in $y$, yields the thesis $\cradius^c(f)\geq \alpha$ because of the equality $\cgradius^c(f)=\cradius^c(f)$, (see \Cref{Prop_characterization_of_subhomog_spectral_radius}). Finally, the equality in \eqref{eq:1_corollary_subhom_spectral_radius} follows by the existence of an eigenvector associated to $\cradius^c(f)$.
Analogously assume $x$ as in the point $2$, then for any $y\in \cone$ $m(f(x)/y)\leq \beta m(x/y)$, i.e. $\cdradius(f)\leq \beta$. Thus we have 
\begin{equation}
    \cradius^c(f)\leq \inf_{x\in\interior(\cone
    \cap\Omega_c}M(f(x)/x).
\end{equation}
If moreover $f$ has an eigenvector $x_c$ relative to $\cradius^c(f)$ in the interior of the cone we can easily conclude taking $x=x_c$. Otherwise taking $f_{\epsilon}$ as in the proof of \Cref{Prop_characterization_of_subhomog_spectral_radius} we have that 
\begin{equation}
    \cradius^c(f)=\lim_{\epsilon \rightarrow 0}\cradius^c(f_\epsilon)=\lim_{\epsilon\rightarrow 0}\inf_{x\in\interior(\cone)\cap\Omega_c}M(f_\epsilon(x)/x)\geq \inf_{x\in\interior(\cone)\cap\Omega_c}M(f(x)/x).
\end{equation}
\end{proof}
Next, we discuss the convergence of the spectral radii $\cradius^c(f)$ for $c$ that goes to zero and infinity to the spectral radii of the homogeneous maps $f_0$ and $f_{\infty}$.
\begin{proposition}\label{prop_limits_of_subhomogeneous_eigenpairs}
    Assume $f_0$ and $f_{\infty}$ to be continuous on $\cone$, then 
    \begin{equation}
        \cradius(f_0)\geq\lim_{c\rightarrow 0}\cradius^c(f) 
        \qquad  \qquad 
        \cradius(f_{\infty})=\lim_{c\rightarrow \infty}\cradius^c(f)
    \end{equation}
    If in addition there exists $x_{f_0}$ an eigenvector of $f_0$ relative to $\cradius(f_0)$ at which condition \textbf{G} is satisfied then $\cradius(f_0)=\lim_{c\rightarrow 0}\cradius^c(f)$.
\end{proposition}
\begin{proof}
Recall from \Cref{Prop:uniform_convergence_limit_maps}, that $f_c(x)=f(cx)/c$ converges uniformly on compact sets to the maps $f_0$ and $f_{\infty}$ for $c$ going to zero and infinity.
Now let $\psi \in \interior(\cone^*)$ and for any $c>0$ let $\big(x_c,\cradius^c(f)\big)$ be the spectral radius and a corresponding eigenfunction of $f$ on the slice $\Omega_c=\{x\::\;\psi(x)=c\}$. Remember that they satisfy:
    \begin{equation}
        f\Big(c \,\frac{x_c}{c}\Big)=f(x_c)=\cradius^c(f)\, x_c=\cradius^c(f) \,c\, \frac{x_c}{c}
    \end{equation}
    In particular, defined $y_c:=x_c/c$, we have that $y_c\in \Omega_1$ for any $c>0$ and $f_c(y_c)=\cradius^c(f) y_c$, as a consequence of this 1-1 correspondence of the eigenvalues of $f$ on $\Omega_c$ and $f_c$ on $\Omega_1$, we conclude that 
    \begin{equation}
    \cradius^c(f)=\cradius^1(f_c).
    \end{equation}
    Now, since $\Omega_1$ is compact, and for $c_1>c_2$ $\cradius^{c_1}(f)\leq \cradius^{c_2}(f)$ (see \Cref{Prop_characterization_of_subhomog_spectral_radius}), we can take, eventually up to a subsequence, the limits of $\big(y_c,\cradius^c(f)\big)$ for $c$ going to $0$ and $\infty$. Such limits yield some $\big(y_0,\lambda^0(f)\big)$ and $\big(y_{\infty},\lambda^{\infty}(f)\big)$ with $y_{\infty}, y_{0}\in \Omega_1$ that, thanks to the uniform convergence, satisfy:
    \begin{equation}
    f_0(y_0)=\lambda^0 y_0\qquad \qquad f_{\infty}(y_{\infty})=\lambda^{\infty} y_{\infty}.    
    \end{equation}
    In particular, we have $\lambda^0\cleq \cradius(f_0)$ and $\lambda^{\infty}\cleq \cradius(f_{\infty})$.
    However, for $c>0$ we can write 
    \begin{equation}
        \cradius^c(f)=\sup_{x\in\Omega_c} m(f(x)/x)= \sup_{x\in\Omega_1} m\big(f_c(x)/x\big)
    \end{equation}
    where we have observed that if $f(cx)\geq m cx$ for some $x\in \Omega_1$, then $f(cx)/c\geq m x$. Since $f_c(x)\cgeq f_{\infty}(x)$, we can conclude that 
    \begin{equation}
        \cradius^c(f)= \sup_{x\in\Omega_1} m\big(f_c(x)/x\big) \geq \sup_{x\in\Omega_1} m\big(f_{\infty}(x)/x\big)=\cradius(f_{\infty}) \qquad \forall c\geq 0.
    \end{equation}
 In particular $\lambda^{\infty}=\lim_{c\rightarrow \infty}\cradius^{c}(f)\geq \cradius(f_{\infty})$, i.e. $\lim_{c\rightarrow \infty}\cradius^{c}(f)=\cradius(f_{\infty})$. 
 On the other hand, assume that there exists $x_{f_0}$ an eigenvector relative to $\cradius(f_0)$ at which condition \textbf{G} is satisfied. Then, for any $\alpha\in(0,1)$, there exists $C$ such that for any $c<C$, $f_c(x_{f_0})\cgeq \alpha\cradius(f_0)x_{f_0}$. In particular, as a consequence of \Cref{Corollary_Collatz-Wielandt_for_cshop_functions} $\lim_{c\rightarrow 0}\cradius^{c}(f)\geq \cradius(f_0)$ concluding the proof.    
\end{proof}
We recall that condition \textbf{G} is automatically satisfied at any point in the case of a polyhedral cone as well as at any internal point of a generic cone.
We also point out that the continuity of the spectral radius on general cones is an open problem also for homogeneous maps (see Problem 5.5.1 of \cite{lemmens2012nonlinear}).

\end{document}

%% file: packages.tex
\usepackage[english]{babel}

\usepackage{authblk}
\usepackage{color}
\usepackage{amsmath, amssymb, mathdots, accents, amsthm}
\usepackage[a4paper,left=3.5cm,right=4cm]{geometry}
\usepackage{ifthen}
\usepackage{caption,subcaption}
\usepackage{mathtools}

\usepackage{hyperref}
\usepackage{nameref,cleveref}

\usepackage{algorithm}
\usepackage{algorithmic}
\usepackage{comment}

\newtheorem{theorem}{Theorem}[section]
\newtheorem{lemma}[theorem]{Lemma}
\newtheorem{definition}[theorem]{Definition}
\newtheorem{remark}[theorem]{Remark}
\newtheorem{corollary}[theorem]{Corollary}
\newtheorem{example}[theorem]{Example}
\newtheorem{proposition}[theorem]{Proposition}

 \usepackage{tikz, tikz-3dplot}
 \usepackage{pgfplotstable}
\usepackage{pgfplots}
\usepackage{pgffor,pgfmath}
\usepackage{booktabs}
\usetikzlibrary{arrows}
\pgfplotsset{compat=1.14}

\usepackage{titlesec}

\usepackage[numbers,sort]{natbib}
\usepackage{blindtext}
\usepackage{cancel}
\usepackage{autonum} %% NOTE IT DOES NOT WORK WITH equation* and align* environments

%% file: commands.tex
\newcommand{\R}{\mathbb{R}}

\newcommand{\N}{\mathbb{N}}

\newcommand{\family}{\mathcal{F}}

\newcommand{\radius}{\rho}

\newcommand{\cradius}{\rho_{\cone}}

\newcommand{\cgradius}{\widehat{\rho}_{\cone}}
\newcommand{\cdradius}{\Check{\rho}_{\cone}}

\newcommand{\comp}{\circ}

\newcommand{\MPN}{\mathrm{MPN}}

\newcommand{\DMPN}{\mathrm{DMPN}}

\newcommand{\cone}{\mathcal{K}}

\newcommand{\interior}{\mathrm{Int}}

\newcommand{\dist}{d_H}

\newcommand{\distT}{d_T}

\newcommand{\cgeq}{\geq_{\cone}}

\newcommand{\cleq}{\leq_{\cone}}

\newcommand{\cl}{<_{\cone}}

\newcommand{\cll}{\ll_{\cone}}

\newcommand{\cgg}{\gg_{\cone}}

\newcommand{\ceigenvector}[1][]{
  \ifthenelse{\equal{#1}{}}
  {{x_f}}
  {{x}_{#1}}}

\DeclareMathOperator*{\argmin}{arg\:min}